\documentclass[matprg]{svjour3}
\usepackage{subcaption}
\usepackage{amsmath,amsxtra,amsfonts,amscd,amssymb,bm}
\usepackage{algorithm}
\usepackage{algpseudocode}
\usepackage[top=2.5cm,bottom=2.5cm,right=2.5cm,left=2.5cm]{geometry}
\usepackage[mathscr]{eucal}
\usepackage{color}
\usepackage{cite}
\numberwithin{equation}{section}
\usepackage{hyperref}
\usepackage{multirow}

\def\sl{\textstyle{\sum}_{\ell=1}^k}

\def\O{\mathcal{O}}

\def\cG{{\cal G}}
\def\cX{{\cal X}}

\newcommand{\revision}[2]{{\color{black}#2}}
\newcommand{\rev}[2]{{\color{black}#2}}

\newcommand{\tsum}{\textstyle\sum}
\newcommand{\bbe}{\mathbb{E}}

\def\cA{{\cal A}}

\def\eqref#1{(\ref{#1})}
\usepackage{epsfig}
\usepackage{calc}
\usepackage{amstext}
\usepackage{amsmath}
\usepackage{multicol}
\usepackage{amsfonts}
\usepackage{amssymb}
\usepackage{mathrsfs}
\usepackage{bm}
\usepackage{xcolor}
\usepackage{comment}
\usepackage[scientific-notation=true]{siunitx}

\newtheorem{assumption}{Assumption}
\newcommand{\bbr}{\Bbb{R}}
\newcommand{\beq}{\begin{equation}}
\newcommand{\eeq}{\end{equation}}
\newcommand{\beqa}{\begin{eqnarray}}
\newcommand{\eeqa}{\end{eqnarray}}
\newcommand{\beqas}{\begin{eqnarray*}}
\newcommand{\eeqas}{\end{eqnarray*}}

\usepackage{algorithm}
\usepackage{algpseudocode}

\def\bbe{{\mathbb{E}}}
\def\vgap{\vspace*{.1in}}
\newcommand{\nn}{\nonumber}

\renewcommand{\top}{{T}}

\def\cL{{\cal L}}

\def\cR{{\cal R}}
\def\cG{{\cal G}}
\def\cF{{\cal F}}

\def\cK{{\cal K}}
\def\cS{{\cal S}}

\def\cI{{\cal I}}

\def\vgap{\vspace*{.1in}}

\title{Projected gradient methods for nonconvex and stochastic smooth optimization: new complexities and auto-conditioned stepsizes
}
\author{
Guanghui Lan \and Tianjiao Li \and Yangyang Xu
\thanks{Coauthors of this paper are listed in alphabetical order. Tianjiao Li was at Georgia Tech when most of this work was performed. Guanghui Lan and Tianjiao Li were partially supported by the Air Force Office of Scientific Research grant FA9550-22-1-0447. Yangyang Xu was partially supported by NSF grant DMS-2208394 and the Office of Naval Research grant N00014-22-1-2573.}
}
\institute{
Guanghui Lan\\
  School of Industrial and Systems
    Engineering, Georgia Institute of Technology, Atlanta, GA 30332.\\
    Email: {\tt george.lan@isye.gatech.edu}.
    \vspace{0.2cm}\\
Tianjiao Li\\
  IBM Research, Yorktown Heights, NY 10598.\\
    Email: {\tt tianjiao-li@ibm.com}.
    \vspace{0.2cm}\\
    Yangyang Xu\\
  Department of Mathematical Sciences, Rensselaer Polytechnic Institute, Troy, NY 12180.\\
    Email: {\tt xuy21@rpi.edu}.\\
}

\date{the date of receipt and acceptance should be inserted later}

\titlerunning{Projected gradient methods for nonconvex and stochastic smooth optimization}

\begin{document}

\maketitle

\begin{abstract}
We present a novel class of projected gradient (PG) methods for minimizing a smooth but not necessarily convex function over a convex compact set. 
We first provide a novel analysis of the \rev{``vanilla'' PG}{constant-stepsize PG} method, achieving \rev{the best-known}{a unified} iteration complexity for finding an approximate stationary point for \rev{}{convex and nonconvex problems}. We then develop an ``auto-conditioned'' projected gradient (AC-PG) variant that achieves the same iteration complexity without requiring the input of the Lipschitz constant of the gradient or any line search procedure. The key idea is to estimate the Lipschitz constant using first-order information gathered from the previous iterations, and to show that the error caused by underestimating the Lipschitz constant can be properly controlled. We then generalize the PG methods to the stochastic setting, by proposing a stochastic projected gradient (SPG) method and a variance-reduced stochastic gradient (VR-SPG) method, achieving new \rev{}{unified} complexity bounds in different oracle settings. We also present auto-conditioned stepsize policies for both stochastic PG methods and establish comparable convergence guarantees.
\end{abstract}


\section{Introduction}\label{sec_intro}
Nonconvex optimization has recently found broad applications in machine learning \cite{mairal2009online, mason1999boosting, shi2014sparse, liu2015sparse}, reinforcement learning  \cite{puterman2014markov, sutton2018reinforcement}, and  simulation optimization \cite{andradottir1998review, fu2002optimization}. In this paper, we consider the following general nonlinear, possibly nonconvex, optimization problem 
\begin{align}\label{main_prob}
f^* := \min_{x \in X} f(x),
\end{align}
where $X \subseteq \bbr^n$ is a convex compact set with diameter $D_X := \max_{x, y\in X}\|x-y\|$, and $f:X\rightarrow \mathbb{R}$ is a continuously differentiable but possibly nonconvex function. Let $x^*$ denote one minimizer of \eqref{main_prob} and assume $f$ is smooth with Lipschitz continuous gradient, i.e., for some $L \geq 0$,
\begin{align}\label{smoothness}
\|\nabla f(x) - \nabla f(y)\| \leq L\|x - y\|,\quad \forall x, y \in X.
\end{align}
It indicates that $L$ is the upper curvature bound such that
\begin{align}\label{smoothness_2}
f(x) - f(y) - \langle \nabla f(y), x - y\rangle \leq \tfrac{L}{2}\|x-y\|^2, \quad \forall x, y \in X.
\end{align}
Additionally, we assume that $f$ is weakly convex with a lower curvature parameter $l \in [0, L]$, i.e.,
\begin{align}\label{weakly_convex}
f(x)- f(y) - \langle \nabla f(y), x - y\rangle \geq - \tfrac{l}{2}\|x-y\|^2, \quad \forall x, y \in X.
\end{align}

It is easy to see that $f$ is a convex function \revision{}{on $X$} when $l = 0$. Incorporating the lower curvature parameter~$l$ enables a unified treatment for both convex and nonconvex problems. This unified approach is particularly valuable for solving real-world nonlinear programming problems. For instance, many general nonconvex objective functions exhibit local convexity, and a unified treatment allows the algorithm to leverage the local convexity structure. Additionally, in black-box settings where the convexity of the objective function is unknown,  a unified approach \rev{helps address such structural ambiguity effectively}{can still achieve efficient convergence without prior knowledge of the underlying convexity}. \rev{}{Moreover, many smooth optimization problems arising from machine learning and data science applications can have much larger upper curvature than lower curvature, such as 
Cauchy-regularized synthetic-aperture-radar reconstruction
\cite{karakus2021solving} and CT/MRI reconstruction with learned
smooth weakly convex regularizers \cite{goujon2024learning}.
These problems take the form
$f(x)=\tfrac12\|Ax-b\|^2+\lambda R(x)$, where $R$ has
$L_R$-Lipschitz continuous gradient and is $\rho$-weakly convex.
Then the curvature bounds are
$L=\|A\|^2+\lambda L_R$ and $l=\lambda\rho$, thus the convex data-fitting term can contribute substantial
positive curvature without increasing the negative-curvature bound. In such settings, leveraging the lower-curvature information allows us to design more efficient optimization algorithms.}

First-order methods are widely used to tackle nonconvex optimization problems. 
In particular, it is well-known that the \revision{``vanilla''}{} gradient descent method requires at most $\mathcal{O}\big(L (f(x_0) - f^*) / \epsilon^2\big)$ gradient evaluations to achieve an $\epsilon$-stationary point such that $\|\nabla f(x)\| \leq \epsilon$, for unconstrained problems (see, e.g., \cite{nesterov2004introductory}). This bound is tight for gradient descent, as supported by Cartis et al. \cite{cartis2010complexity}. In the stochastic setting with access to an unbiased stochastic first-order oracle, Ghadimi and Lan \cite{ghadimi2013stochastic} proposed a stochastic gradient method, achieving $\mathcal{O}(1/\epsilon^4)$ sample complexity for outputting an $\epsilon$-stationary solution where $\bbe[\|\nabla f(x)\|^2] \leq \epsilon^2$. 
Moreover, Ghadimi et al. \cite{GhaLanZhang13-1} analyzed the projected gradient descent method and its stochastic variants for constrained composite problems and established the same complexity bounds as in \cite{ghadimi2013stochastic}, for achieving an $\epsilon$-approximate stationary point in terms of the projected gradient mapping (see Subsection \ref{subsec_notation}). 
Further studies on stochastic subgradient methods for nonsmooth nonconvex problems can be found in \cite{davis2019stochastic}.  In general, the $\mathcal{O}(1/\epsilon^4)$ sample complexity is optimal for smooth nonconvex stochastic problems, as supported by \cite{arjevani2023lower}. However, with an additional assumption on the (mean-squared) Lipschitz continuity of the stochastic gradient, the sample complexity can be reduced to $\mathcal{O}(1/\epsilon^3)$ through variance-reduction techniques, see, e.g., \cite{LanBook2020, fang2018spider, cutkosky2019momentum, wang2019spiderboost, tran2022hybrid, pham2020proxsarah,xu2023momentum, mancino2023proximal}.

Despite these advances, all the above-mentioned analyses of (projected) gradient descent and its stochastic variants do not leverage the lower curvature information $l$, thus failing to provide a unified treatment of convex and nonconvex problems. 
In an effort to address this issue, Ghadimi and Lan \cite{ghadimi2016accelerated} proposed an accelerated gradient method (generalized from Nesterov's accelerated gradient descent \cite{nesterov1983method}) and established a unified gradient complexity of 
\beq \label{eq:acc_complexity1}
\mathcal{O}\big((\tfrac{LD_X}{\epsilon})^{\frac{2}{3}} + \tfrac{L l D_X^2}{\epsilon^2}\big)
\eeq 
in the deterministic setting; see further developments for accelerated gradient methods with unified complexity bounds in \cite{liang2021average, liang2023average}. 
On the other hand, proximal-point methods \cite{rockafellar1976monotone} have also been applied to tackle nonconvex optimization problems \cite{hare2009computing, davis2019proximally,LanOuyangZhang2023}, where \rev{}{the work \cite{LanOuyangZhang2023} establishes the gradient complexity of
\beq \label{eq:acc_complexity2}
\mathcal{O}\big( \sqrt{\tfrac{L\|x_0-x^*\|}{\epsilon}} + \tfrac{\sqrt{Ll}(f(x_0) - f(x^*))}{\epsilon^2} \big).
\eeq
}
Observe that, while the first term in \eqref{eq:acc_complexity2} improves upon the corresponding term in \eqref{eq:acc_complexity1}
especially in terms of its dependence on $\epsilon$, the second term in \eqref{eq:acc_complexity2} is not necessarily better than its counterpart in \eqref{eq:acc_complexity1} when the feasible set is bounded. This is because the second term can only be bounded by $\sqrt{lL^3} D_X^2 / \epsilon^2$, as indicated in \eqref{smoothness_2}. \rev{}{Additionally, Kong et al.~\cite{kong2019complexity} proposes a proximal-point method which achieves the gradient complexity of 
\begin{align}\label{eq:acc_complexity3}
\mathcal{O}\big(\tfrac{\sqrt{Ll}\cdot\min\{f(x_0) - f(x^*), ~l\|x_0-x^*\|^2\}}{\epsilon^2} + \sqrt{\tfrac{L}{l}}\log_1^+\big(\tfrac{L}{l}\big)\big).
\end{align}
The first term in this complexity is no worse than the second term in \eqref{eq:acc_complexity2} and, when considering the dependence on $D_X^2$, it saves a factor of $O(L/l)$. However, \eqref{eq:acc_complexity3} is not a unified complexity bound like \eqref{eq:acc_complexity1} and \eqref{eq:acc_complexity2}, as it does not reduce to a convex rate when we send $l\rightarrow 0$, since the second term blows up. 
}

Compared to the aforementioned accelerated gradient methods or proximal-point type methods, the ``vanilla'' gradient descent method and its stochastic variants are more broadly applicable in modern large-scale optimization tasks, such as training \revision{large language models \cite{zhao2023survey}}{large-scale machine learning models \cite{LanBook2020, bottou2018optimization}}, due to their simple algorithmic structure and low memory requirements. However, it remains unclear whether gradient descent and its stochastic variants can achieve the unified complexities that incorporate the lower curvature information.

Another common challenge in implementing the first-order methods lies in the choice of the stepsizes. In particular, we usually need to estimate the upper curvature $L$ to define the stepsizes, but $L$ can be unavailable or difficult to compute in many practical scenarios. For example, for a quadratic function $f(x) := \tfrac{1}{2}x^\top A x$, it can be time-consuming to estimate $\|A\|$ when $A$ is a large-scale dense matrix. Additionally, even when $L$ is available, the stepsizes that rely on this global parameter can be too conservative, leading to slow convergence of the algorithms. Backtracking or line search procedures \cite{armijo1966minimization} are commonly used in tackling this issue. 
However, line search subroutines usually increase the per-iteration computational cost. Moreover, performing line search or backtracking in the stochastic settings can be particularly challenging, due to the additional ``bias'' of the stochastic gradient estimators created by the search stopping condition. To alleviate these limitations, there is a growing interest in the development of adaptive first-order methods that do not require line search procedures, especially for convex optimization (see, e.g., \cite{orabona2016coin, carmon2022making, ivgi2023dog, malitsky2019adaptive, latafat2023adaptive, li2023simple, zhou2024adabb, lan2024auto}). Notably, the recent work \cite{li2023simple} proposed an ``auto-conditioned'' fast gradient method (AC-FGM), achieving the uniformly optimal rate for smooth, weakly smooth, and nonsmooth convex optimization problems. The term ``auto-conditioned” refers to the property that the method can automatically adapt to the problem parameters, such as the upper curvature $L$, without the employment of line search procedures. A key motivating question for our study is whether it is feasible to develop such ``auto-conditioned'' methods for smooth nonconvex optimization. \rev{}{It is noteworthy that, after we released the initial version of this work, \cite{ye2025simple} proposes a line search-free adaptive proximal gradient method for nonconvex optimization. However, its analyses for convex and nonconvex problems are separated, and the convex guarantees are on the function-value gap rather than the stronger projected gradient mapping.
}

\subsection{Contributions and organization} This paper addresses the above challenges by providing new complexity bounds for a class of projected gradient methods that unify convex and nonconvex optimization problems.
 We also introduce novel auto-conditioned stepsize policies that do not require prior knowledge of upper or lower curvatures, or the employment of line search procedures. Our contributions are summarized as follows.

Firstly, we provide a new analysis of the \revision{``vanilla''}{constant-stepsize} projected gradient (PG) method that incorporates both upper and lower curvatures. In particular, we show that PG achieves an $\epsilon$-approximate stationary point with at most 
\begin{align}\label{bound_1}
\mathcal{O}\left( \tfrac{LD_X}{\epsilon} + \tfrac{L l D_X^2}{\epsilon^2} \right)
\end{align}
iterations. For nonconvex problems ($l >0$), the dominating term matches \rev{best-known}{the complexity bound \eqref{eq:acc_complexity1}} achieved by accelerated methods \cite{ghadimi2016accelerated}. For convex problems ($l = 0$), the complexity bound reduces to $
\mathcal{O}\left( \tfrac{LD_X}{\epsilon}\right),
$
matching the sharp convergence rate of PG for \revision{}{smooth convex optimization \cite{teboulle2023elementary} up to a constant factor}.\footnote{Here, we only refer to projected gradient methods with short stepsizes less than $2/L$. Long-stepsize variants of gradient descent with improved convergence rates can be found in \cite{grimmer2025accelerated, bok2024accelerating,kim2021optimizing} and references therein.}

Secondly, we develop an auto-conditioned variant of the projected gradient method (AC-PG) that does not require the input of either the upper curvature $L$ or the lower curvature $l$, or any line search procedures. AC-PG only uses the past iterates to estimate the upper curvature in each iteration. Although the potential underestimation of $L$ could cause increases in function value, the total number of iterations with such increases can be properly bounded, and consequently, the error can be properly controlled. In particular, AC-PG attains the following oracle complexity 
$$
\mathcal{O}\left( \tfrac{LD_X}{\epsilon} +\tfrac{L l D_X^2}{\epsilon^2} + \log \tfrac{L}{L_0}\right),
$$
matching the new complexity \eqref{bound_1} of PG up to an additive logarithmic constant. 

Thirdly, we extend the new analyses of PG to the stochastic setting where only noisy first-order information about $f$ is available. We first propose a stochastic projected gradient (SPG) method equipped with mini-batches and establish the sample complexity bound
\begin{align}\label{bound_2}
\mathcal{O}\left(\tfrac{L D_X}{\epsilon} + \tfrac{L l D_X^2}{\epsilon^2} + \tfrac{LD_X \sigma^2}{\epsilon^3} + \tfrac{Ll D_X^2\sigma^2}{\epsilon^4} \right),
\end{align}
for finding an $\epsilon$-approximate solution in expectation. 
Here $\sigma^2$ is an upper bound on the variance of the stochastic gradient.
This complexity improves the one in \cite{GhaLanZhang13-1} by incorporating the lower curvature information. Moreover, we propose auto-conditioned stepsizes for SPG, namely AC-SPG, under the extra assumption on the Lipschitz continuity of the stochastic gradient and \rev{}{prove unified iteration and sample complexity bounds that nearly match the SPG method.}

Fourthly, under the extra assumption on Lipschitz continuity of the stochastic gradient, we further improve the sample complexity \eqref{bound_2} by incorporating the variance-reduction techniques. Specifically, our novel variance-reduced stochastic projected gradient (VR-SPG) method can obtain an $\epsilon$-approximate solution with a sample complexity of
\begin{align}\label{bound_3}
\O\left( \tfrac{\bar\cL D_X}{\epsilon} + \tfrac{\bar\cL lD_X^2}{\epsilon^2} + \tfrac{\sigma^2}{\epsilon^2} \log \tfrac{\sigma}{\epsilon}  + \tfrac{\bar\cL D_X\sigma}{\epsilon^2} + \tfrac{\bar\cL l D_X^2 \sigma}{\epsilon^3}\right),
\end{align}
where $\bar \cL$ is the mean-squared Lipschitz constant of the stochastic gradient. To the best of our knowledge, this is the first $\mathcal{O}(1/\epsilon^3)$ complexity bound for nonconvex optimization that takes advantage of the lower curvature~$l$. When considering a convex problem with $l=0$, the \revision{}{dominating term in} complexity bound reduces to $\widetilde \O (1/\epsilon^2)$, which is nearly optimal for \revision{}{stochastic convex optimization when $\sigma$ is large}. Additionally, we propose auto-conditioned stepsize policies for VR-SPG, namely AC-VR-SPG, and establish \rev{$\mathcal{O}(1/\epsilon^3)$ sample complexity guarantees \revision{}{for stochastic nonconvex optimization}\footnote{\revision{}{It should be noted that, for AC-VR-SPG, establishing unified complexity bounds that incorporate the lower curvature $l$ as in \eqref{bound_3} is highly nontrivial. This is because the algorithm also exploits the Lipschitz continuity of stochastic gradients in the variance-reduction operators, which introduces additional challenges in controlling the error caused by underestimation of the local Lipschitz constants. Therefore, we leave this question for future work.}}}{unified iteration and sample complexity bounds which are comparable to the VR-SPG method}.

The remaining sections of the paper are organized as follows. In Section~\ref{sec_deterministic}, we present the new analysis and complexities for the PG method, together with its parameter-free variant, AC-PG. In Section~\ref{sec_stochastic}, we introduce the SPG method for the stochastic setting and provide new complexity bounds, as well as the parameter-free AC-SPG method with comparable complexity bounds. Section \ref{sec_variance_reduction} further equips SPG with the variance-reduction techniques, establishes improved complexity bounds, and provides auto-conditioned stepsize policies. Finally, in Section~\ref{sec_numerics}, we present numerical experiments that demonstrate the empirical advantages of the proposed methods.

\subsection{Notation}\label{subsec_notation}
In this paper, we use the convention that $\tfrac{0}{0} = 0$. Unless stated otherwise, we let $\langle \cdot, \cdot\rangle$ denote the Euclidean inner product and $\|\cdot\|$ denote the corresponding Euclidean norm ($\ell_2$-norm). Given a matrix $A$, we use $\|A\|$ to denote its spectral norm. We let $[k]$ denote the set $\{1, ..., k\}$. For $a, b \in \mathbb{Z}_+$, we use $\mathrm{mod}(a, b)$ to denote the modulo operation, i.e., the remainder after $a$ is divided by $b$.

For convenience, we define
\begin{align}
P_X(x, g, \gamma) := \gamma(x - x^+), \text{ where  
$x^+ = \arg \min_{u \in X}\left\{\langle g, u\rangle + \tfrac{\gamma}{2}\|x-u\|^2\right\}.$}
\end{align}
If we take $g = \nabla f(x)$, then $P_X(x, \nabla f(x), \gamma)$ defines a projected gradient of $f$. Clearly, $P_X(x, \nabla f(x), \gamma)$ reduces to $\nabla f(x)$ if $x^+$ is an interior point of $X$. It follows from Lemma 6.3 in \cite{LanBook2020} that if $\|P_X(x, \nabla f(x), \gamma)\| \leq \epsilon$, then
\begin{align*}
-\nabla f(x^{+}) \in N_X(x^{+}) + B\big( \epsilon (\tfrac{L}{\gamma} + 1)\big),
\end{align*}
where $N_X$ denotes the normal cone of $X$ given by $N_X(x^{+}):= \{d \in \bbr^n: \langle d, x- x^+\rangle\leq 0, \forall x \in X \}$ and $B(r):= \{x \in \bbr^n: \|x\|\leq r\}$. Therefore, when $\gamma = \Omega(L)$, the condition $\|P_X(x, \nabla f(x), \gamma)\| \leq \epsilon$ indicates that $x^+$ satisfies the first-order optimality condition with an $\epsilon$-level perturbation, thus it can serve as a valid termination criterion.

\section{Projected gradient methods for nonconvex optimization}\label{sec_deterministic}
In this section, we consider deterministic nonconvex optimization where we have access to exact first-order information. In Subsection \ref{subsec_deterministic_1}, we present a novel analysis of the \revision{``vanilla''}{constant-stepsize} projected gradient (PG) method that achieves the best-known convergence rate for finding an approximate stationary point, unifying convex and nonconvex optimization. In Subsection~\ref{subsec_deterministic_2}, we introduce an ``auto-conditioned'' variant of the PG method, which does not require the input of the Lipschitz constant of the gradient or any line search procedure, but can still achieve the same convergence rate.

\subsection{A novel analysis of the projected gradient (PG) method}\label{subsec_deterministic_1}
We consider the simplest projected gradient (PG) method in Algorithm \ref{alg_1_0}. 
\begin{algorithm}[H]\caption{Projected gradient (PG) method}\label{alg_1_0}
	\begin{algorithmic}
		\State{\textbf{Input}: initial point $x_0 \in X$. 
  }
		\For{$t=1,\cdots, k$}
		\State{
  \begin{align}
  x_t &= \arg \min_{x\in X} \left\{\langle \nabla f(x_{t-1}) , x \rangle + \tfrac{\gamma_t}{2}\|x_{t-1} - x\|^2 \right\}. \label{prox-mapping_0}
  \end{align}
  }
		\EndFor
	\end{algorithmic}
\end{algorithm}
Here, the update \eqref{prox-mapping_0} can be equivalently written as 
\begin{align*}
    x_t = \mathcal{P}_X\left(x_{t-1} - \gamma_t^{-1} \nabla f(x_{t-1})\right),
\end{align*}
where $\mathcal{P}_X$ defines the Euclidean projection onto the feasible region $X$, and $\gamma_t^{-1}$ can be viewed as the stepsize. For simplicity, we let $g_{X, t}$ denote the projected gradient at step $t$, i.e.,
\begin{align}\label{def_projected_gradient}
g_{X, t} \equiv P_X(x_t, \nabla f(x_t), \gamma_{t+1}) := \gamma_{t+1}(x_{t} - x_{t+1}).
\end{align}

It is well-known that the (projected) gradient descent method can find an $\epsilon$-stationary solution of problem~\eqref{main_prob} s.t. $\|g_{X, k}\| \le \epsilon$ within at most $k = \mathcal{O}\left([L (f(x_0) - f^*)]/\epsilon^2\right)$ iterations. However, this complexity bound does not take advantage of the weak-convexity property \eqref{weakly_convex}. In the following theorem, we present a new analysis for PG and obtain a new complexity bound which reflects both upper curvature $L$ and lower curvature $l$.
\begin{theorem}\label{thm_0}
Let $\{x_t\}$ be generated by Algorithm \ref{alg_1_0} \revision{}{with an initial search point $x_0$ and} $\gamma_t = \gamma \geq L, \forall\, t$. \revision{}{Let $\mathcal{E}_X(x_0):= \{x \in X| f(x) \leq f(x_0)\}$ and $D_{X, x_0}:= \max_{x, y\in \mathcal{E}_X(x_0)}\|x-y\|$.}  Then we have
\begin{align}\label{res_1_thm_0}
\min_{0\leq t \leq k-1}\|g_{X,t}\|^2 \leq \left(\tsum_{t=1}^{k} \big(\tfrac{2t-1}{2\gamma} - \tfrac{tL}{2\gamma^2}\big)\right)^{-1}\left(\tfrac{\gamma}{2}\revision{D_X^2}{\|x_0 - x^*\|^2} + \tfrac{lk}{2}\revision{D_X^2}{D_{X, x_0}^2}\right).
\end{align}
Specifically, if $\gamma_t = L, \forall\, t$, then
\begin{align}\label{res_2_thm_0}
\min_{0\leq t \leq k-1}\|g_{X,t}\|^2 \leq \tfrac{2L^2\revision{D_X^2}{\|x_0 - x^*\|^2}}{k(k-1)} + \tfrac{2 L l \revision{D_X^2}{D_{X, x_0}^2}}{k-1}.
\end{align}
\end{theorem}
\begin{proof}
By the optimality condition of \eqref{prox-mapping_0}, we have the following three-point inequality, i.e., for any $x \in X$,
\begin{align*}
\langle \nabla f(x_{t-1}), x_t - x \rangle + \tfrac{\gamma_t}{2}\|x_{t-1} - x_t\|^2 + \tfrac{\gamma_t}{2}\|x_t - x\|^2 \leq \tfrac{\gamma_t}{2}\|x_{t-1} - x\|^2.
\end{align*}
Adding $f(x_{t-1})$ on both sides of the above inequality and rearranging the terms, we have
\begin{align}\label{eq_1_0_0}
&f(x_{t-1}) + \langle \nabla f(x_{t-1}), x_t - x_{t-1}\rangle + \tfrac{\gamma_t}{2}\|x_{t-1} - x_t\|^2 + \tfrac{\gamma_t}{2}\|x_t - x\|^2\nn\\
&\leq f(x_{t-1}) +\langle \nabla f(x_{t-1}), x - x_{t-1}\rangle + \tfrac{\gamma_t}{2}\|x_{t-1} - x\|^2.
\end{align}
Then using the smoothness condition \eqref{smoothness_2} on the LHS of the above inequality, we obtain
\begin{align}\label{eq_1_0}
&f(x_t) + (\tfrac{\gamma_t}{2} - \tfrac{L}{2})\|x_{t-1} - x_t\|^2 + \tfrac{\gamma_t}{2}\|x_t - x\|^2\nn\\
&\leq f(x_{t-1}) +\langle \nabla f(x_{t-1}), x - x_{t-1}\rangle + \tfrac{\gamma_t}{2}\|x_{t-1} - x\|^2.
\end{align}
Taking $x = x_{t-1}$ in the above inequality provides us with
\begin{align}\label{eq_2_0}
(\gamma_t -\tfrac{L}{2})\|x_t - x_{t-1}\|^2 \leq f(x_{t-1}) - f(x_t).
\end{align}
Clearly, this inequality indicates that the function value is non-increasing over iterations as we choose $\gamma_t \geq L/2$. \revision{}{Therefore, we have  $\forall t \in \mathbb{Z}_+, x_t \in \mathcal{E}_{X}(x_0)$, and consequently, $\|x_t - x^*\|^2 \leq D_{X, x_0}^2$.}
Setting $x = x^*$ in Ineq.~\eqref{eq_1_0} and using the weakly convex condition $f(x_{t-1}) +\langle \nabla f(x_{t-1}), x^* - x_{t-1}\rangle \leq \revision{}{f(x^*)}+ l\|x_{t-1} - x^*\|^2/2$, we obtain
\begin{align}\label{eq_3_0}
&f(x_t) + (\tfrac{\gamma_t}{2} - \tfrac{L}{2})\|x_{t-1} - x_t\|^2 + \tfrac{\gamma_t}{2}\|x_t - x^*\|^2\leq f(x^*) + \tfrac{l}{2}\|x_{t-1} - x^*\|^2 + \tfrac{\gamma_t}{2}\|x_{t-1} - x^*\|^2.
\end{align}
By taking a telescope sum of the above inequality from $t=1$ to $k$, and invoking $\gamma_t = \gamma$, we have
\begin{align}\label{eq_4_0}
&\tsum_{t=1}^k \left[f(x_t) -f(x^*) + (\tfrac{\gamma}{2} - \tfrac{L}{2})\|x_{t-1} - x_t\|^2 \right]\nn\\
& \leq \tfrac{\gamma}{2}\|x_0 - x^*\|^2  - \tfrac{\gamma}{2}\|x_k -x^*\|^2 + \tsum_{t=1}^k \tfrac{l}{2}\|x_{t-1}-x^*\|^2\nn\\
& \leq \tfrac{\gamma}{2}\revision{D_X^2}{\|x_0 - x^*\|^2} + \tfrac{lk}{2}\revision{D_X^2}{D_{X, x_0}^2}.
\end{align}
We can further lower bound the LHS of the above inequality by 
\begin{align}\label{eq_5_0_0}
&\tsum_{t=1}^k \left[f(x_t) -f(x^*) + (\tfrac{\gamma}{2} - \tfrac{L}{2})\|x_{t-1} - x_t\|^2 \right]\nn\\
& \geq \tsum_{t=1}^k \left[f(x_t) -f(x_k) + (\tfrac{\gamma}{2} - \tfrac{L}{2})\|x_{t-1} - x_t\|^2 \right]\nn\\
& = \tsum_{t=1}^k \left[\tsum_{j=t}^{k-1}\big(f(x_j) -f(x_{j+1})\big) + (\tfrac{\gamma}{2} - \tfrac{L}{2})\|x_{t-1} - x_t\|^2 \right]\nn\\
& \overset{(i)}\geq \tsum_{t=1}^{k} \left[\tsum_{j=t}^{k-1}(\gamma - \tfrac{L}{2})\|x_j - x_{j+1}\|^2 + (\tfrac{\gamma}{2} - \tfrac{L}{2})\|x_{t-1} - x_t\|^2 \right]\nn\\
&  = \tsum_{t=1}^{k}\big((t-\tfrac{1}{2})\gamma - \tfrac{t}{2}L\big)\|x_{t-1} - x_t\|^2\nn\\
& = \tsum_{t=1}^{k}\big(\tfrac{2t-1}{2\gamma} - \tfrac{tL}{2\gamma^2}\big)\|g_{X, t-1}\|^2,
\end{align}
where $(i)$ follows from \eqref{eq_2_0}.
By combining Ineqs.~\eqref{eq_4_0} and \eqref{eq_5_0_0}, we obtain
\begin{align*}
\min_{0\leq t \leq k-1}\|g_{X,t}\|^2 &\leq \left(\tsum_{t=1}^{k} \big(\tfrac{2t-1}{2\gamma} - \tfrac{tL}{2\gamma^2}\big)\right)^{-1}\tsum_{t=1}^{k}\big(\tfrac{2t-1}{2\gamma} - \tfrac{tL}{2\gamma^2}\big)\|g_{X, t-1}\|^2\\
&\leq \left(\tsum_{t=1}^{k} \big(\tfrac{2t-1}{2\gamma} - \tfrac{tL}{2\gamma^2}\big)\right)^{-1}\left(\tfrac{\gamma}{2}\revision{D_X^2}{\|x_0 - x^*\|^2} + \tfrac{lk}{2}\revision{D_X^2}{D_{X, x_0}^2}\right),
\end{align*}
which completes the proof.
\end{proof}
\vgap

In view of Theorem~\ref{thm_0}, to find an $\epsilon$-stationary solution such that $\|g_{X, t}\| \leq \epsilon$, the number of PG iterations is bounded by
\begin{align}\label{complexity_thm_0}
    \mathcal{O}\left(\tfrac{L\revision{D_X}{\|x_0 - x^*\|}}{\epsilon} + \tfrac{Ll\revision{D_X^2}{D_{X, x_0}^2}}{\epsilon^2}\right).
\end{align}
This complexity bound unifies both convex and nonconvex optimization \revision{}{for the constant-stepsize PG method}. Specifically, if we set $l=0$, \revision{}{this bound is suboptimal among all first-order methods, but matches (within a constant factor) the complexity of the constant-stepsize PG method for minimizing the projected gradient of a smooth convex function; see \cite{teboulle2023elementary} and references therein. When $l>0$, the dominating term matches the best-known complexity bound achieved by accelerated methods \cite{ghadimi2016accelerated}.} To the best of our knowledge, this is the first analysis in the literature that achieves \revision{}{this unified complexity bound \eqref{complexity_thm_0}} through a gradient descent type method, although an early observation was made by \cite{LanMonteiro2018}. \revision{}{Moreover, the non-ascent property of constant-stepsize PG allows us to achieve the complexity bound \eqref{complexity_thm_0} that depends on $D_{X, x_0}$, which is stronger than depending on the diameter constant $D_X$. Such dependence on the initialization $x_0$ enables the algorithm to benefit from the restarting scheme when the problem exhibits local growth or error-bound structures.} 

\begin{remark}
The guarantees in Theorem~\ref{thm_0} and all the remaining sections of this paper naturally extend to the problem of minimizing a composite function
\begin{align}\label{def_composite_problem}
\Psi^* := \min_{x \in X} \{\Psi(x):= f(x) +h(x) \},
\end{align}
where $h: X \rightarrow \bbr$ is a ``prox-friendly'' closed convex function such that \revision{}{the} ``proximal-mapping'' subproblem
\begin{align*} 
x^+ = \arg \min_{u \in X}\left\{\langle g, u\rangle + h(u) + \tfrac{\gamma}{2}\|x-u\|^2\right\}
\end{align*}
can be efficiently solved. Specifically, we define $P_X(x, g, \gamma) := \gamma(x-x^+)$ and slightly modify step \eqref{prox-mapping_0} to
\begin{align*}
  x_t = \arg \min_{x\in X} \left\{\langle \nabla f(x_{t-1}) , x \rangle + h(x) + \tfrac{\gamma_t}{2}\|x_{t-1} - x\|^2 \right\}. 
  \end{align*}
Then by utilizing the convexity of $h$, we can show that Theorem~\ref{thm_0} also holds for the composite problem~\eqref{def_composite_problem}, with the composite projected gradient $P_X(x, \nabla f(x), \gamma)$ serving as the termination criterion.  \revision{}{Here, we require $\bar X:=\{x \in X| h(x) < +\infty\}$ to be a closed convex set, and $\mathcal{E}_{\bar X}(x_0):= \{x \in \bar X| \Psi(x) \leq \Psi(x_0)\}$ to be bounded.} \revision{}{The detailed derivation can be found in Appendix~\ref{appendix_prox_extension}.} Moreover, all the following results in this paper can be generalized similarly, \revision{}{but these extensions require the feasible set of the composite problem $\bar X$ to be a compact set. How to extend the unified complexity bounds, e.g., Theorem~\ref{thm_0}, to the general composite setting with an arbitrary closed convex feasible set like $\mathbb{R}^n$ remains open. 
}
\end{remark}

\revision{}{
\begin{remark}
The unified analysis in Theorem~\ref{thm_0} can also be generalized to the setting where $f$ can be potentially weakly convex, convex, or strongly convex, i.e., for $l \in [-L, L]$,
\begin{align}\label{general_convexity_0}
f(x)- f(y) - \langle \nabla f(y), x - y\rangle \geq - \tfrac{l}{2}\|x-y\|^2, \quad \forall x, y \in X.
\end{align}
In Appendix \ref{subsection_strongly_convex_extension}, we provide unified convergence guarantees for PG (Algorithm \ref{alg_1_0}) in terms of the norm of the projected gradient mapping, which covers the strongly convex setting ($l < 0$). However, when using the ``auto-conditioned'' stepsizes presented in the next subsection, our unified guarantees presented in the following sections do not extend straightforwardly to the strongly convex setting.
\end{remark}}

\subsection{Auto-conditioned projected gradient (AC-PG) method}\label{subsec_deterministic_2}
In this subsection, we consider a parameter-free variant of the PG method, which does not require any prior knowledge of the upper curvature $L$ or lower curvature $l$. Meanwhile, the algorithm is also free of any line search or backtracking procedures.
\begin{algorithm}[H]\caption{Auto-conditioned projected gradient (AC-PG) method}\label{alg_1}
	\begin{algorithmic}
		\State{\textbf{Input}: initial point $x_0 \in X$ and $L_0$ such that 
  and $0 < L_0 \leq L$. 
  }
		\For{$t=1,\cdots, k$}
		\State{
  \begin{align}
  x_t &= \arg \min_{x\in X} \left\{\langle \nabla f(x_{t-1}) , x \rangle + \tfrac{\gamma_t}{2}\|x_{t-1} - x\|^2 \right\}, \label{prox-mapping}\\
  &\text{where }\gamma_t = \hat L_{t-1}:=  \max\{L_0,..., L_{t-1}\}, \label{evaluate_L_t_hat}\\
  L_t &= \tfrac{2(f(x_{t}) - f(x_{t-1}) - \langle \nabla f(x_{t-1}), x_t - x_{t-1}\rangle)}{\|x_t - x_{t-1}\|^2}. \label{evaluate_L_t}
  \end{align}
  }
		\EndFor
	\end{algorithmic}
\end{algorithm}

In Algorithm~\ref{alg_1}, the quantity $L_0$ can be determined by simply choosing two points in $X$ and computing a local smoothness factor using \eqref{evaluate_L_t}. In the following iterations, we calculate in \eqref{evaluate_L_t} the local smoothness estimator $L_t$ whenever the search point $x_t$ becomes available, and the new stepsize will be determined purely by the past local smoothness estimations. \revision{}{Here, computing the local smoothness estimator $L_t$ requires both function value and gradient information, which may introduce additional computation. However, in most cases, estimating the function value is no more expensive than estimating the gradient\footnote{Similar arguments also apply in the stochastic setting.}. Therefore, our scheme does not significantly increase either the per-iteration cost or the total amount of elementary computations. On the other hand, recall that we adopt the convention $\tfrac{0}{0} = 0$, so $L_t = 0$ when $x_t = x_{t-1}$. In numerical implementations, when $x_t$ is very close to $x_{t-1}$, we may add a small constant, e.g., $10^{-10}$, to the denominator of \eqref{evaluate_L_t} for numerical stability, which is adopted in Section~\ref{sec_numerics}. }

It is important to note that by only using past local Lipschitz estimators, we may underestimate the local Lipschitz constant, especially when this constant increases sharply at the current search point. Consequently, the AC-PG method cannot always ensure a decrease in the function value. However, it is worth highlighting that the function value only increases in a limited number of iterations, and the total magnitude of these increases can be appropriately bounded, as they form a summable sequence.
To this end, we partition the iterations $\{x_0, x_1, x_2, ..., x_k\}$ into distinct segments, initiating a new segment whenever 
there exists a sudden increase in $L_t$
such that $L_t > \tfrac{3}{2}\hat L_{t-1}$.
Specifically, the iterates $\{x_0, x_1, x_2, ..., x_k\}$  are divided into $m$ segments, each with a length of $s_i$, $i=1, \ldots, m$:\footnote{With a slight abuse of notation, the first segment has length $s_1+1$.}
\begin{align}\label{def_seg}
\{x_0, x_1, x_2, ..., x_k\} = \big\{\underbrace{x_0^{(1)}, x_1^{(1)}, ..., x_{s_1}^{(1)}}_{\text{segment 1}}, \underbrace{x_1^{(2)}, ..., x_{s_2}^{(2)}}_{\text{segment 2}}, ..., \underbrace{x_1^{(m)}, ..., x_{s_m}^{(m)}}_{\text{segment $m$}}\big\},
\end{align}
where $x_t^{(i)}$ represents the $t$-th iterate ($t=\revision{0}{1}, \ldots, s_i$)
in the $i$-th segment ($i=1, \ldots, m$). Similarly, $\hat L_{t}^{(i)}$ and $L_{t}^{(i)}$ denote the largest and local Lipschitz constants, respectively, at $t$-th iteration of the $i$-th segment. These constants are defined in the same way as  $\hat L_{t}$ and $L_{t}$ in \eqref{evaluate_L_t_hat} and \eqref{evaluate_L_t}, but with this updated indexing. \revision{}{We also apply the same indexing rule to define $\gamma_{t}^{(i)}$, referring to the stepsize at $t$-th iteration of the $i$-th segment.}
By construction, the following properties hold:
\begin{align}
L_1^{(i)} &> \tfrac{3}{2} \hat L_{s_{i-1}}^{(i-1)}, ~\text{for }i \geq 2 \label{eq_5_0}\\
L_t^{(i)} &\leq \tfrac{3}{2} \hat L_{t-1}^{(i)}, ~\text{for } 2\le t \le s_i,~ i\geq 1.\label{eq_5}
\end{align}
 \revision{For simplicity}{For notational simplicity, we denote $x_0^{(1)}=x_0$ and $x_0^{(i)} = x_{s_{i-1}}^{(i-1)}$ for $i\geq 2$. }
It should be noted that, since $L_t \leq L$ for all $t \geq 0$, the total number of segments $m$ can be bounded by
\begin{align}\label{eq_6}
    m \leq \lfloor\log_{\frac{3}{2}}\tfrac{L}{L_0}\rfloor + 1.
\end{align}
Regarding the lower curvature, we define  
\begin{align}\label{def_l_t}
l_t := \tfrac{2(f(x_{t}) + \langle \nabla f(x_{t}), x^* - x_{t}\rangle -f(x^*))}{\|x_t - x^*\|^2}
\end{align}
such that $l_t \leq l$.
It is worth noting that $l_t$ does not need to be explicitly computed during the algorithm's execution and is only utilized in the analysis. 

\vgap

The following theorem establishes the convergence results for Algorithm~\ref{alg_1}.

\begin{theorem}\label{thm_1}
Let $\{x_t\}$ be generated by Algorithm~\ref{alg_1} and assume $k > 2m$.  We have
\begin{align}\label{eq_thm_1}
\min_{t \in [k]}\|g_{X, t}\|^2 \leq \tfrac{48\hat L_{k}^2 D_X^2}{(k-2m)^2} + \tfrac{12\hat L_{k}D_X^2 \cdot (\sum_{t=0}^{k-1}\max\{0, l_t\})}{(k-2m)^2},
\end{align}
where $\hat L_t$ is defined in \eqref{evaluate_L_t_hat} and  $m \leq \lfloor\log_{\frac{3}{2}}\tfrac{L}{L_0}\rfloor + 1$.
\end{theorem}
\begin{proof}
By the optimality condition of \eqref{prox-mapping}, we have the 
same three-point inequality as in \eqref{eq_1_0_0}, i.e.,
\begin{align*}
&f(x_{t-1}) + \langle \nabla f(x_{t-1}), x_t - x_{t-1}\rangle + \tfrac{\gamma_t}{2}\|x_{t-1} - x_t\|^2 + \tfrac{\gamma_t}{2}\|x_t - x\|^2\nn\\
&\leq f(x_{t-1}) +\langle \nabla f(x_{t-1}), x - x_{t-1}\rangle + \tfrac{\gamma_t}{2}\|x_{t-1} - x\|^2.
\end{align*}
By using the definition of \revision{$L_t$}{ $L_t$, we have $f(x_{t-1}) + \langle \nabla f(x_{t-1}), x_t - x_{t-1}\rangle = f(x_t) - \tfrac{L_t}{2}\|x_{t-1} - x_t\|^2$, thus}
\begin{align}\label{eq_1}
&f(x_t) + (\tfrac{\gamma_t}{2} - \tfrac{L_t}{2})\|x_{t-1} - x_t\|^2 + \tfrac{\gamma_t}{2}\|x_t - x\|^2\nn\\
&\leq f(x_{t-1}) +\langle \nabla f(x_{t-1}), x - x_{t-1}\rangle + \tfrac{\gamma_t}{2}\|x_{t-1} - x\|^2.
\end{align}
Taking $x = x_{t-1}$ in the above inequality provides us with
\begin{align}\label{eq_2}
(\gamma_t -\tfrac{L_t}{2})\|x_t - x_{t-1}\|^2 \leq f(x_{t-1}) - f(x_t).
\end{align}
Clearly, when $\gamma_t = \hat L_{t-1} \leq \tfrac{L_t}{2}$, we may not have the descent property of the function values. 
On the other hand, setting $x = x^*$ in Ineq.~\eqref{eq_1} and using the definition of $l_t$ in \eqref{def_l_t}, we obtain
\begin{align}\label{eq_3}
&f(x_t) + (\tfrac{\gamma_t}{2} - \tfrac{L_t}{2})\|x_{t-1} - x_t\|^2 + \tfrac{\gamma_t}{2}\|x_t - x^*\|^2\nn\\
&\leq f(x^*) + \tfrac{l_{t-1}}{2}\|x_{t-1} - x^*\|^2 + \tfrac{\gamma_t}{2}\|x_{t-1} - x^*\|^2.
\end{align}
By taking a telescope sum of the above inequality from $t=1$ to $k$, we have
\begin{align}\label{eq_4}
&\tsum_{t=1}^k \left[f(x_t) -f(x^*) + (\tfrac{\gamma_t}{2} - \tfrac{L_t}{2})\|x_{t-1} - x_t\|^2 \right]\nn\\
& \leq \tfrac{\gamma_1}{2}\|x_0 - x^*\|^2 + \tsum_{t=1}^{k-1} (\tfrac{\gamma_{t+1}}{2} - \tfrac{\gamma_t}{2})\|x_t - x^*\|^2 - \tfrac{\gamma_k}{2}\|x_k -x^*\|^2 + \tsum_{t=1}^k \tfrac{l_{t-1}}{2}\|x_{t-1}-x^*\|^2\nn\\
& \overset{(i)}\leq \tfrac{\gamma_k}{2}D_X^2 + \tfrac{\sum_{t=0}^{k-1}\max\{0, l_t\}}{2}D_X^2 \overset{(ii)}= \tfrac{\hat L_{k-1}}{2}D_X^2 + \tfrac{\sum_{t=0}^{k-1}\max\{0, l_t\}}{2}D_X^2,
\end{align}
where step (i) follows from $\gamma_{t+1} \geq \gamma_t$ 
and the definition of $D_X$, and step (ii) uses $\gamma_k = \hat L_{k-1}$ as defined in~\eqref{prox-mapping}.

Next, we lower bound the LHS of Ineq.~\eqref{eq_4} by using the segments defined in \eqref{def_seg}-\eqref{eq_5}.
For convenience, let us denote $S_i:= \tsum_{t=1}^{s_i}\left[f(x^{(i)}_t) -f(x^*) + (\tfrac{\gamma_t^{(i)}}{2} - \tfrac{L^{(i)}_t}{2})\|x^{(i)}_{t-1} - x^{(i)}_t\|^2 \right]$. Then for any $i \in [m]$, 
\begin{align*}
S_i &\geq \tsum_{t=1}^{s_i}\left[f(x^{(i)}_t) -f(x_{s_i}^{(i)}) + (\tfrac{\gamma^{(i)}_t}{2} - \tfrac{L^{(i)}_t}{2})\|x^{(i)}_{t-1} - x^{(i)}_t\|^2 \right]\nn\\
& = \tsum_{t=1}^{s_i} \left[\tsum_{j=t}^{s_i-1}\big(f(x^{(i)}_j) - f(x^{(i)}_{j+1}) \big) + (\tfrac{\gamma^{(i)}_t}{2} - \tfrac{L^{(i)}_t}{2})\|x^{(i)}_{t-1} - x^{(i)}_t\|^2 \right]\nn\\
& \overset{(i)}\geq \tsum_{t=1}^{s_i} \left[\tsum_{j=t}^{s_i-1}(\gamma_{j+1}^{(i)} - \tfrac{L_{j+1}^{(i)}}{2})\|x^{(i)}_j - x_{j+1}^{(i)}\|^2 + (\tfrac{\gamma^{(i)}_t}{2} - \tfrac{L^{(i)}_t}{2})\|x^{(i)}_{t-1} - x^{(i)}_t\|^2 \right]\nn\\
& \overset{(ii)} = \tsum_{t=2}^{s_i}[(t-\tfrac{1}{2})\gamma_t^{(i)} - \tfrac{t}{2}L_t^{(i)}]\|x^{(i)}_{t-1} - x^{(i)}_t\|^2 + \tfrac{1}{2}(\gamma_1^{(i)}-L_1^{(i)})\|x_1^{(i)} - x_0^{(i)}\|^2,
\end{align*}
where step (i) follows from Ineq.~\eqref{eq_2}, and step (ii) is obtained by rearranging the terms. Recalling that $\gamma_t^{(i)} = \hat L^{(i)}_{t-1}$ and Ineq.~\eqref{eq_5}, we have $L_t^{(i)} \leq \tfrac{3}{2}\gamma_t^{(i)}$ for $t\geq 2$. Then we can further bound the above inequality by 
\begin{align}\label{eq_7}
S_i &\geq \tsum_{t=2}^{s_i}[(t-\tfrac{1}{2})\gamma_t^{(i)} - \tfrac{3t}{4}\gamma_t^{(i)}] \|x^{(i)}_{t-1} - x^{(i)}_t\|^2+ \tfrac{1}{2}(\gamma_1^{(i)}-L_1^{(i)})\|x_1^{(i)} - x_0^{(i)}\|^2\nn\\
& = \tsum_{t=2}^{s_i}(\tfrac{t}{4}-\tfrac{1}{2})\gamma_t^{(i)} \|x^{(i)}_{t-1} - x^{(i)}_t\|^2+ \tfrac{1}{2}(\gamma_1^{(i)}-L_1^{(i)})\|x_1^{(i)} - x_0^{(i)}\|^2\nn\\
& \overset{(i)}\geq  \tsum_{t=2}^{s_i}(\tfrac{t}{4}-\tfrac{1}{2}) \tfrac{1}{\gamma_t^{(i)}}\|g_{X, t-1}^{(i)}\|^2 -\tfrac{L_1^{(i)}}{2}\|x_1^{(i)} - x_0^{(i)}\|^2\nn\\
& \overset{(ii)} \geq \tfrac{(s_i - 1)(s_i-2)}{8\gamma_{s_i}^{(i)}} \min_{t \in [1, s_i-1]}\|g_{X, t}^{(i)}\|^2-\tfrac{L_1^{(i)}}{2}\|x_1^{(i)} - x_0^{(i)}\|^2,
\end{align}
where step (i) follows from the definition of projected gradient in \eqref{def_projected_gradient} and the fact $\gamma_1^{(i)} \geq 0$, and step (ii) is due to $\gamma_{s_i}^{(i)} \geq \gamma_{s_i-1}^{(i)}\geq ...\geq \gamma_{2}^{(i)}$. Then by summing up $S_i$ for $i=1,...,m$, we obtain
\begin{align*}
\text{LHS of \eqref{eq_4}} & = \tsum_{i=1}^m S_i\\
& \overset{(i)}\geq \tsum_{i=1}^m \left(\tfrac{(s_i - 1)(s_i-2)}{8\gamma_{s_i}^{(i)}} \min_{t \in [1, s_i-1]}\|g_{X, t}^{(i)}\|^2-\tfrac{L_1^{(i)}}{2}\|x_1^{(i)} - x_0^{(i)}\|^2 \right)\\
&\overset{(ii)}\geq \tfrac{(\sum_{i=1}^m\sqrt{(s_i-1)(s_i-2)})^2}{8(\sum_{i=1}^m \gamma_{s_i}^{(i)})}\cdot \min_{t \in [1, k]}\|g_{X, t}\|^2 - \tsum_{i=1}^m \tfrac{\hat L_1^{(i)}}{2}\|x_1^{(i)} - x_0^{(i)}\|^2\\
& \overset{(iii)} \geq \tfrac{(k-2m)^2}{24 \gamma_{s_m}^{(m)}}\min_{t \in [1, k]}\|g_{X, t}\|^2 - \tfrac{3\hat L_1^{(m)}}{2}D_X^2\\
& \overset{(iv)}\geq \tfrac{(k-2m)^2}{24 \hat L_{k}}\min_{t \in [1, k]}\|g_{X, t}\|^2 - \tfrac{3\hat L_{k}}{2}D_X^2,
\end{align*}
where step (i) utilizes Ineq.~\eqref{eq_7}, step (ii) is due to the Cauchy-Schwarz inequality, step (iii) stems from the facts that $\tsum_{i=1}^m \hat L_1^{(i)} \le 3 \hat L_1^{(m)}$ due to $\hat L_1^{(i-1)} < \tfrac{2}{3}\hat L_1^{(i)}$,
$\sum_{i=1}^m \gamma_{s_i}^{(i)} \le 3 \gamma_{s_m}^{(m)}$
due to $\gamma_{s_{i-1}}^{(i-1)} \leq \gamma_1^{(i)}<\tfrac{2}{3} \hat L_1^{(i)} \leq \tfrac{2}{3}\gamma_{s_i}^{(i)}$ when $s_i\geq 2$, and $\sum_{i=1}^m\sqrt{(s_i-1)(s_i-2)}
\ge \sum_{i=1}^m (s_i-2) \ge \sum_{i=1}^m s_i - 2m = k - 2m$,
and step (iv) follows from the definition of the segments and the definition of $\hat L_t$. By combining the above inequality with Ineq.~\eqref{eq_4}, we obtain  
\begin{align*}
\tfrac{(k-2m)^2}{24 \hat L_{k}}\min_{t \in [1, k]}\|g_{X, t}\|^2 \leq 2 \hat L_{k} D_X^2 + \tfrac{\sum_{t=0}^{k-1}\max\{0, l_t\}}{2}D_X^2,
\end{align*}
which implies that
\begin{align*}
\min_{t \in [1, k]}\|g_{X, t}\|^2 \leq \tfrac{48\hat L_{k}^2 D_X^2}{(k-2m)^2} + \tfrac{12\hat L_{k}D_X^2 \cdot (\sum_{t=0}^{k-1}\max\{0, l_t\})}{(k-2m)^2},
\end{align*}
and we have thus completed the proof.
\end{proof}

\vgap

Theorem~\ref{thm_1} enables us to establish the complexity bounds for the AC-PG algorithm (Algorithm~\ref{alg_1}). Recall that $m$, the number of times where $L_t > \frac{3}{2} \hat{L}_{t-1}$, is upper bounded by $\mathcal{O}(\log L/L_0)$. Based on the bound in \eqref{eq_thm_1}, the number of AC-PG iterations required to find a solution $x_t \in X$ such that the projected gradient satisfies $\|g_{X, t}\|\leq \epsilon$ is given by:  
\begin{align}\label{complexity_thm_1}
\mathcal{O}\left(\tfrac{LD_X}{\epsilon} + \tfrac{L l D_X^2}{\epsilon^2} + \log \tfrac{L}{L_0}\right),
\end{align}
which aligns with the complexity of PG in \eqref{complexity_thm_0}, up to an additive logarithmic constant. The proof of Theorem~\ref{thm_1} suggests that we achieve this complexity bound by strategically eliminating the first iteration of each segment in \eqref{def_seg}, or equivalently, the one where the local Lipschitz estimator increases sharply by a factor of more than $3/2$. Notice that our guarantees on AC-PG do not achieve dependence on the stronger diameter constant $D_{X, x_0}$ as in \eqref{complexity_thm_0}, due to the loss of non-ascent property when using auto-conditioned stepsizes.


\begin{remark}
Algorithm~\ref{alg_1} employs a non-increasing stepsize policy, which may be conservative for some problems. However, the algorithm can be readily adapted to enhance convergence efficiency. For example, periodically incorporating line search steps or resetting $L_t = L_0$ (e.g., when a new segment is formed) can improve performance without substantially impacting the theoretical complexity bounds.
\end{remark}

\section{Stochastic projected gradient methods for nonconvex stochastic optimization}\label{sec_stochastic}
In this section, we consider the stochastic setting where only stochastic first-order information about $f$ is available.  
In particular, we assume access to a stochastic oracle that provides the stochastic gradient $\cG(x, \xi)$ at the query point $x\in X$, where $\xi \in \Xi$ denotes a random variable with probability distribution supported on the set $\Xi$. 
The $\xi$'s generated by the stochastic oracle are assumed to be independent and identically distributed. We further impose the following two assumptions:
\begin{assumption}[Unbiasness]\label{assump_unbiasness} For any $x \in X$,
\begin{align}
\bbe_\xi [\cG(x, \xi)] = \nabla f(x).\label{unbiasness_G}
\end{align}
\end{assumption}
\begin{assumption}[Bounded variance]\label{assump_variance}  There exists 
$\sigma^2 > 0$ such that for any $x\in X$,
\begin{align}\label{uniform_variance}
\bbe_\xi\|\cG(x, \xi) - \nabla f(x) \|^2 \leq \sigma^2.
\end{align}
\end{assumption}
We consider the mini-batch approach, which is widely used in practice. Specifically, at iteration $t$, we define
$
\hat G_t(x) = \tfrac{1}{b_t} \tsum_{j=1}^{b_t} \cG(x, \xi^t_j), 
$
where $\{\xi^t_j\}_{j=1}^{b_t}$ are mutually independent random variables. Then, we have $\bbe[\hat G_t(x)] = \nabla f(x)$, and it follows from Assumption~\ref{assump_variance} that 
\begin{align*}
\bbe[\|\hat G_t(x) - \nabla f(x)\|^2] \leq \tfrac{\sigma^2}{b_t}.
\end{align*}

In analogy to Section~\ref{sec_deterministic}, we first present a new analysis of the stochastic projected gradient (SPG) method in Subsection~\ref{sub_sec_stoch_1} and then introduce an auto-conditioned variant of SPG with adaptive stepsizes in Subsection~\ref{sec_stoch_expectation}. 

\subsection{Stochastic projected gradient (SPG) method}\label{sub_sec_stoch_1}
In this subsection, we consider the following stochastic variant of the PG method (Algorithm~\ref{alg_1_0}) equipped with mini-batches.
\begin{algorithm}[H]\caption{Stochastic projected gradient (SPG) method}\label{alg_2_0}
	\begin{algorithmic}
		\State{\textbf{Input}: initial point $x_0 \in X$, stepsize parameter $\{\gamma_t\}$, and batch size parameter $\{b_t\}$. 
  }
		\For{$t=1,\cdots, k$}
		\State{
  \begin{align}
  x_t = \arg \min_{x\in X} \left\{\langle  \hat G_t(x_{t-1}) , x \rangle + \tfrac{\gamma_t}{2}\|x_{t-1} - x\|^2 \right\}, \label{prox-mapping_2_0}
  \end{align}
  where 
  \begin{align}
  \hat G_t(x) &:= \tfrac{1}{b_t}\tsum_{j=1}^{b_t} \cG(x, \xi_j^t).\label{def_hat_g_2_0}
  \end{align}
  }
\EndFor
	\end{algorithmic}
\end{algorithm}

For convenience, we define the projected stochastic gradient at iterate $x_t$ as
\begin{align}\label{def_projected_stochastic_gradient}
\hat g_{X, t}:= P_X(x_t, \hat G_{t+1}(x_t), \gamma_{t+1}) = \gamma_{t+1}(x_{t} - x_{t+1}),
\end{align}
and recall the definition of the projected gradient (same as \eqref{def_projected_gradient} in Section~\ref{sec_deterministic}):
\begin{align}\label{def_projected_deterministic_gradient}
g_{X, t}:= P_X(x_t, \nabla f(x_t), \gamma_{t+1}).
\end{align}
 The convergence guarantees of SPG are summarized in the following theorem.

\begin{theorem}\label{thm_2_0}
    Assume Assumptions \ref{assump_unbiasness}-\ref{assump_variance} hold. Let $\{x_t\}$ be generated by Algorithm~\ref{alg_2_0} with $\gamma_t = \gamma > L$. For any $k\geq 2$, let the output $x_{R(k)}$ be randomly chosen from $\{x_1, ..., x_{k-1}\}$ with the following probability mass function:
\begin{align}\label{random_output_0}
\mathbb{P}(x_{R(k)} = x_{t-1}) = \left(\tsum_{j=2}^k W(j)\right)^{-1} \cdot W(t),
\end{align}
where 
$
W(t) := \tfrac{3t-2}{8\gamma}-\tfrac{tL}{4\gamma^2}$. Then we have
    \begin{align}
  &\bbe\left[\|g_{X, R(k)}\|^2 \right]\leq \left(\tsum_{t=2}^k (\tfrac{3t-2}{8\gamma}-\tfrac{tL}{4\gamma^2})\right)^{-1}\left[ (\tfrac{\gamma}{2}+ \tfrac{L}{4} + \tfrac{lk}{2})D_X^2 +\tsum_{t=2}^k(\tfrac{7t-2}{4\gamma}-\tfrac{tL}{2\gamma^2})\tfrac{\sigma^2}{b_t} + \tfrac{\sigma^2}{\gamma b_1}\right].
\end{align}
Specifically, if we take $\gamma = 2L$ and $b_t = \max \left\{ 1, \min\big\{\lceil\tfrac{3t \sigma^2}{4Ll D_X^2}\rceil, \lceil \tfrac{3t k\sigma^2}{4L^2 D_X^2} \rceil \big\}\right\}$, then
\begin{align}\label{SPG_unfied_rate}
  &\bbe\left[\|g_{X, R(k)}\|^2 \right]\leq \tfrac{40L^2 D_X^2}{k (k-1)} + \tfrac{24 L l D_X^2}{k-1}.
\end{align}
\end{theorem}
\begin{proof}
First, the optimality condition of \eqref{prox-mapping_2_0} implies that for any $x \in X$,
\begin{align*}
\langle \hat G_t(x_{t-1}), x_t - x\rangle + \tfrac{\gamma_t}{2}\|x_t - x_{t-1}\|^2 + \tfrac{\gamma_t}{2}\|x_t - x\|^2 \leq \tfrac{\gamma_t}{2}\|x_{t-1} - x\|^2.
\end{align*}
Define $\hat \delta_t(x):=\hat G_t(x) - \nabla f(x)$. Adding $f(x_{t-1})$ on both sides of the above inequality and rearranging the terms yield
\begin{align*}
&f(x_{t-1}) + \langle \nabla f(x_{t-1}), x_t - x_{t-1} \rangle + \langle \hat \delta_t(x_{t-1}), x_t - x\rangle + \tfrac{\gamma_t}{2}\|x_t - x_{t-1}\|^2 + \tfrac{\gamma_t}{2}\|x_t - x\|^2\\
&\leq f(x_{t-1}) + \langle \nabla f(x_{t-1}), x - x_{t-1} \rangle + \tfrac{\gamma_t}{2}\|x_{t-1} - x\|^2. 
\end{align*}
By further using Young's inequality $\langle \hat \delta_t(x_{t-1}), x_t - x_{t-1}\rangle + \tfrac{\gamma_t}{4}\|x_t - x_{t-1}\|^2 \geq -\tfrac{\|\hat \delta_t(x_{t-1})\|^2}{\gamma_t}$, we have
\begin{align*}
&f(x_{t-1}) + \langle \nabla f(x_{t-1}), x_t - x_{t-1} \rangle + \langle \hat \delta_t(x_{t-1}), x_{t-1} - x\rangle - \tfrac{\|\hat \delta_t(x_{t-1})\|^2}{\gamma_t} + \tfrac{\gamma_t}{4}\|x_t - x_{t-1}\|^2 + \tfrac{\gamma_t}{2}\|x_t - x\|^2\\
&\leq f(x_{t-1}) + \langle \nabla f(x_{t-1}), x - x_{t-1}  \rangle + \tfrac{\gamma_t}{2}\|x_{t-1} - x\|^2. 
\end{align*}
Next, we apply the smoothness condition of $f$ (Ineq.~\eqref{smoothness_2}) on the LHS of the above inequality, and obtain
\begin{align}\label{proof_stoch_1}
&f(x_{t}) + (\tfrac{\gamma_t}{4} - \tfrac{L}{2})\|x_t - x_{t-1}\|^2 + \langle \hat \delta_t(x_{t-1}), x_{t-1} - x\rangle - \tfrac{\|\hat \delta_t(x_{t-1})\|^2}{\gamma_t}  + \tfrac{\gamma_t}{2}\|x_t - x\|^2\nn\\
&\leq f(x_{t-1}) + \langle \nabla f(x_{t-1}), x - x_{t-1} \rangle + \tfrac{\gamma_t}{2}\|x_{t-1} - x\|^2. 
\end{align}
Taking $x = x_{t-1}$ in the above inequality and rearranging the terms, then we have
\begin{align}\label{proof_stoch_2}
f(x_{t-1}) - f(x_{t}) \geq (\tfrac{3\gamma_t}{4} - \tfrac{L}{2})\|x_t - x_{t-1}\|^2 - \tfrac{\|\hat \delta_t(x_{t-1})\|^2}{\gamma_t}. 
\end{align}
Setting $x = x^*$ in Ineq. \eqref{proof_stoch_1}, applying the weakly convex condition $f(x_{t-1}) + \langle \nabla f(x_{t-1}), x^* - x_{t-1}\rangle \leq f(x^*) + \tfrac{l}{2}\|x_{t-1} - x^*\|^2$, and rearranging the terms, we obtain
\begin{align}\label{proof_stoch_3}
&f(x_t) - f(x^*) + (\tfrac{\gamma_t}{4} - \tfrac{L}{2})\|x_t - x_{t-1}\|^2 - \tfrac{\|\hat \delta_t(x_{t-1})\|^2}{\gamma_t}\nn\\
& \leq \tfrac{l}{2}\|x_{t-1} - x^*\|^2 + \tfrac{\gamma_t}{2}\|x_{t-1} - x^*\|^2 - \tfrac{\gamma_t}{2}\|x_t - x^*\|^2 + \langle \hat \delta_t(x_{t-1}), x^* - x_{t-1}\rangle.
\end{align}
Then, by taking the telescope sum of the above inequality for $t=1,...,k$, and using $\gamma_t = \gamma$, we arrive at
\begin{align}\label{proof_stoch_4}
&\tsum_{t=1}^k \left[f(x_t) - f(x^*) + (\tfrac{\gamma}{4} - \tfrac{L}{2})\|x_t - x_{t-1}\|^2 - \tfrac{\|\hat \delta_t(x_{t-1})\|^2}{\gamma} \right]\nn\\
& \leq \tsum_{t=1}^k\tfrac{l}{2}\|x_{t-1} - x^*\|^2 + \tfrac{\gamma}{2}\|x_0 - x^*\|^2 - \tfrac{\gamma}{2}\|x_k - x^*\|^2 + \tsum_{t=1}^k\langle \hat \delta_t(x_{t-1}), x^* - x_{t-1}\rangle\nn\\
& \leq \tfrac{l k}{2}D_X^2 + \tfrac{\gamma}{2}D_X^2 + \tsum_{t=1}^k\langle \hat \delta_t(x_{t-1}), x^* - x_{t-1}\rangle.
\end{align}
Next, we further lower bound the LHS of the above inequality, i.e.,
\begin{align}\label{proof_stoch_5}
&\tsum_{t=1}^k \left[f(x_t) - f(x^*) + (\tfrac{\gamma}{4} - \tfrac{L}{2})\|x_t - x_{t-1}\|^2 - \tfrac{\|\hat \delta_t(x_{t-1})\|^2}{\gamma} \right]\nn\\
& \geq \tsum_{t=1}^k  \left[f(x_t) - f(x_k) + (\tfrac{\gamma}{4} - \tfrac{L}{2})\|x_t - x_{t-1}\|^2 - \tfrac{\|\hat \delta_t(x_{t-1})\|^2}{\gamma} \right]\nn\\
& = \tsum_{t=1}^k  \left[\tsum_{j=t}^{k-1} [f(x_j) - f(x_{j+1})] + (\tfrac{\gamma}{4} - \tfrac{L}{2})\|x_t - x_{t-1}\|^2 - \tfrac{\|\hat \delta_t(x_{t-1})\|^2}{\gamma} \right]\nn\\
& \overset{(i)}\geq \tsum_{t=1}^k\left[\tsum_{j=t}^{k-1} [(\tfrac{3\gamma}{4} - \tfrac{L}{2})\|x_{j} - x_{j+1}\|^2 - \tfrac{\|\hat \delta_{j+1}(x_{j})\|^2}{\gamma}] + (\tfrac{\gamma}{4} - \tfrac{L}{2})\|x_t - x_{t-1}\|^2 - \tfrac{\|\hat \delta_t(x_{t-1})\|^2}{\gamma} \right]\nn\\
& =  \tsum_{t=1}^k \left[(\tfrac{(3t-2)\gamma}{4}-\tfrac{tL}{2})\|x_t - x_{t-1}\|^2 - \tfrac{t \|\hat \delta_t(x_{t-1})\|^2}{\gamma} \right]\nn\\
& \overset{(ii)}\geq\tsum_{t=2}^k \left[(\tfrac{3t-2}{4\gamma}-\tfrac{tL}{2\gamma^2})\|\hat g_{X, t-1}\|^2\right] - \tfrac{L}{4}\|x_1-x_0\|^2 - \tsum_{t=1}^k \tfrac{t \|\hat \delta_t(x_{t-1})\|^2}{\gamma} \nn\\
& \overset{(iii)}\geq \tsum_{t=2}^k \left[(\tfrac{3t-2}{8\gamma}-\tfrac{tL}{4\gamma^2})\|g_{X, t-1}\|^2 -(\tfrac{3t-2}{4\gamma}-\tfrac{tL}{2\gamma^2})\|\hat \delta_t(x_{t-1})\|^2 \right]- \tfrac{L}{4}\|x_1-x_0\|^2 - \tsum_{t=1}^k \tfrac{t \|\hat \delta_t(x_{t-1})\|^2}{\gamma}.
\end{align}
Here, step (i) follows from Ineq.~\eqref{proof_stoch_2}, step (ii) uses the definition of $\hat g_{X, t}$ and $ \gamma > L$, and step (iii) is due to Young's inequality and the fact that $\|g_{X, t-1} - \hat g_{X, t-1}\|\leq \|\hat \delta_t(x_{t-1})\|$ (see Proposition 6.1 of \cite{LanBook2020}). By combining Ineqs.~\eqref{proof_stoch_4} and \eqref{proof_stoch_5}, and rearranging the terms, we have
\begin{align*}
  &\tsum_{t=2}^k\left[(\tfrac{3t-2}{8\gamma}-\tfrac{tL}{4\gamma^2})\|g_{X, t-1}\|^2 \right]\\
 &\leq (\tfrac{\gamma}{2}+ \tfrac{L}{4} + \tfrac{lk}{2})D_X^2 +\tsum_{t=2}^k(\tfrac{7t-2}{4\gamma}-\tfrac{tL}{2\gamma^2})\|\hat \delta_t(x_{t-1})\|^2 + \tfrac{\|\hat \delta_1(x_{0})\|^2}{\gamma} + \tsum_{t=1}^k\langle \hat \delta_t(x_{t-1}), x^* - x_{t-1}\rangle.
\end{align*}
Taking the expectation on both sides of the above inequality, we obtain
\begin{align*}
  &\tsum_{t=2}^k\bbe\left[(\tfrac{3t-2}{8\gamma}-\tfrac{tL}{4\gamma^2})\|g_{X, t-1}\|^2 \right]\leq (\tfrac{\gamma}{2}+ \tfrac{L}{4} + \tfrac{lk}{2})D_X^2 +\tsum_{t=2}^k(\tfrac{7t-2}{4\gamma}-\tfrac{tL}{2\gamma^2})\tfrac{\sigma^2}{b_t} + \tfrac{\sigma^2}{\gamma b_1}.
\end{align*}
Then, due to the random output rule \eqref{random_output_0}, we have $\tsum_{t=2}^k\bbe\left[(\tfrac{3t-2}{8\gamma}-\tfrac{tL}{4\gamma^2})\|g_{X, t-1}\|^2 \right] = \tsum_{t=2}^k(\tfrac{3t-2}{8\gamma}-\tfrac{tL}{4\gamma^2})\bbe[\|g_{X, R(k)}\|^2]$, and thus complete the proof.
\end{proof}
\vgap

A few remarks about Theorem~\ref{thm_2_0} are in place. \revision{Firstly, Theorem~\ref{thm_2_0} assumes the knowledge of the problem parameters $L$, $l$, $D_X^2$, and $\sigma^2$ to properly choose stepsizes and batch sizes.}{Firstly, Theorem~\ref{thm_2_0} assumes the knowledge of the problem parameter $L$ to properly choose stepsizes. To achieve the unified convergence rate in \eqref{SPG_unfied_rate}, we also require the knowledge of $l$, $D_X^2$, and $\sigma^2$ to choose the proper batch sizes.} Secondly, to find an $\epsilon$-approximate stationary point $x_R \in X$ such that  $\bbe\|g_{X, R}\|^2\leq \epsilon^2$, the number of SPG iterations is bounded by 
\begin{align*}
\mathcal{O}\left(\tfrac{L D_X}{\epsilon} + \tfrac{L l D_X^2}{\epsilon^2} \right),
\end{align*}
which matches the iteration complexity \eqref{complexity_thm_0} of the PG method in the deterministic setting. Thirdly, considering the batch size in each iteration, the total number of oracle calls $\tsum_{t=1}^k b_t$ (or the sample complexity) can be bounded by 
\begin{align}\label{complexity_bound_spg}
\mathcal{O}\left(\tfrac{L D_X}{\epsilon} + \tfrac{L l D_X^2}{\epsilon^2} + \tfrac{LD_X \sigma^2}{\epsilon^3} + \tfrac{Ll D_X^2\sigma^2}{\epsilon^4} \right).
\end{align}
\rev{To the best of our knowledge, this complexity bound is new for the SPG method in the literature.}{It is worth comparing this complexity bound with closely related literature. Specifically, Ghadimi and Lan \cite{ghadimi2016accelerated} analyze a stochastic accelerated gradient method and establish a sample complexity bound of
\begin{align}\label{complexity_bound_acsa}
\mathcal{O}\left( (\tfrac{LD_X}{\epsilon})^{\frac{2}{3}} + \tfrac{LlD_X^2}{\epsilon^2} + \tfrac{L^{1/3}\sigma^2}{l D_X^{2/3}\epsilon^{4/3}} + \tfrac{LlD_X^2 \sigma^2}{\epsilon^4} \right).
\end{align}
While the dominating term has the same order as \eqref{complexity_bound_spg}, this is not a unified complexity bound, as the third term blows up when $l\rightarrow 0$. However, the first term in \eqref{complexity_bound_acsa} is sharper than the first term in \eqref{complexity_bound_spg} due to the use of acceleration techniques.
}

\subsection{Auto-conditioned stochastic projected gradient (AC-SPG) method}\label{sec_stoch_expectation}
In this subsection, we develop a parameter-free variant of the SPG method. To achieve this goal, we require a stronger assumption on the stochastic oracle such that it can output both the stochastic function value $\cF(x, \xi)$ and the stochastic gradient $\cG(x, \xi) = \nabla \cF(x,\xi)$, where $\cF(x, \xi)$ is smooth uniformly for all $\xi$.

\begin{assumption}[Smoothness of the stochastic function]  \label{assump_smooth_stoch}
Assume $\bbe_\xi[\cF(x, \xi)] = f(x)$, $\bbe_\xi[\cG(x, \xi)] = \nabla f(x)$, and there exists an $\hat \cL \geq 0$ such that for any $x, y \in X$, $\xi \in \Xi$,
\begin{align}\label{smoothness_of_stohcastic_gradient_0}
\|\cG(x, \xi) - \cG(y, \xi)\| \leq \hat \cL\|x-y\|,
\end{align}
and
\begin{align}\label{smoothness_of_stohcastic_gradient}
\cF(x, \xi) - \cF(y, \xi) - \langle \cG(y, \xi), x-y\rangle \leq \tfrac{\hat \cL}{2}\|x-y\|^2,
\end{align}
almost surely.
\end{assumption}

\revision{}{For notational simplicity, for any $x, y \in X$, we define
\begin{align}\label{define_L_cL}
\cL_{x, y} (\xi)&:= \tfrac{2(\cF(x, \xi) - \cF(y, \xi) - \langle \cG(y, \xi), x-y \rangle )}{\|x-y\|^2},~\text{and}~
L_{x, y} := \tfrac{2(f(x) - f(y) - \langle \nabla f(y), x-y \rangle  )}{\|x-y\|^2}.
\end{align}
Thus we have $\cL_{x, y}(\xi) \leq \hat \cL$ and $L_{x, y} \leq L$. 
}

\begin{algorithm}[H]\caption{Auto-conditioned stochastic projected gradient (AC-SPG) method}\label{alg_2}
	\begin{algorithmic}
		\State{\textbf{Input}: initial point $x_0 \in X$, $\bar \cL_0$ such that $0< \bar\cL_0 = \hat \cL_0 \leq \hat\cL$, and $\Gamma_0 = 0$. 
  }
		\For{$t=1,\cdots, k$}
        \State{Set $\gamma_t = 12\hat \cL_{t-1}$.
        }
		\State{Generate i.i.d. observations $\xi_1^t, ..., \xi_{b_t}^t$ and define 
        \begin{align}
            \hat G_t(x) &:= \tfrac{1}{b_t}\tsum_{j=1}^{b_t} \cG(x, \xi_j^t).\label{def_hat_g_2}
        \end{align}}
        \State{Compute
  \begin{align}
  x_t = \arg \min_{x\in X} \left\{\langle  \hat G_t(x_{t-1}) , x \rangle + \tfrac{\gamma_t}{2}\|x_{t-1} - x\|^2 \right\}, \label{prox-mapping_2}
  \end{align}}
  \State{
  Collect i.i.d. observations $\bar \xi_1^t, ..., \bar \xi_{b_t}^t$ and compute 
  \begin{align}
  \bar\cL_t := \tfrac{1}{b_t}\tsum_{j=1}^{b_t} \cL_{x_t, x_{t-1}}(\bar \xi_j^t), \quad \hat \cL_{t}:= \max\{ \bar \cL_0, ..., \bar \cL_{t}\}.\label{evaluate_L_t_2}
  \end{align}
  }
  \State{If $\bar \cL_t > \tfrac{3}{2}\hat\cL_{t-1}$, set $\Gamma_t = \gamma_t$; otherwise set $\Gamma_t = \Gamma_{t-1} + \gamma_t$.}
\EndFor
\State{\textbf{Output}: we let the output at any time step $k$, denoted as $x_{R(k)}$, be randomly selected from $\{x_0, ..., x_{k-1}\}$ according to the following probability mass function:
\begin{align}\label{random_output}
\mathbb{P}(x_{R(k)} = x_{t-1}|\cF_k) = \left(\tsum_{j=1}^k \tfrac{\Gamma_j}{\gamma_j}\right)^{-1} \cdot \tfrac{\Gamma_t}{\gamma_t}.
\end{align}}
	\end{algorithmic}
\end{algorithm}

The AC-SPG method 
is presented in Algorithm~\ref{alg_2}. In particular, in each iteration of AC-SPG, we collect two sample batches with same sizes using the stochastic oracle, i.e., $\{\xi_1^{t}, ..., \xi_{b_t}^t\}$ and $\{\bar\xi_{1}^{t}, ..., \bar\xi_{b_t}^t\}$. The first batch is used to calculate the operator $\hat G_t(\cdot)$ to perform the prox-mapping update in \eqref{prox-mapping_2}, while the second batch is employed to compute the estimator $\bar \cL_t$ of the local smoothness constant. For 
$t\in \mathbb{Z}$, 
the filtration is defined as
\begin{align*}
\cF_t&:= \sigma(\xi_1^1, ..., \xi_{b_1}^1, \bar \xi_1^1, ..., \bar\xi_{b_1}^1, 
......, 
\xi_1^t, ..., \xi_{b_t}^t, \bar \xi_1^t, ..., \bar\xi_{b_t}^t),\\
\cF_{t+1/2}&:= \sigma(\xi_1^1, ..., \xi_{b_1}^1, \bar \xi_1^1, ..., \bar\xi_{b_1}^1, 
......., 
\xi_1^t, ..., \xi_{b_t}^t, \bar \xi_1^t, ..., \bar\xi_{b_t}^t, \xi_1^{t+1}, ..., \xi_{b_{t+1}}^{t+1}).
\end{align*}
Then clearly, $x_t$ is $\cF_{t-1/2}$-measurable, $\bar \cL_t$ and $\Gamma_t$ are $\cF_t$-measurable, and $\gamma_t$ is $\cF_{t-1}$-measurable. Meanwhile, the batch size $b_t$ should be set to be $\cF_{t-1}$-measurable to avoid any search routines. 

Similar to the deterministic setting, we divide the iterates $\{x_0, x_1, x_2, ..., x_k\}$ into segments based on the growth of the local smoothness estimator $\hat \cL_{t}$. Specifically, we start a new segment whenever $\bar \cL_t > \tfrac{3}{2} \hat \cL_{t-1}$, and this creates a new indexing rule
\begin{align}\label{new_indexing_rule}
 x_t = x^{(S(t))}_{I(t)},
\end{align}
where $S(t)$ is the segment index of the search point, and $I(t)$ is the index of the search point within the current segment. Then we define the quantities $\bar \cL_{I(t)}^{(S(t))}$, $\hat \cL_{I(t)}^{(S(t))}$, and $\Gamma_{I(t)}^{(S(t))}$ in the same way as $\bar \cL_{t}$, $\hat \cL_{t}$, and $\Gamma_t$,
just with the new indexing rule. Equivalently, we have
\begin{align}\label{def_seg_2}
\{x_0, x_1, x_2, ..., x_k\} = \big\{\underbrace{x_0^{(1)}, x_1^{(1)}, ..., x_{s_1}^{(1)}}_{\text{segment 1}}, \underbrace{x_1^{(2)}, ..., x_{s_2}^{(2)}}_{\text{segment 2}}, ..., \underbrace{x_1^{(m)}, ..., x_{s_m}^{(m)}}_{\text{segment $m$}}\big\},
\end{align}
where $s_i$ is the length of epoch $i$ and
\begin{align}
\bar \cL_{1}^{(i)} &> \tfrac{3}{2} \hat \cL_{s_{i-1}}^{(i-1)}, ~\text{for }i \geq 2 \label{eq_5_0_2}\\
\bar \cL_t^{(i)} &\leq \tfrac{3}{2} \hat \cL_{t-1}^{(i)}, ~\text{for } 2 \leq t \leq s_i, ~i\geq 1.\label{eq_5_2}
\end{align}
For notational simplicity, we denote $x_0^{(i)} = x_{s_{i-1}}^{(i-1)}$ and $\Gamma_0^{(i)} = 0$ for $i\geq 2$. Moreover, since $\bar \cL_t \leq \hat \cL$ for all $t \geq 0$ due to Assumption~\ref{assump_smooth_stoch}, we have that the total number of segments $m$ satisfy
\begin{align}\label{eq_6_2}
    m \leq \lfloor\log_{\frac{3}{2}}\tfrac{\hat \cL}{\bar \cL_0}\rfloor + 1.
\end{align}

\begin{remark}
We apply the random output rule as in \eqref{random_output} since we cannot output the iterate with the smallest projected gradient as in the deterministic setting. 
It is worth noting that it is not necessary to store all the search points $\{x_1, ..., x_{k-1}\}$ to generate the output $x_{R(k)}$. Instead, $x_{R(k)}$ can be determined inductively using the following rule
\begin{align*}
\mathcal{W}_k &= \mathcal{W}_{k-1} + \tfrac{\Gamma_k}{\gamma_k},\\
\mathbb{P}(x_{R(k)} = x_{t}|\cF_k) &= \begin{cases}
\tfrac{\mathcal{W}_{k-1}}{\mathcal{W}_k}, & \text{for $t = R(k-1)$}\\
\tfrac{\Gamma_k}{\gamma_k \mathcal{W}_k}, & \text{for $t = k-1$},
\end{cases}
\end{align*}
starting from $\mathcal{W}_0 = 0$ and $x_{R(0)} = x_0$.
\end{remark}

Before presenting the convergence guarantees of AC-SPG (Algorithm~\ref{alg_2}), we first provide two useful lemmas. The first one characterizes the summation of weights in the output rule of AC-SPG.
\begin{lemma}\label{lemma_bound_M_k}
Define $M_k := \tsum_{t=1}^k \tfrac{\Gamma_t}{\gamma_t}$. Then we have 
\begin{align}\label{eq:lower_bound_M_k}
M_k \geq \tfrac{k^2}{3\left(1+ 2 \lfloor \log_{\frac{3}{2}} \frac{\hat \cL}{\bar \cL_0}\rfloor\right)}.
\end{align}
Moreover, 
\begin{align}\label{eq:bound_gamma_t_Gamma_t}
	\tsum_{t=1}^k\tfrac{\gamma_t}{\Gamma_t} \leq (1+\lfloor \log_{\frac{3}{2}}\tfrac{\hat \cL}{\bar \cL_0}\rfloor)(1+\log k)+ \log \tfrac{\hat \cL}{\bar \cL_0}.
\end{align}
\end{lemma}
\begin{proof}
First, the segmentation rule for stepsizes in \eqref{eq_5_0_2}-\eqref{eq_5_2} has at most $\lfloor\log_{\frac{3}{2}}\tfrac{\hat \cL}{\bar \cL_0}\rfloor + 1$ segments. We refine these segments by taking intersections with the geometric bins for the stepsizes:
\begin{align*}
\left[12 \bar \cL_0 (\tfrac{3}{2})^j, 12\bar \cL_0 (\tfrac{3}{2})^{j+1}\right), ~j=0,1,...
\end{align*}
There are also at most $\lfloor\log_{\frac{3}{2}}\tfrac{\hat \cL}{\bar \cL_0}\rfloor + 1$ such bins. The refined segmentation after taking intersections has at most $2\lfloor\log_{\frac{3}{2}}\tfrac{\hat \cL}{\bar \cL_0}\rfloor + 1$ nonempty groups. Then, within each group, $\Gamma_t$ keeps growing, and $\gamma_t \leq \tfrac{3}{2}\gamma_{\text{start}}$, where $\gamma_{\text{start}}$ is the first stepsize within the group. For an iteration $t$, we let $r$ denote its index in the group, then $\Gamma_t/\gamma_t \geq 2\Gamma_t/(3\gamma_{\text{start}}) \geq \tfrac{2r}{3}$. Let us denote the group sizes by $n_j$. We obtain
\begin{align*}
M_k \geq \tfrac{2}{3}\tsum_{\text{all groups $j$}}\tfrac{n_j(n_j+1)}{2} \geq \tfrac{1}{3}\tsum_{\text{all groups $j$}} n_j^2 \geq \tfrac{k^2}{3\left(1+ 2 \lfloor \log_{\frac{3}{2}} \frac{\hat \cL}{\bar \cL_0}\rfloor\right)},
\end{align*}
where the last step follows from Cauchy-Schwarz inequality and the upper bound $2\lfloor\log_{\frac{3}{2}}\tfrac{\hat \cL}{\bar \cL_0}\rfloor + 1$ for the number of groups.

Next, we prove Ineq.~\eqref{eq:bound_gamma_t_Gamma_t}. Let us first consider a single segment $i$ with length $n_i$ such that $\hat \cL_t \leq \tfrac{3}{2}\hat \cL_{t-1}$. Then the first iteration within the segment satisfies $\tfrac{\gamma^{(i)}_{t}}{\Gamma^{(i)}_{t}}= \tfrac{\gamma^{(i)}_{t}}{\gamma^{(i)}_{t}}=1$. For the subsequent iterations, we have
\begin{align*}
	\tfrac{\gamma^{(i)}_{t}}{\Gamma^{(i)}_{t}} = \tfrac{\Gamma^{(i)}_{t} - \Gamma^{(i)}_{t-1}}{\Gamma^{(i)}_{t}} \leq \log \Gamma^{(i)}_{t} - \log \Gamma^{(i)}_{t-1},
\end{align*}
where the last inequality follows from $1-\tfrac{1}{x} \leq \log x$. Then we have
\begin{align*}
	\tsum_{t=1}^{n_i} \tfrac{\gamma^{(i)}_{t}}{\Gamma^{(i)}_{t}} \leq 1 + \tsum_{t=2}^{n_i} (\log \Gamma^{(i)}_{t} - \log \Gamma^{(i)}_{t-1}) = 1+ \log \tfrac{\Gamma^{(i)}_{n_i} }{\Gamma^{(i)}_{1} } \leq 1+ \log n_i + \log \tfrac{\gamma^{(i)}_{n_i}}{\gamma^{(i)}_{1} },
\end{align*}
where the last inequality follows from $\Gamma^{(i)}_{n_i} \leq n_i \gamma^{(i)}_{n_i}$ and $\Gamma^{(i)}_{1} = \gamma^{(i)}_{1} $. Invoking that the total number of segments is bounded by $m\leq 1+\lfloor \log_{\frac{3}{2}}\tfrac{\hat \cL}{\bar \cL_0}\rfloor$, we have
\begin{align*}
	\tsum_{t=1}^{k} \tfrac{\gamma_{t}}{\Gamma_{t}}  &= \tsum_{i=1}^m \tsum_{t=1}^{n_i} \tfrac{\gamma^{(i)}_{t}}{\Gamma^{(i)}_{t}} \\
	&\leq m + \tsum_{i=1}^m \log n_i + \tsum_{i=1}^m \log \tfrac{\gamma^{(i)}_{n_i}}{\gamma^{(i)}_{1} }\\
	& \leq m + m\log (\tfrac{1}{m}\tsum_{i=1}^m n_i)  + \tsum_{i=1}^m \log \tfrac{\gamma^{(i)}_{n_i}}{\gamma^{(i)}_{1} }\\
	& \leq m + m\log \tfrac{k}{m} + \log \tfrac{\gamma_{k}}{\gamma_{1} }\\
	& \leq (1+\lfloor \log_{\frac{3}{2}}\tfrac{\hat \cL}{\bar \cL_0}\rfloor)(1+\log k)+ \log \tfrac{\hat \cL}{\bar \cL_0},
\end{align*}
which completes the proof.
\end{proof}

\vgap

Recall that $\bar \cL_t$  in \eqref{evaluate_L_t_2} is an averaged estimator constructed with i.i.d. observations; thus, it is also an unbiased estimator of $L_t$, i.e.,
\begin{align}\label{define_L_t_stochastic}
\bbe[\bar\cL_t|\cF_{t-1/2}] = L_t := \tfrac{2(f(x_{t}) - f(x_{t-1}) - \langle \nabla f(x_{t-1}), x_t - x_{t-1}\rangle)}{\|x_t - x_{t-1}\|^2},
\end{align}
under Assumption~\ref{assump_smooth_stoch}. 
The next lemma characterizes the variance of $\bar \cL_t$ around $L_t$.

\begin{lemma}\label{lemma_L_t_stochastic_new}
Suppose Assumptions \ref{assump_unbiasness}-\ref{assump_smooth_stoch} hold. We have for any $x, y \in X$,
\begin{align}\label{new_variance_bound_x_y}
\bbe[(\bar \cL_{x, y} - L_{x, y})^2]\cdot\|x - y\|^2 \leq 16 \sigma^2.
\end{align}
and consequently,
\begin{align}\label{new_stochastic_L_t_variance_bound}
\bbe[(\bar \cL_t - L_t)^2|\cF_{t-1/2}]\cdot\|x_t - x_{t-1}\|^2 \leq \tfrac{16 \sigma^2}{b_t}.
\end{align}
\end{lemma}
\begin{proof}
First, we have 
    \begin{align*}
    \cF(x, \xi)&= \cF(y, \xi) + \int_{0}^1 \langle \cG(y + s (x-y), \xi), x-y\rangle ds,\\
    f(x) &= f(y) + \int_{0}^1 \langle \nabla f(y + s (x-y)), x-y\rangle ds.
    \end{align*}
    Let $\delta(x, \xi) = \cG(x, \xi) - \nabla f(x)$. Then, recalling the definitions in \eqref{define_L_cL}, we have
    \begin{align*}
    \tfrac{1}{2}(\cL_{x,y}(\xi) - L_{x, y})\|x-y\|^2  = \int_{0}^1 \langle \delta(y + s (x-y), \xi) - \delta(y , \xi), x-y\rangle ds.
    \end{align*}
    Therefore, by using Jensen's inequality and Cauchy-Schwarz inequality, we have 
    \begin{align*}
    \bbe\left[\left(\tfrac{1}{2}(\cL_{x,y}(\xi) - L_{x, y})\|x-y\|^2\right)^2\right]  &\leq \int_{0}^1 \bbe[(\langle \delta(y + s (x-y), \xi) - \delta(y , \xi), x-y\rangle)^2] ds\\
    & \leq \|x -y\|^2\int_{0}^1 \bbe[\|\delta(y + s (x-y), \xi) - \delta(y , \xi)\|^2] ds\\
    & \leq 4\sigma^2\|x-y\|^2,
    \end{align*}
    which proves Ineq.~\eqref{new_variance_bound_x_y}. Ineq.~\eqref{new_stochastic_L_t_variance_bound} follows directly from the independence of the batch samples. 
\end{proof}
\vgap

In the next proposition, we characterize the main recursion of the AC-SPG update.
\begin{proposition}\label{prop:ac-spg_main_recursion}
Let $\{x_t\}$ be generated by Algorithm~\ref{alg_2}, and define $\hat \delta_t(x_{t-1}) := \hat G_t(x_{t-1}) - \nabla f(x_{t-1})$. We have for any $\rho \in [0, \tfrac{1}{8}]$,
\begin{align}\label{ac-spg_main_recursion}
&\tsum_{t=1}^{k}\tfrac{\rho\Gamma_t}{2\gamma_t}\|g_{X, t-1}\|^2 + \tsum_{t=1}^k (\tfrac{1}{8} - \rho)\Gamma_t\gamma_t\|x_{t-1} - x_t\|^2\nn\\
& \leq 94 \hat \cL_{k}^2D_X^2 + 6kl\hat\cL_{k}D_X^2 + \tsum_{t=1}^{k} \tfrac{(1+\rho)\Gamma_t}{\gamma_t}\| \hat \delta_t(x_{t-1})\|^2\nn\\
&\quad+ \tsum_{t=1}^{k}  \gamma_t\langle \hat \delta_t(x_{t-1}),  x^* - x_{t-1}\rangle + \tsum_{t=1}^{k} \Gamma_{t}(\tfrac{L_{t}}{2}- \hat \cL_{t})\|x_{t-1} - x_{t}\|^2. 
\end{align}
Note that this proposition does not rely on  the construction of the stochastic gradient operator $\hat G_t(x_{t-1})$.
\end{proposition}
\begin{proof}
By the optimality condition of \eqref{prox-mapping_2}, we have the three-point inequality, i.e., for any $x \in X$,
\begin{align}\label{three-point_lemma}
\langle \hat G_t(x_{t-1}), x_t - x \rangle + \tfrac{\gamma_t}{2}\|x_{t-1} - x_t\|^2 + \tfrac{\gamma_t}{2}\|x_t - x\|^2 \leq \tfrac{\gamma_t}{2}\|x_{t-1} - x\|^2.
\end{align}
Adding $f(x_{t-1}) - \hat \cL_{t}\|x_{t-1}-x_t\|^2$ on both sides of the above inequality and rearranging the terms, we have
\begin{align*}
&f(x_{t-1}) + \langle \nabla f(x_{t-1}) , x_t - x_{t-1}\rangle + \langle \hat \delta_t(x_{t-1}), x_t - x\rangle +( \tfrac{\gamma_t}{2}-\hat \cL_{t})\|x_{t-1} - x_t\|^2 + \tfrac{\gamma_t}{2}\|x_t - x\|^2\nn\\
&\leq f(x_{t-1}) +\langle \nabla f(x_{t-1}), x - x_{t-1}\rangle + \tfrac{\gamma_t}{2}\|x_{t-1} - x\|^2 - \hat \cL_{t}\|x_{t-1} - x_t\|^2.
\end{align*}
Then, using the definition of $L_t$ in \eqref{define_L_t_stochastic} and rearranging the terms, we obtain
\begin{align}\label{eq_1_2}
&f(x_t) + (\tfrac{\gamma_t}{2} - \hat \cL_{t})\|x_{t-1} - x_t\|^2 + \langle \hat \delta_t(x_{t-1}), x_t - x\rangle + \tfrac{\gamma_t}{2}\|x_t - x\|^2\nn\\
&\leq f(x_{t-1}) +\langle \nabla f(x_{t-1}) , x - x_{t-1}\rangle + \tfrac{\gamma_t}{2}\|x_{t-1} - x\|^2 + (\tfrac{L_t}{2}- \hat \cL_{t})\|x_{t-1} - x_t\|^2.
\end{align}
Taking $x = x_{t-1}$ in Ineq.~\eqref{eq_1_2} yields
\begin{align}\label{eq_2_2}
 (\gamma_t - \hat \cL_{t})\|x_{t-1} - x_t\|^2 + \langle \hat \delta_t(x_{t-1}), x_t - x_{t-1}\rangle  - (\tfrac{L_t}{2}- \hat \cL_{t})\|x_{t-1} - x_t\|^2\leq f(x_{t-1}) - f(x_t).
\end{align}
On the other hand, setting $x = x^*$ and multiplying $\gamma_t$ on both sides of Ineq.~\eqref{eq_1_2},  and using the weakly convex property \eqref{weakly_convex}, we obtain
\begin{align}\label{eq_3_2}
&\gamma_t(f(x_t) - f(x^*)) + \gamma_t(\tfrac{\gamma_t}{2} - \hat \cL_{t})\|x_{t-1} - x_t\|^2 + \gamma_t\langle \hat \delta_t(x_{t-1}), x_t - x^*\rangle + \tfrac{\gamma_t^2}{2}\|x_t - x^*\|^2\nn\\
&\leq  \tfrac{\gamma_t l}{2}\|x_{t-1} - x^*\|^2  + \tfrac{\gamma_t^2}{2}\|x_{t-1} - x^*\|^2 + \gamma_t(\tfrac{L_t}{2}- \hat \cL_{t})\|x_{t-1} - x_t\|^2.
\end{align}
By taking a telescope sum of the above inequality from $t=1$ to $k$, we have
\begin{align}\label{eq_4_2}
&\tsum_{t=1}^k \left[\gamma_t(f(x_t) -f(x^*)) + \gamma_t(\tfrac{\gamma_t}{2} -\hat \cL_{t})\|x_{t-1} - x_t\|^2 + \gamma_t\langle \hat \delta_t(x_{t-1}), x_t - x^*\rangle \right]\nn\\
& \leq \tfrac{\gamma_1^2}{2}\|x_0 - x^*\|^2 + \tsum_{t=1}^{k-1} (\tfrac{\gamma_{t+1}^2}{2} - \tfrac{\gamma_t^2}{2})\|x_t - x^*\|^2 - \tfrac{\gamma_k^2}{2}\|x_k -x^*\|^2\nn\\
&\quad + \tsum_{t=1}^k \tfrac{\gamma_tl}{2}\|x_{t-1}-x^*\|^2 + \tsum_{t=1}^k \gamma_t(\tfrac{L_t}{2}- \hat\cL_{t})\|x_{t-1} - x_t\|^2\nn\\
& \overset{(i)}\leq \tfrac{\gamma_k^2}{2}D_X^2 + \tfrac{\gamma_k k l}{2}D_X^2 + \tsum_{t=1}^k \gamma_t(\tfrac{L_t}{2}- \hat\cL_{t})\|x_{t-1} - x_t\|^2.
\end{align}
where step (i) follows from $\gamma_{t+1} \geq \gamma_t$ 
and the definition $D_X:= \max_{x, y\in X}\|x-y\|$.
Next, we lower bound the LHS of Ineq.~\eqref{eq_4_2}. Recalling the segmentation defined in \eqref{def_seg_2}, we define $$\cS_i:= \tsum_{t=1}^{s_i}\left[\gamma_t^{(i)}\big(f(x^{(i)}_t) -f(x^*)\big) + \gamma^{(i)}_t(\tfrac{\gamma_t^{(i)}}{2} - \hat \cL^{(i)}_{t})\|x^{(i)}_{t-1} - x^{(i)}_t\|^2 + \gamma_t^{(i)}\langle\hat\delta_t^{(i)}(x^{(i)}_{t-1}), x_t^{(i)} - x^*\rangle \right],$$
where we use the new indexing rule as in \eqref{new_indexing_rule} to define $\gamma_{t}^{(i)}$ as the stepsize at $t$-th iteration of the $i$-th segment.
Then we have for any $i \in [m]$,
\begin{align}\label{eq_7_0_2}
\cS_i &\overset{(i)}\geq \tsum_{t=1}^{s_i}\left[\gamma_t^{(i)}\big(f(x^{(i)}_t) -f(x_{s_i}^{(i)})\big) + \gamma^{(i)}_t(\tfrac{\gamma^{(i)}_t}{2} -  \hat\cL^{(i)}_{t})\|x^{(i)}_{t-1} - x^{(i)}_t\|^2  + \gamma^{(i)}_t\langle \hat \delta_t^{(i)}(x^{(i)}_{t-1}), x_t^{(i)} - x^*\rangle\right]\nn\\
& = \tsum_{t=1}^{s_i} \left[\gamma_t^{(i)}\cdot\tsum_{j=t}^{s_i-1}\big(f(x^{(i)}_j) - f(x^{(i)}_{j+1}) \big)  + \gamma^{(i)}_t(\tfrac{\gamma^{(i)}_t}{2} -  \hat\cL^{(i)}_{t})\|x^{(i)}_{t-1} - x^{(i)}_t\|^2  + \gamma^{(i)}_t\langle \hat \delta_t^{(i)}(x^{(i)}_{t-1}), x_t^{(i)} - x^*\rangle\right]\nn\\
& \overset{(ii)}= \tsum_{t=1}^{s_i} \left[\Gamma_{t-1}^{(i)}\big(f(x^{(i)}_{t-1}) - f(x^{(i)}_{t}) \big)  + \gamma^{(i)}_t(\tfrac{\gamma^{(i)}_t}{2} -  \hat\cL^{(i)}_{t})\|x^{(i)}_{t-1} - x^{(i)}_t\|^2  + \gamma^{(i)}_t\langle \hat \delta_t^{(i)}(x^{(i)}_{t-1}), x_t^{(i)} - x^*\rangle\right]\nn\\
& \overset{(iii)}\geq \tsum_{t=1}^{s_i} \left[\Gamma_{t-1}^{(i)}(\gamma_t^{(i)} - \hat \cL_t^{(i)})\|x^{(i)}_{t-1} - x^{(i)}_t\|^2 +   \Gamma_{t-1}^{(i)}(\hat \cL_t^{(i)} - \tfrac{L_t^{(i)}}{2})\|x^{(i)}_{t-1} - x^{(i)}_t\|^2
+ \Gamma_t^{(i)} \langle \hat \delta_t(x^{(i)}_{t-1}), x^{(i)}_t - x^{(i)}_{t-1}\rangle\right.\nn\\
&\qquad\qquad
\left.+ \gamma^{(i)}_t(\tfrac{\gamma^{(i)}_t}{2} -  \hat\cL^{(i)}_{t})\|x^{(i)}_{t-1} - x^{(i)}_t\|^2 + \gamma^{(i)}_t\langle \hat \delta_t^{(i)}(x^{(i)}_{t-1}), x_{t-1}^{(i)} - x^*\rangle\right]\nn\\
& \overset{(iv)}\geq \tsum_{t=1}^{s_i} \left[\Gamma_{t-1}^{(i)}(\gamma_t^{(i)} - \hat \cL_t^{(i)})\|x^{(i)}_{t-1} - x^{(i)}_t\|^2 +   \Gamma_{t-1}^{(i)}(\hat \cL_t^{(i)} - \tfrac{L_t^{(i)}}{2})\|x^{(i)}_{t-1} - x^{(i)}_t\|^2
- \tfrac{\Gamma_t^{(i)}\gamma_t^{(i)}}{4}\| x^{(i)}_t - x^{(i)}_{t-1}\|^2\right.\nn\\
&\qquad\qquad
\left.-\tfrac{\Gamma_t^{(i)}}{\gamma_t^{(i)}}\|\hat \delta_t(x^{(i)}_{t-1})\|^2+ \gamma^{(i)}_t(\tfrac{\gamma^{(i)}_t}{2} -  \hat\cL^{(i)}_{t})\|x^{(i)}_{t-1} - x^{(i)}_t\|^2 + \gamma^{(i)}_t\langle \hat \delta_t^{(i)}(x^{(i)}_{t-1}), x_{t-1}^{(i)} - x^*\rangle\right]\nn\\
& \overset{(v)}=\tsum_{t=1}^{s_i} \left[\left(\tfrac{3}{4}\Gamma_{t}^{(i)}\gamma_t^{(i)} - \tfrac{1}{2}(\gamma_t^{(i)})^2 - \Gamma_{t}^{(i)} \hat \cL_{t}^{(i)}\right)\|x^{(i)}_{t-1} - x^{(i)}_t\|^2  +   \Gamma_{t-1}^{(i)}(\hat \cL_t^{(i)} - \tfrac{L_t^{(i)}}{2})\|x^{(i)}_{t-1} - x^{(i)}_t\|^2
\right.\nn\\
&\qquad\qquad
\left. -\tfrac{\Gamma_t^{(i)}}{\gamma_t^{(i)}}\|\hat \delta_t(x^{(i)}_{t-1})\|^2+ \gamma^{(i)}_t\langle \hat \delta_t^{(i)}(x^{(i)}_{t-1}), x_{t-1}^{(i)} - x^*\rangle\right]\nn\\
& \overset{(vi)}\geq\underbrace{\tsum_{t=1}^{s_i} \tfrac{1}{8}\Gamma_t^{(i)}\gamma_t^{(i)}\|x^{(i)}_{t-1} - x^{(i)}_t\|^2}_{\cS_{i,1}} - \gamma_1^{(i)}\hat \cL_{1}^{(i)}\|x^{(i)}_{0} - x^{(i)}_1\|^2   +   \tsum_{t=1}^{s_i}\Gamma_{t-1}^{(i)}(\hat \cL_t^{(i)} - \tfrac{L_t^{(i)}}{2})\|x^{(i)}_{t-1} - x^{(i)}_t\|^2
\nn\\
&\qquad\qquad -\tsum_{t=1}^{s_i}\tfrac{\Gamma_t^{(i)}}{\gamma_t^{(i)}}\|\hat \delta_t(x^{(i)}_{t-1})\|^2
 + \tsum_{t=1}^{s_i}\gamma^{(i)}_t\langle \hat \delta_t^{(i)}(x^{(i)}_{t-1}), x_{t-1}^{(i)} - x^*\rangle,
\end{align}
where step (i) follows from $f(x_{s_i}^{(i)}) \geq f(x^*)$, step (ii) follows from counting the $\gamma_t$ terms and the definition of $\Gamma_t$, step (iii) follows from Ineq.~\eqref{eq_2_2} and rearranging the terms, step (iv) stems from Young's inequality, step (v) follows from rearranging the terms, and step (vi) follows from $\gamma_t^{(i)} = 12 \hat \cL_{t-1}^{(i)}$, indicating that for $t \geq 2$,
\begin{align*}
\tfrac{3}{4}\Gamma_{t}^{(i)}\gamma_t^{(i)} - \tfrac{1}{2}(\gamma_t^{(i)})^2 - \Gamma_{t}^{(i)} \hat \cL_{t}^{(i)} \geq (\tfrac{3}{4} - \tfrac{1}{2}) \Gamma_{t}^{(i)}\gamma_t^{(i)}  - \tfrac{3}{2}\Gamma_{t}^{(i)} \hat \cL_{t-1}^{(i)} =  (\tfrac{3}{4} - \tfrac{1}{2} - \tfrac{1}{8}) \Gamma_{t}^{(i)}\gamma_t^{(i)} = \tfrac{1}{8} \Gamma_{t}^{(i)}\gamma_t^{(i)},
\end{align*}
and for $t=1$, $\tfrac{3}{4}\Gamma_{1}^{(i)}\gamma_1^{(i)} - \tfrac{1}{2}(\gamma_1^{(i)})^2 - \Gamma_{1}^{(i)} \hat \cL_{1}^{(i)} = \tfrac{1}{4}(\gamma_1^{(i)})^2 - \Gamma_{1}^{(i)} \hat \cL_{1}^{(i)} \geq \tfrac{1}{8}(\gamma_1^{(i)})^2 - \gamma_{1}^{(i)} \hat \cL_{1}^{(i)}$.
Then, for any $\rho \in [0, \tfrac{1}{8}]$, we can further bound the term $\cS_{i,1}$ by
\begin{align}\label{eq_7_2}
\cS_{i, 1} 
& \overset{(i)}=   \tsum_{t=1}^{s_i}\tfrac{\rho\Gamma_t^{(i)}}{\gamma_t^{(i)}}\|\hat g_{X, t-1}^{(i)}\|^2 + \tsum_{t=1}^{s_i} (\tfrac{1}{8} - \rho)\Gamma_t^{(i)}\gamma_t^{(i)}\|x^{(i)}_{t-1} - x^{(i)}_t\|^2 \nn\\
& \overset{(ii)}\geq  \tsum_{t=1}^{s_i}\tfrac{\rho\Gamma_t^{(i)}}{2\gamma_t^{(i)}}\| g_{X, t-1}^{(i)}\|^2 - \tsum_{t=1}^{s_i}\tfrac{\rho\Gamma_t^{(i)}}{\gamma_t^{(i)}}\| \hat g_{X, t-1}^{(i)} - g_{X, t-1}^{(i)}\|^2 + \tsum_{t=1}^{s_i} (\tfrac{1}{8} - \rho)\Gamma_t^{(i)}\gamma_t^{(i)}\|x^{(i)}_{t-1} - x^{(i)}_t\|^2 \nn\\
& \overset{(iii)}\geq  \tsum_{t=1}^{s_i}\tfrac{\rho\Gamma_t^{(i)}}{2\gamma_t^{(i)}}\| g_{X, t-1}^{(i)}\|^2 - \tsum_{t=1}^{s_i}\tfrac{\rho\Gamma_t^{(i)}}{\gamma_t^{(i)}}\| \hat \delta_{t}^{(i)}(x_{t-1}^{(i)}) \|^2 + \tsum_{t=1}^{s_i} (\tfrac{1}{8} - \rho)\Gamma_t^{(i)}\gamma_t^{(i)}\|x^{(i)}_{t-1} - x^{(i)}_t\|^2,
\end{align}
where step (i) follows from the definition of $\hat g_{X,t}$ in \eqref{def_projected_stochastic_gradient}, step (ii) follows from Young's inequality, and step (iii) follows from the definitions of \eqref{def_projected_stochastic_gradient} and \eqref{def_projected_deterministic_gradient} and the nonexpansive property of projection. By combining Ineqs.~\eqref{eq_7_0_2} and \eqref{eq_7_2}, we have
\begin{align}\label{eq_8_2}
\cS_i &\geq \tsum_{t=1}^{s_i}\tfrac{\rho\Gamma_t^{(i)}}{2\gamma_t^{(i)}}\|g_{X, t-1}^{(i)}\|^2 -\gamma_1^{(i)}\hat \cL_{1}^{(i)}\|x_1^{(i)} - x_0^{(i)}\|^2  + \tsum_{t=1}^{s_i} (\tfrac{1}{8} - \rho)\Gamma_t^{(i)}\gamma_t^{(i)}\|x^{(i)}_{t-1} - x^{(i)}_t\|^2\nn\\
&\quad - \tsum_{t=1}^{s_i}  \left(\tfrac{(1+\rho)\Gamma_t^{(i)}}{\gamma_t^{(i)}}\| \hat \delta_t^{(i)}(x^{(i)}_{t-1})\|^2 - \gamma_t^{(i)}\langle \hat \delta_t^{(i)}(x^{(i)}_{t-1}), x_{t-1}^{(i)} - x^*\rangle + \Gamma_{t-1}^{(i)}(\tfrac{L_{t}^{(i)}}{2}- \hat \cL_{t}^{(i)})\|x^{(i)}_{t-1} - x^{(i)}_{t}\|^2\right).
\end{align}
Then we can sum up $\cS_i$ for $i=1,...,m$ to get
\begin{align}
&\text{LHS of Ineq.~\eqref{eq_4_2}} = \tsum_{i=1}^m \cS_i\nn\\
& \geq \tsum_{t=1}^{k}\tfrac{\rho\Gamma_t}{2\gamma_t}\|g_{X, t-1}\|^2  - \tsum_{i=1}^m \gamma_1^{(i)}\hat \cL_{1}^{(i)}\|x_1^{(i)} - x_0^{(i)}\|^2 + \tsum_{t=1}^k (\tfrac{1}{8} - \rho)\Gamma_t\gamma_t\|x_{t-1} - x_t\|^2\nn\\
& \quad -  \tsum_{t=1}^{k}  \left(\tfrac{(1+\rho)\Gamma_t}{\gamma_t}\| \hat \delta_t(x_{t-1})\|^2 - \gamma_t\langle \hat \delta_t(x_{t-1}), x_{t-1} - x^*\rangle + (\Gamma_{t}-\gamma_t)(\tfrac{L_{t}}{2}- \hat \cL_{t})\|x_{t-1} - x_{t}\|^2 \right) \nn\\
& \overset{(i)}\geq \tsum_{t=1}^{k}\tfrac{\rho\Gamma_t}{2\gamma_t}\|g_{X, t-1}\|^2  - \tfrac{108}{5}\hat \cL^2_{k}D_X^2 + \tsum_{t=1}^k (\tfrac{1}{8} - \rho)\Gamma_t\gamma_t\|x_{t-1} - x_t\|^2\nn\\
& \quad -  \tsum_{t=1}^{k}  \left(\tfrac{(1+\rho)\Gamma_t}{\gamma_t}\| \hat \delta_t(x_{t-1})\|^2 - \gamma_t\langle \hat \delta_t(x_{t-1}), x_{t-1} - x^*\rangle + (\Gamma_{t}-\gamma_t)(\tfrac{L_{t}}{2}- \hat \cL_{t})\|x_{t-1} - x_{t}\|^2 \right),
\label{eq_9_2_new}
\end{align}
where the second inequality follows from using $\hat\cL_{1}^{(i-1)} < \tfrac{2}{3}\hat\cL_1^{(i)}$ for $i\geq 2$ to obtain 
\begin{align*}
\tsum_{i=1}^m \gamma_1^{(i)} \hat\cL_{1}^{(i)} = \tsum_{i=1}^m12\hat\cL_{0}^{(i)}\hat\cL_{1}^{(i)} \leq \tsum_{i=1}^m 12 (\hat\cL_{1}^{(i)})^2 \leq \tfrac{108}{5}  (\hat\cL_{1}^{(m)})^2 \leq \tfrac{108}{5} \hat \cL^2_{k}.
\end{align*}
By substituting Ineq.~\eqref{eq_9_2_new} into \eqref{eq_4_2} and rearranging the terms, we obtain
\begin{align*}
&\tsum_{t=1}^{k}\tfrac{\rho\Gamma_t}{2\gamma_t}\|g_{X, t-1}\|^2 + \tsum_{t=1}^k (\tfrac{1}{8} - \rho)\Gamma_t\gamma_t\|x_{t-1} - x_t\|^2\\
& \leq 94 \hat \cL_{k}^2D_X^2 + 6kl\hat\cL_{k}D_X^2 + \tsum_{t=1}^{k} \tfrac{(1+\rho)\Gamma_t}{\gamma_t}\| \hat \delta_t(x_{t-1})\|^2\\
&\qquad+ \tsum_{t=1}^{k}  \gamma_t\langle \hat \delta_t(x_{t-1}),  x^* - x_{t-1}\rangle + \tsum_{t=1}^{k} \Gamma_{t}(\tfrac{L_{t}}{2}- \hat \cL_{t})\|x_{t-1} - x_{t}\|^2, 
\end{align*}
which completes the proof of this proposition. 
\end{proof}

\vgap

Now we are ready to present the main theorem of this subsection, which describes the convergence properties of AC-SPG in expectation.
\begin{theorem}\label{thm_2}
Assume Assumptions \ref{assump_unbiasness}-\ref{assump_smooth_stoch} hold. Let $\{x_t\}$ be generated by Algorithm~\ref{alg_2}. We have
\begin{align}\label{res_1_thm_2_case_1}
\bbe\left[\|g_{X, R(k)}\|^2\right] 
\leq48\left(1+ 2 \lfloor \log_{\frac{3}{2}} \tfrac{\hat \cL}{\bar \cL_0}\rfloor\right)\left[\tfrac{94\hat\cL^2 D_X^2}{k^2} +  \tfrac{6
\hat\cL l D_X^2}{k}+ \tfrac{\sigma^2}{k^2}\cdot\bbe\left[\tsum_{t=1}^k \tfrac{26(\Gamma_{t-1} + \gamma_t)}{\gamma_t b_t}\right]\right].
\end{align}

(a) If we choose an adaptive 
batch size $b_t = \lceil\tfrac{\alpha(\Gamma_{t-1} + \gamma_t)}{\gamma_t} \rceil$
with some $\alpha >0$, we have
\begin{align}
\bbe\left[\|g_{X, R(k)}\|^2\right] \leq 48\left(1+ 2 \lfloor \log_{\frac{3}{2}} \tfrac{\hat \cL}{\bar \cL_0}\rfloor\right)\left[\tfrac{94\hat\cL^2 D_X^2}{k^2} +  \tfrac{6\hat\cL l D_X^2}{k}+ \tfrac{26 \sigma^2}{\alpha k}\right].
\label{res_2_thm_2_case_1}
\end{align}

(b) If we choose a constant batch size $b_t = b$, we have  
\begin{align}\label{res_0_thm_2_case_1}
\bbe\left[\|g_{X, R(k)}\|^2\right] \leq 48\left(1+ 2 \lfloor \log_{\frac{3}{2}} \tfrac{\hat \cL}{\bar \cL_0}\rfloor\right)\left[\tfrac{94\hat\cL^2 D_X^2}{k^2} +  \tfrac{6\hat\cL l D_X^2}{k}+ \tfrac{13(k+1) \sigma^2}{k b}\right].
\end{align}
\end{theorem}

\begin{proof} Setting $\rho = \tfrac{1}{8}$ in Ineq.~\eqref{ac-spg_main_recursion} of Proposition~\ref{prop:ac-spg_main_recursion}, and taking expectation, we have 
\begin{align}\label{expected_ac-spg_main_recursion}
&\bbe\left[\tsum_{t=1}^{k}\tfrac{\Gamma_t}{16\gamma_t}\|g_{X, t-1}\|^2 \right]\nn\\
& \leq \bbe[94 \hat \cL_{k}^2D_X^2] + \bbe[6kl\hat\cL_{k}D_X^2] + \bbe\left[\tsum_{t=1}^{k} \tfrac{9\Gamma_t}{8\gamma_t}\| \hat \delta_t(x_{t-1})\|^2\right]\nn\\
&\quad+ \bbe\left[\tsum_{t=1}^{k}  \gamma_t\langle \hat \delta_t(x_{t-1}),  x^* - x_{t-1}\rangle\right] + \bbe\left[\tsum_{t=1}^{k} \Gamma_{t}(\tfrac{L_{t}}{2}- \hat \cL_{t})\|x_{t-1} - x_{t}\|^2\right]\nn\\
& \overset{(i)}\leq 94 \hat \cL^2D_X^2 + 6kl\hat\cL D_X^2 + \bbe\left[\tsum_{t=1}^{k} \tfrac{9\Gamma_t}{8\gamma_t}\| \hat \delta_t(x_{t-1})\|^2\right]+ \bbe\left[\tsum_{t=1}^{k} \Gamma_{t}(\tfrac{L_{t}}{2}- \hat \cL_{t})\|x_{t-1} - x_{t}\|^2\right]\nn\\
& \overset{(ii)}\leq 94 \hat \cL^2D_X^2 + 6kl\hat\cL D_X^2 + \bbe\left[\tsum_{t=1}^{k} \tfrac{9(\Gamma_{t-1} + \gamma_t) \sigma^2}{8\gamma_t b_t}\right]+ \bbe\left[\tsum_{t=1}^{k} \Gamma_{t}(\tfrac{L_{t}}{2}- \hat \cL_{t})\|x_{t-1} - x_{t}\|^2\right],
\end{align}
where step (i) follows from $\hat \cL_{k-1} \leq \hat \cL$ and the martingale property $\bbe[\gamma_t\langle \hat \delta_t(x_{t-1}), x_{t-1} - x^*\rangle] = 0$; step (ii) follows from using $\Gamma_t \leq \Gamma_{t-1} + \gamma_t$,  $\Gamma_{t-1}$ and $\gamma_t$ being $\cF_{t-1}$-measurable, and Assumption~\ref{assump_variance} to show
\begin{align*}
\bbe\left[\tsum_{t=1}^k   \tfrac{9\Gamma_t}{8\gamma_t}\|\hat \delta_t (x_{t-1})\|^2\right] \leq \bbe\left[\tsum_{t=1}^k   \tfrac{9(\Gamma_{t-1} + \gamma_t)}{8\gamma_t}\|\hat \delta_t (x_{t-1})\|^2\right] \leq \bbe\left[\tsum_{t=1}^k \tfrac{9(\Gamma_{t-1} + \gamma_t)\sigma^2}{8\gamma_t b_t}\right].
\end{align*}
Recall the definition $M_k:=\tsum_{t=1}^k \tfrac{\Gamma_t}{\gamma_t}$, the LHS of Ineq.~\eqref{expected_ac-spg_main_recursion} satisfies
\begin{align*}
 \bbe\left[\tsum_{t=1}^{k}\tfrac{\Gamma_t}{16\gamma_t}\|g_{X, t-1}\|^2 \right]= \bbe\left[\left(\tsum_{t=1}^{k}\tfrac{\Gamma_t}{16\gamma_t}\right)\|g_{X, R(k)}\|^2 \right] = \bbe\left[\tfrac{M_k}{16}\|g_{X, R(k)}\|^2 \right],
\end{align*}
which indicates that 
\begin{align}\label{eq_10_2}
\bbe\left[\tfrac{M_k}{16}\|g_{X, R(k)}\|^2\right]  \leq 94\hat\cL^2 D_X^2 +  6k\hat\cL l D_X^2+ \bbe\left[\tsum_{t=1}^k \tfrac{9(\Gamma_{t-1} + \gamma_t)\sigma^2}{8\gamma_t b_t}\right] + \bbe\left[\tsum_{t=1}^{k} \Gamma_{t}(\tfrac{L_{t}}{2}- \hat \cL_{t})\|x_{t-1} - x_{t}\|^2\right].
\end{align}
Now it remains to upper bound $\bbe\left[\tsum_{t=1}^{k} \Gamma_{t}(\tfrac{L_{t}}{2}- \hat \cL_{t})\|x_{t-1} - x_{t}\|^2\right]$. Notice that $\hat \cL_t \geq \max\{\hat \cL_{t-1}, \bar \cL_t\} \geq \tfrac{\hat \cL_{t-1} + \bar \cL_t}{2}$, which indicates that 
\begin{align*}
\tfrac{L_t}{2} - \hat\cL_t \leq \tfrac{L_t - \bar \cL_t}{2} - \tfrac{\hat \cL_{t-1}}{2}\leq \tfrac{(L_t - \bar \cL_t)^2}{8 \hat \cL_{t-1}}.
\end{align*}
Then we have
\begin{align*}
\bbe\left[\tsum_{t=1}^{k} \Gamma_{t}(\tfrac{L_{t}}{2}- \hat \cL_{t})\|x_{t-1} - x_{t}\|^2\right] & \overset{(i)}\leq \bbe\left[\tsum_{t=1}^k(\Gamma_{t-1}+ \gamma_t) \tfrac{(\bar \cL_t - L_t)^2}{8\hat \cL_{t-1}} \|x_{t-1} - x_t\|^2 \right]\\
& \overset{(ii)}= \bbe\left[\tsum_{t=1}^k\tfrac{3(\Gamma_{t-1}+ \gamma_t)}{2\gamma_t}(\bar \cL_t - L_t)^2\|x_{t-1} - x_t\|^2 \right]\\
& \overset{(iii)}\leq \bbe\left[\tsum_{t=1}^k\tfrac{24(\Gamma_{t-1}+ \gamma_t)\sigma^2}{\gamma_t b_t} \right],
\end{align*}
where step (i) follows from $\Gamma_t \leq \Gamma_{t-1} + \gamma_t$, step (ii) follows from the definition $\gamma_t = 12 \hat \cL_{t-1}$, and step (iii) follows from Lemma~\ref{lemma_L_t_stochastic_new}. By substituting the inequality above into Ineq.~\eqref{eq_10_2}, we arrive at
\begin{align}\label{eq_10_3}
\bbe\left[\tfrac{M_k}{16}\|g_{X, R(k)}\|^2\right] 
\leq 94\hat\cL^2 D_X^2 +  6k\hat\cL l D_X^2+ \bbe\left[\tsum_{t=1}^k \tfrac{26(\Gamma_{t-1} + \gamma_t)\sigma^2}{\gamma_t b_t}\right].
\end{align}
Finally, by using Ineq.~\eqref{eq:lower_bound_M_k} in Lemma~\ref{lemma_bound_M_k}, we have 
\begin{align*}
\bbe\left[\|g_{X, R(k)}\|^2\right] 
\leq\tfrac{48\left(1+ 2 \lfloor \log_{\frac{3}{2}} \frac{\hat \cL}{\bar \cL_0}\rfloor\right)}{k^2}\left[94\hat\cL^2 D_X^2 +  6k\hat\cL l D_X^2+ \sigma^2\bbe\left[\tsum_{t=1}^k \tfrac{26(\Gamma_{t-1} + \gamma_t)}{\gamma_t b_t}\right]\right],
\end{align*}
which completes the proof of Ineq.~\eqref{res_1_thm_2_case_1}. By using the fact that $\tfrac{\Gamma_{t-1} + \gamma_t}{\gamma_t} \leq t$ and the two choices of batch sizes, we obtain Ineqs.~\eqref{res_2_thm_2_case_1} and \eqref{res_0_thm_2_case_1}.
\end{proof}
\vgap

Based on Theorem~\ref{thm_2}, let us derive the iteration and sample complexities of AC-SPG in order to find an $\epsilon$-approximate stationary point in expectation, i.e., an $x_R \in X$, such that $\bbe[\|g_{X, R}\|^2] \leq \epsilon^2$.

Firstly, consider the adaptive batch size $b_t = \lceil\tfrac{\alpha(\Gamma_{t-1} + \gamma_t)}{\gamma_t} \rceil \leq \lceil \alpha t\rceil$ for some $\alpha > 0$. According to \eqref{res_2_thm_2_case_1}, AC-SPG can find an $\epsilon$-stationary point in at most
\begin{align*}
\mathcal{O}\left(\tfrac{\hat \cL D_X \sqrt{\log (\hat \cL/\bar \cL_0)}}{\epsilon} + \tfrac{\hat \cL l D_X^2 \log (\hat \cL/\bar \cL_0)}{\epsilon^2} + \tfrac{ \sigma^2 \log (\hat \cL/\bar \cL_0)}{\alpha \epsilon^2}\right)
\end{align*}
iterations. Consequently, the total sample complexity $\tsum_{t=1}^k 2b_t$ is bounded by
\begin{align*}
\mathcal{O}\left(\tfrac{\hat \cL D_X \sqrt{\log (\hat \cL/\bar \cL_0)}}{\epsilon} + \tfrac{\hat \cL l D_X^2 \log (\hat \cL/\bar \cL_0)}{\epsilon^2} + \tfrac{ \sigma^2 \log (\hat \cL/\bar \cL_0)}{\alpha \epsilon^2} + \tfrac{\alpha \hat \cL^2 D_X^2 \log (\hat \cL/\bar \cL_0)}{\epsilon^2} + \tfrac{\alpha\hat \cL^2 l^2 D_X^4 \log^2 (\hat \cL/\bar \cL_0)}{\epsilon^4} + \tfrac{ \sigma^4 \log^2 (\hat \cL/\bar \cL_0)}{\alpha \epsilon^4}\right).
\end{align*}
Consider the case when $l>0$ and we choose $\alpha$ in the order of $\mathcal{O}\big(\min\big\{\tfrac{\sigma^2 \sqrt{\log (\hat \cL/\bar \cL_0)}}{\hat \cL\epsilon D_X}, \tfrac{\sigma^2}{\hat \cL lD_X^2}\big\}\big)$ (assuming we have a good enough estimation of both terms), AC-SPG achieves an iteration complexity of 
\begin{align}\label{iteration_complexity_1}
\mathcal{O}\left(\tfrac{\hat \cL D_X \sqrt{\log (\hat \cL/\bar \cL_0)}}{\epsilon} + \tfrac{\hat \cL l D_X^2 \log (\hat \cL/\bar \cL_0)}{\epsilon^2} \right),
\end{align}
and a sample complexity of 
\begin{align}\label{sample_complexity_1}
\widetilde{\mathcal{O}}\left(\tfrac{\hat \cL D_X}{\epsilon} + \tfrac{\hat \cL l D_X^2}{\epsilon^2} + \tfrac{\hat \cL D_X \sigma^2}{\epsilon^3} + \tfrac{\hat \cL l D_X^2 \sigma^2}{\epsilon^4}\right).
\end{align}
These unified iteration and sample complexity bounds nearly matches the complexity \eqref{complexity_bound_spg} for the SPG method (with $L$ replaced by $\hat \cL$), up to logarithmic factors of $\tfrac{\hat \cL}{\bar \cL_0}$. 

Secondly, we consider the choice of a fixed batch size $b_t = b$ for $b \in \mathbb{Z}_+$. To guarantee the convegence of the algorithm, we set $b = \max\{1, \beta k \}$, then AC-SPG achieves the iteration complexity of 
\begin{align*}
\mathcal{O}\left(\tfrac{\hat \cL D_X \sqrt{\log (\hat \cL/\bar \cL_0)}}{\epsilon} + \tfrac{\hat \cL l D_X^2 \log (\hat \cL/\bar \cL_0)}{\epsilon^2} + \tfrac{\sigma^2 \log (\hat \cL/\bar \cL_0)}{\beta \epsilon^2}\right),
\end{align*}
and the sample complexity of
\begin{align*}
\widetilde{\mathcal{O}}\left(\tfrac{\hat \cL D_X }{\epsilon} + \tfrac{\hat \cL l D_X^2 }{\epsilon^2} + \tfrac{\sigma^2 }{\beta \epsilon^2} + \tfrac{\beta \hat \cL^2 D_X^2 }{\epsilon^2} + \tfrac{\beta\hat \cL^2 l^2 D_X^4 }{\epsilon^4} + \tfrac{\sigma^4 }{\beta \epsilon^4}\right).
\end{align*}
Moreover, if we choose $\beta$ in the order of $\mathcal{O}\big(\min\big\{\tfrac{\sigma^2 \sqrt{\log (\hat \cL/\bar \cL_0)}}{\hat \cL\epsilon D_X}, \tfrac{\sigma^2}{\hat \cL lD_X^2}\big\}\big)$, AC-SPG can also achieve the iteration complexity of \eqref{iteration_complexity_1} and the sample complexity of \eqref{sample_complexity_1}.

\revision{}{
Finally, we would like to point out that the convergence guarantees of AC-SPG rely on the smoothness assumption of the stochastic function (Assumption~\ref{assump_smooth_stoch}). However, under this assumption, the lower bound for sample complexity is $\mathcal{O}\left(\epsilon^{-3}\right)$, rather than $\mathcal{O}\left(\epsilon^{-4}\right)$; see \cite{arjevani2023lower}.
This observation motivates our exploration of further acceleration techniques, particularly through variance reduction, as discussed in Section~\ref{sec_variance_reduction}.}

\section{Variance reduction methods for nonconvex stochastic optimization}\label{sec_variance_reduction}

This section leverages variance reduction techniques to further speed up the convergence rate of stochastic projected gradient methods. In Subsection~\ref{subsec_vr_nonadaptive}, we introduce a variance-reduced stochastic projected gradient (VR-SPG) method and present a nearly optimal complexity bound that unifies both convex and nonconvex cases for the first time in the literature. In Subsection~\ref{subsec_vr_adaptive}, we propose an ``auto-conditioned'' step-size policy for VR-SPG and provide improved convergence guarantees compared to the AC-SPG method discussed in Subsection~\ref{sec_stoch_expectation}.

\subsection{Variance-reduced stochastic projected gradient method}\label{subsec_vr_nonadaptive}
In this subsection, we only require the mean-squared smoothness of the stochastic gradient $\cG(x, \xi)$, which is weaker than the almost sure smoothness condition in Assumption~\ref{assump_smooth_stoch}.
\begin{assumption}[Mean-squared smoothness]\label{assump_mean_squared_smoothness}
    There exists a $\bar \cL \geq 0$ such that for any $x, y \in X$,
    \begin{align}\label{eq:mean-squared_smoothness}
    \bbe_\xi\|\cG(x, \xi) - \cG(y, \xi)\|^2 \leq \bar \cL^2 \|x - y\|^2.
    \end{align}
\end{assumption}
Clearly, we have $\bar \cL \leq \hat \cL$ when Assumption~\ref{assump_smooth_stoch} is satisfied.

In Algorithm~\ref{alg_vr}, we present a variance-reduced stochastic projected gradient (VR-SPG) method. VR-SPG computes a batched gradient estimator $\tilde G_t$, using a sample size $N$ at every $T$ iterations. During the intermediate iterations, the gradient estimator $\tilde G_t$ will be updated in a recursive manner based on $\tilde G_{t-1}$, using a mini-batch of size $b_t$. 
\begin{algorithm}[h]\caption{Variance-reduced stochastic projected gradient (VR-SPG) method}\label{alg_vr}
	\begin{algorithmic}
		\State{\textbf{Input}: initial point $x_0 \in X$, 
         step size parameter $\{\gamma_t\}$, batch size parameters $\{b_t\}$ and $N$, epoch length parameter $T \geq 2$,
        and epoch index $s=0$. 
  }
		\For{$t=1,\cdots, k$}
        \If{$\mathrm{mod}(t, T) = 1$}
        \State{Set $s \leftarrow s+1$.}
        \State{Generate i.i.d. observations $\xi^t_1, ..., \xi^t_N$ 
        and compute
        \begin{align}
        \tilde G_t = \tfrac{1}{N}\tsum_{i=1}^N \cG(x_{t-1}, \xi^t_i).
        \end{align}
        }
        \Else
        \State{Generate i.i.d. observations $\xi_1^t, ..., \xi_{b_t}^{t}$ and compute
        \begin{align}
        \tilde G_t = \tfrac{1}{b_t}\tsum_{i=1}^{b_t}[\cG(x_{t-1}, \xi_i^t) - \cG(x_{t-2}, \xi_i^t)] + \tilde G_{t-1}.
        \end{align}
        }
        \EndIf
        \State{Compute
       \begin{align}
  x_t = \arg \min_{x\in X} \left\{\langle \tilde G_t , x \rangle + \tfrac{\gamma_t}{2}\|x_{t-1} - x\|^2 \right\}. \label{prox-mapping_3}
  \end{align} 
        } 
\EndFor
\State{\textbf{Output}: we let the output at time step $k$, denoted as $x_{R(k)}$, be randomly selected from $\{x_0, ..., x_{k-1}\}$ according to the following probability mass function:
\begin{align}\label{random_output_vr}
\mathbb{P}(x_{R(k)} = x_{t-1}) = t\left(\tsum_{j=1}^k j\right)^{-1}.
\end{align}}
	\end{algorithmic}
\end{algorithm}

We define the projected stochastic gradient at the iterate $x_t$ as
\begin{align}\label{def_projected_stochastic_gradient_vr}
\tilde g_{X, t}:= P_X(x_t, \tilde G_{t+1}, \gamma_{t+1}) = \gamma_{t+1}(x_{t} - x_{t+1}),
\end{align}
and recall the corresponding projected gradient defined in \eqref{def_projected_deterministic_gradient}.
In the following discussion, we use two equivalent indexing conventions interchangeably:
\begin{align}\label{def_index_rule}
x_t = x_{s, u}, \quad \text{where $t = (s-1) T + u$.}
\end{align}
Here $s$ denotes the epoch index as defined in Algorithm~\ref{alg_vr}, and $u$ represents the index within the epoch. Additionally, We  have $x_{s, 0} = x_{s-1, T}$. For convenience, we define $\tilde \delta_t := \tilde G_t - \nabla f(x_{t-1})$. The following two lemmas outline the properties of the operator $\tilde G_t$. 

\begin{lemma}\label{bound_error_variance}
Define $q_{s, j} := \tfrac{\bbe_\xi\|\cG(x_{s, j-1}, \xi) - \cG(x_{s, j}, \xi)\|^2}{\|x_{s, j-1}-x_{s, j}\|^2}$. 
If Assumptions~\ref{assump_unbiasness}, \ref{assump_variance}, and \ref{assump_mean_squared_smoothness} hold, we have
\begin{align*}
\bbe\|\tilde \delta_{s, u}\|^2 \leq \tsum_{j=2}^u \bbe\left[\tfrac{q_{s,j}}{b_{s,j}}\|x_{s, j-1} - x_{s, j-2}\|^2\right] + \tfrac{\sigma^2}{N} \leq  \tsum_{j=2}^u \bbe\left[\tfrac{\bar \cL^2}{b_{s,j}}\|x_{s, j-1} - x_{s, j-2}\|^2\right] + \tfrac{\sigma^2}{N}.
\end{align*}
\end{lemma}
\begin{proof}
When $u=1$ the result follows directly. For $u \geq 2$, we can bound
\begin{align*}
\bbe\|\tilde \delta_{s, u}\|^2 & = \bbe\|\tfrac{1}{b_{s, u}}\tsum_{i=1}^{b_{s, u}}[\cG(x_{s, u-1}, \xi_i^{s, u}) - \cG(x_{s, u-2}, \xi_i^{s, u})] - \nabla f(x_{s, u-1}) + \nabla f(x_{s, u-2}) +\tilde \delta_{s, u-1}\|^2\\
& = \bbe\|\tfrac{1}{b_{s, u}}\tsum_{i=1}^{b_{s, u}}[\cG(x_{s, u-1}, \xi_i^{s, u}) - \cG(x_{s, u-2}, \xi_i^{s, u})] - \nabla f(x_{s, u-1}) + \nabla f(x_{s, u-2})\|^2 + \bbe \|\tilde \delta_{s, u-1}\|^2\\
& = \bbe\left[\tfrac{1}{b^2_{s, u}} \tsum_{i=1}^{b_{s, u}}\|\cG(x_{s, u-1}, \xi_i^{s, u}) - \cG(x_{s, u-2}, \xi_i^{s, u}) - \nabla f(x_{s, u-1}) + \nabla f(x_{s, u-2})\|^2\right] + \bbe \|\tilde \delta_{s, u-1}\|^2\\
& \leq \bbe\left[\tfrac{1}{b^2_{s, u}} \tsum_{i=1}^{b_{s, u}}\|\cG(x_{s, u-1}, \xi_i^{s, u}) - \cG(x_{s, u-2}, \xi_i^{s, u})\|^2\right] + \bbe \|\tilde \delta_{s, u-1}\|^2\\
& \leq \bbe\left[\tfrac{q_{s, u}}{b_{s, u}} \|x_{s, u-1} - x_{s, u-2}\|^2\right] + \bbe \|\tilde \delta_{s, u-1}\|^2\\
& \leq \bbe\left[\tfrac{\bar \cL^2}{b_{s, u}} \|x_{s, u-1} - x_{s, u-2}\|^2\right] + \bbe \|\tilde \delta_{s, u-1}\|^2,
\end{align*}
where the last inequality follows from Assumption~\ref{assump_mean_squared_smoothness}. By utilizing the above inequality recursively, we obtain the desired result. 
\end{proof}

\begin{lemma}\label{bound_error_inner_product}
If Assumptions~\ref{assump_unbiasness}-\ref{assump_variance} hold, for any $u \geq 1$ and any sequence of positive real numbers $\{\beta_{s, j}\}_{j=1}^{u-1}$, we have 
\begin{align*}
\bbe[\langle \tilde\delta_{s, u},  x^* - x_{s, u-1}\rangle] \leq \tsum_{j=2}^{u} \bbe[\tfrac{1}{2\beta_{s, j-1}}\|x_{s, j-1}- x_{s, j-2}\|^2] + \tsum_{j=2}^u\bbe[\tfrac{\beta_{s, j-1}}{2}\|\tilde \delta_{s, j-1} \|^2].
\end{align*}
\end{lemma}
\begin{proof}
It is not difficult to see that when $u=1$,
\begin{align}\label{recursion_zero}
\bbe[\langle \tilde\delta_{s, 1}, x^* - x_{s, 0}\rangle] = 0.
\end{align}
For $u \geq 2$, we have
\begin{align*}
&\bbe[\langle \tilde\delta_{s, u}, x^* - x_{s, u-1} \rangle]\nn\\
&= \bbe[\langle \nabla f(x_{s, u-1}) - \nabla f(x_{s, u-2}) + \tilde G_{s, u-1} - \nabla f(x_{s, u-1}),  x^* - x_{s, u-1}\rangle ]\nn\\
& = \bbe[\langle \tilde \delta_{s, u-1},  x_{s, u-2} - x_{s, u-1}\rangle] + \bbe[\langle \tilde \delta_{s, u-1},  x^* -  x_{s, u-2}\rangle]\nn\\
& \overset{(i)}\leq \bbe[\tfrac{1}{2\beta_{s, u-1}}\|x_{s, u-1} - x_{s, u-2}\|^2] + \bbe[\tfrac{\beta_{s, u-1}}{2}\|\tilde \delta_{s, u-1}\|^2]+ \bbe[\langle \tilde \delta_{s, u-1},  x^* - x_{s, u-2} \rangle],
\end{align*}
where step (i) follows from Young's inequality.
By applying the above relationship recursively and using the fact \eqref{recursion_zero}, we obtain the desired result. 
\end{proof}
\vgap


We are now ready to establish the main convergence results of VR-SPG. 
\begin{theorem}\label{thm_6}
Suppose that Assumptions~\ref{assump_unbiasness}, \ref{assump_variance}, and \ref{assump_mean_squared_smoothness} hold. Let $\{x_t\}$ be generated by Algorithm~\ref{alg_vr}, and assume $k\geq T$. Set $\gamma_t = \gamma \geq 4 \bar\cL$ and define the batch size as:
\begin{align}\label{batch_size_vr}
b_{s, u} := \begin{cases}
    \tfrac{T^2}{u-1}, &\text{for $s=1$}\\
    \tfrac{13 T}{2}, &\text{for $s\geq 2$}.
\end{cases}
\end{align}
Then we have
\begin{align}\label{res_thm6_1}
& \bbe\left[ \|g_{X, R(k)}\|^2  \right]\leq \left(\tsum_{t=1}^k\tfrac{t}{8\gamma}\right)^{-1}\left[\tfrac{l k}{2}D_X^2 + \tfrac{17\gamma}{16}D_X^2
+ \tfrac{21(k+1)k\sigma^2}{8 \gamma N}\right].
\end{align}
Specifically, if $\gamma = 4\bar\cL$, we have
\begin{align}\label{res_thm5_1}
& \bbe\left[ \|g_{X, R(k)}\|^2  \right]\leq \left[\tfrac{32\bar\cL l D_X^2}{k+1} + \tfrac{272\bar\cL^2 D_X^2}{k (k+1)} 
+ \tfrac{42 \sigma^2}{  N}\right].
\end{align}
\end{theorem}
\begin{proof}
First, recall that in the proof of Theorem~\ref{thm_2_0}, the first few steps leading to Ineq.~\eqref{proof_stoch_4} were derived using the optimality condition of \eqref{prox-mapping_2_0}, the smoothness of $f$, and the weak convexity of $f$. Consequently, Ineqs.~\eqref{proof_stoch_2} and  \eqref{proof_stoch_4} remain valid in this proof, with $\hat \delta_t(x_{t-1})$ in \eqref{proof_stoch_4} replaced by $\tilde \delta_t$. Thus, we obtain
\begin{align}\label{proof_vr_0}
f(x_{t-1}) - f(x_{t}) \geq (\tfrac{3\gamma_t}{4} - \tfrac{L}{2})\|x_t - x_{t-1}\|^2 - \tfrac{\|\tilde \delta_t\|^2}{\gamma_t},
\end{align}
and
\begin{align}\label{proof_vr_1}
&\tsum_{t=1}^k \left[f(x_t) - f(x^*) + (\tfrac{\gamma}{4} - \tfrac{L}{2})\|x_t - x_{t-1}\|^2 - \tfrac{\|\tilde \delta_t\|^2}{\gamma} \right]\nn\\
& \leq \tfrac{l k}{2}D_X^2 + \tfrac{\gamma}{2}D_X^2 + \tsum_{t=1}^k\langle \tilde \delta_t, x^* - x_{t-1}\rangle.
\end{align}
Now we further derive a lower bound for the LHS of the above inequality.
\begin{align}\label{proof_vr_2}
&\tsum_{t=1}^k \left[f(x_t) - f(x^*) + (\tfrac{\gamma}{4} - \tfrac{L}{2})\|x_t - x_{t-1}\|^2 - \tfrac{\|\tilde \delta_t\|^2}{\gamma} \right]\nn\\
& \geq \tsum_{t=1}^k  \left[f(x_t) - f(x_k) + (\tfrac{\gamma}{4} - \tfrac{L}{2})\|x_t - x_{t-1}\|^2 - \tfrac{\|\tilde \delta_t\|^2}{\gamma} \right]\nn\\
& = \tsum_{t=1}^k  \left[\tsum_{j=t}^{k-1} [f(x_j) - f(x_{j+1})] + (\tfrac{\gamma}{4} - \tfrac{L}{2})\|x_t - x_{t-1}\|^2 - \tfrac{\|\tilde \delta_t\|^2}{\gamma} \right]\nn\\
& \overset{(i)}\geq \tsum_{t=1}^k\left[\tsum_{j=t}^{k-1} [(\tfrac{3\gamma}{4} - \tfrac{L}{2})\|x_{j} - x_{j+1}\|^2 - \tfrac{\|\tilde \delta_{j+1}\|^2}{\gamma}] + (\tfrac{\gamma}{4} - \tfrac{L}{2})\|x_t - x_{t-1}\|^2 - \tfrac{\|\tilde \delta_t\|^2}{\gamma} \right]\nn\\
& \overset{(ii)}=  \tsum_{t=1}^k \left[(\tfrac{(3t-2)\gamma}{4}-\tfrac{tL}{2})\|x_t - x_{t-1}\|^2 - \tfrac{t \|\tilde \delta_t\|^2}{\gamma} \right]\nn\\
& \overset{(iii)}=  \tsum_{t=1}^k \left[\tfrac{(t-1)\gamma - tL}{2}\|x_t - x_{t-1}\|^2 + \tfrac{t}{4\gamma}\|\tilde g_{X, t-1}\|^2 - \tfrac{t \|\tilde \delta_t\|^2}{\gamma} \right]\nn\\
& \overset{(iv)}\geq  \tsum_{t=1}^k \left[\tfrac{(t-1)\gamma - tL}{2}\|x_t - x_{t-1}\|^2 + \tfrac{t}{8\gamma}\|g_{X, t-1}\|^2 - \tfrac{5t \|\tilde \delta_t\|^2}{4\gamma} \right],
\end{align}
where step (i) follows from Ineq.~\eqref{proof_vr_0}, step (ii) stems from rearranging the terms, step (iii) is due to the definition of $\tilde g_{X, t}$ in \eqref{def_projected_stochastic_gradient_vr}, and step (iv) is obtained from Young's inequality and the nonexpansive property of projection. 
By combining Ineqs.~\eqref{proof_vr_1} and \eqref{proof_vr_2} and taking expectation, we have
\begin{align}\label{proof_vr_3}
&\tsum_{t=1}^k \bbe\left[\tfrac{(t-1)\gamma - tL}{2}\|x_t - x_{t-1}\|^2 + \tfrac{t}{8\gamma}\|g_{X, t-1}\|^2  \right]\nn\\
&\leq \tfrac{l k}{2}D_X^2 + \tfrac{\gamma}{2}D_X^2 + \bbe\left[\tsum^{k}_{t=1}  \tfrac{5t \|\tilde \delta_t\|^2}{4\gamma}\right]  + \bbe\left[\tsum_{t=1}^k\langle \tilde \delta_t, x^* - x_{t-1}\rangle\right].
\end{align}
Next, we bound $\bbe\left[\tsum_{t=1}^k\langle \tilde \delta_t, x^* - x_{t-1}\rangle\right]$ by epochs. First, for $s \geq 2$ and $2 \leq U \leq T$, it follows from  Lemma~\ref{bound_error_inner_product} with $\beta_{i, j} = \tfrac{4}{\gamma}$ that 
\begin{align}\label{proof_vr_4}
&\bbe\left[ \tsum_{u=1}^U \langle \tilde \delta_{s, u}, x^* - x_{s, u-1} \rangle\right]\nn\\
& \leq \tsum_{u=1}^U \tsum_{j=2}^{u} \bbe\left[\tfrac{\gamma}{8}\|x_{s, j-1}- x_{s, j-2}\|^2+ \tfrac{2}{\gamma}\|\tilde \delta_{s, j-1} \|^2\right]\nn\\
& \leq \tsum_{u=1}^{U-1}\bbe\left[\tfrac{(U-u)\gamma}{8}\|x_{s, u}- x_{s, u-1}\|^2+ \tfrac{2(U-u)}{\gamma}\|\tilde \delta_{s, u} \|^2\right]\nn\\
& \leq \bbe\left[\tsum_{u=1}^U \tfrac{((s-1)T + u) \gamma}{8}\|x_{s, u}- x_{s, u-1}\|^2\right] + \bbe\left[\tsum_{u=1}^U \tfrac{2((s-1)T + u)}{\gamma}\|\tilde \delta_{s, u} \|^2\right].
\end{align}
For $s = 1$ and $2 \leq U \leq T$, we have
\begin{align}\label{proof_vr_5}
\bbe\left[ \tsum_{u=1}^U \langle \tilde \delta_{1, u}, x^* - x_{1, u-1} \rangle\right] & \overset{(i)} \leq \tsum^{U}_{u=1} \bbe\left[ \tfrac{\gamma}{8U} \|x^* - x_{1, u-1}\|^2  + \tfrac{2U}{\gamma } \|\tilde \delta_{1, u}\|^2\right]\nn\\
&\leq \tfrac{\gamma}{8}D_X^2 + \tsum^{U}_{u=1} \bbe\left[ \tfrac{2U}{\gamma } \|\tilde \delta_{1, u}\|^2\right]\nn\\
&\overset{(ii)}\leq \tfrac{\gamma}{8}D_X^2 + \bbe\left[\tsum_{u=1}^U \tfrac{2 U}{\gamma}\cdot \left(\tsum_{j=2}^u \tfrac{\bar\cL^2}{b_{s,j}}\|x_{1, j-1} - x_{1, j-2}\|^2 + \tfrac{\sigma^2}{N}\right)\right]\nn\\
&\overset{(iii)}\leq \tfrac{\gamma}{8}D_X^2 + \tsum_{u=1}^{U} \tfrac{2 U \sigma^2}{\gamma N} +  \bbe\left[\tsum_{u=1}^{U-1} \tfrac{2\bar\cL^2 (U-u)U}{ \gamma b_{1, u+1}}\|x_{1, u} - x_{1, u-1}\|^2\right]\nn\\
& \overset{(iv)}\leq \tfrac{\gamma}{8}D_X^2+ \tfrac{2 U^2 \sigma^2}{\gamma N} + \bbe\left[\tsum_{u=1}^{U-1} \tfrac{2u \bar\cL^2}{\gamma}\|x_{1,u} - x_{1, u-1}\|^2 \right]\nn\\
& \overset{(v)}\leq \tfrac{\gamma}{8}D_X^2+ \tfrac{2 U^2 \sigma^2}{\gamma N} + \bbe\left[\tsum_{u=1}^{U-1} \tfrac{u \gamma}{8}\|x_{1, u} - x_{1, u-1}\|^2 \right],
\end{align}
where step (i) follows from Young's inequality, step (ii) applies Lemma~\ref{bound_error_variance}, step (iii) is obtained from rearranging the terms, step (iv) is due to $b_{1, u} \geq \tfrac{T^2}{u-1}$, and step (v) stems from $\gamma \geq 4\bar\cL$. We assume $k \geq T$, and combine Ineqs.~\eqref{proof_vr_3}-\eqref{proof_vr_5} to obtain
\begin{align}\label{proof_vr_6}
&\tsum_{t=1}^k \bbe\left[\tfrac{(t-1)\gamma - tL}{2}\|x_t - x_{t-1}\|^2 + \tfrac{t}{8\gamma}\|g_{X, t-1}\|^2  \right]\nn\\
&\leq \tfrac{l k}{2}D_X^2 + \tfrac{5\gamma}{8}D_X^2 + \tfrac{2 T^2 \sigma^2}{\gamma N} + \bbe\left[\tsum^{T}_{t=1}  \tfrac{5t \|\tilde \delta_t\|^2}{4\gamma}\right] + \bbe\left[\tsum_{t=T+1}^k\tfrac{13t\|\tilde \delta_{t}\|^2}{4\gamma}\right]+ \bbe\left[\tsum_{t=1}^{k} \tfrac{t \gamma}{8}\|x_t - x_{t-1}\|^2 \right]   .
\end{align}
Next, we derive an upper bound for $\bbe\left[\tsum^{T}_{t=1}  \tfrac{5t \|\tilde \delta_t\|^2}{4\gamma}\right]+\bbe\left[\tsum_{t=T+1}^k \tfrac{13t \|\tilde \delta_t\|^2}{4\gamma}\right]$. We achieve this by bounding them epoch-wise, i.e., for $s \geq 1$, $1 \leq U \leq T$, as follows.
\begin{align}
&\bbe\left[\tsum_{u = 1}^{U} \tfrac{((s-1)T + u)\|\tilde \delta_{s, u}\|^2}{ \gamma}\right]\nn\\
&\overset{(i)}\leq\bbe\left[\tsum_{u = 1}^{U} \tfrac{(s-1)T + u}{ \gamma} \cdot \left(\tsum_{j=2}^u \tfrac{\bar\cL^2}{b_{s,j}}\|x_{s, j-1} - x_{s, j-2}\|^2 + \tfrac{\sigma^2}{N}\right)\right]\nn\\
&\leq \tsum_{u = 1}^{U} \tfrac{(s-1)T + u}{ \gamma}\cdot \tfrac{\sigma^2}{N} + \bbe\left[\tsum_{u=1}^{U-1}\left( \tfrac{\bar\cL^2}{\gamma b_{s, u+1}}\|x_{s, u} - x_{s, u-1}\|^2 \cdot \left(\tsum_{j=u+1}^U ((s-1)T + j) \right) \right) \right]\nn\\
&\leq \tsum_{u = 1}^{U} \tfrac{(s-1)T + u}{ \gamma}\cdot \tfrac{\sigma^2}{N} + \bbe\left[\tsum_{u=1}^{U-1}\left( \tfrac{\bar\cL^2}{\gamma b_{s, u+1}}\|x_{s, u} - x_{s, u-1}\|^2 \cdot (U-u)\left((s-1)T + \tfrac{u+1+U}{2} \right)\right)  \right],
\end{align}
where step (i) follows from Lemma~\ref{bound_error_variance}.
When $s = 1$, since $b_{1, u} \geq \tfrac{T^2}{u-1}$ for $u \geq 2$, it follows that 
\begin{align*}
\bbe\left[\tsum_{u = 1}^{T} \tfrac{5u\|\tilde \delta_{1, u}\|^2}{4 \gamma}\right] &\leq \tsum_{u = 1}^{T} \tfrac{5u\sigma^2}{4 \gamma N} + \bbe\left[\tsum_{u=1}^{T-1} \tfrac{5\bar\cL^2}{4\gamma b_{1, u+1}}\|x_{1, u} - x_{1, u-1}\|^2 \cdot \tfrac{T(T+1)}{2}  \right]\nn\\
& \leq \tsum_{u = 1}^{T} \tfrac{5u\sigma^2}{4 \gamma N} + \bbe\left[\tsum_{u=1}^{T-1} \tfrac{ u \bar\cL^2}{\gamma}\|x_{1, u} - x_{1, u-1}\|^2   \right].
\end{align*}
When $s > 1$, we set $b_{s, u} \geq \tfrac{13T}{2}$, so that
\begin{align*}
& \bbe\left[\tsum_{u = 1}^{U} \tfrac{13((s-1)T + u)\|\tilde \delta_{s, u}\|^2}{4 \gamma}\right] \\ 
&\leq \tsum_{u = 1}^{U} \tfrac{13((s-1)T + u)\sigma^2}{4 \gamma N} + \bbe\left[\tsum_{u=1}^{U-1}\left( \tfrac{13\bar\cL^2}{4\gamma b_{s, u+1}}\|x_{s, u} - x_{s, u-1}\|^2 \cdot 2T((s-1)T + u)\right) \right]\nn\\
& \leq \tsum_{u = 1}^{U} \tfrac{13((s-1)T + u)\sigma^2}{4 \gamma N} + \bbe\left[\tsum_{u=1}^{U-1} \tfrac{((s-1)T + u) \bar\cL^2}{\gamma}\|x_{s, u} - x_{s, u-1}\|^2   \right].
\end{align*}
Therefore, by summing over the iterates across all epochs, we obtain
\begin{align}\label{proof_vr_7}
\bbe\left[\tsum^{T}_{t=1}  \tfrac{5t \|\tilde \delta_t\|^2}{4\gamma}\right]+\bbe\left[\tsum_{t=T+1}^k \tfrac{13t \|\tilde \delta_t\|^2}{4\gamma}\right] &\leq \tsum^{k}_{t=1} \tfrac{13 t \sigma^2}{4\gamma N} - \tsum^{T}_{t=1} \tfrac{ 2t \sigma^2}{\gamma N} + \bbe\left[\tsum_{t=1}^k \tfrac{t \bar\cL^2}{\gamma} \|x_t - x_{t-1}\|^2 \right]\nn\\
& \leq \tfrac{13 (k+1)k\sigma^2}{8 \gamma N} - \tfrac{T(T+1)\sigma^2}{\gamma N} + \bbe\left[\tsum_{t=1}^k \tfrac{t \bar\cL^2}{\gamma} \|x_t - x_{t-1}\|^2 \right].
\end{align}
Combining Ineqs.~\eqref{proof_vr_6} and \eqref{proof_vr_7} yields
\begin{align*}
&\tsum_{t=1}^k \bbe\left[\left(\tfrac{(3t-4)\gamma - 4tL}{8} - \tfrac{t \bar\cL^2}{\gamma}\right)\|x_t - x_{t-1}\|^2 + \tfrac{t}{8\gamma}\|g_{X, t-1}\|^2  \right]\nn\\
&\leq \tfrac{l k}{2}D_X^2 + \tfrac{5\gamma}{8}D_X^2 + \tfrac{ T^2 \sigma^2}{\gamma N} + \tfrac{13 (k+1)k\sigma^2}{8 \gamma N}, 
\end{align*}
which, 
in view of the fact that $\tfrac{(3t-4)\gamma - 4tL}{8} - \tfrac{t \bar\cL^2}{\gamma} > 0$ for $t\geq 3$ due to $\gamma \geq 4\bar\cL$, then implies that
\begin{align}\label{proof_vr_9}
& \bbe\left[ \tsum_{t=1}^k\tfrac{t}{8\gamma}\|g_{X, t-1}\|^2  \right]\leq \tfrac{l k}{2}D_X^2 + \tfrac{17\gamma}{16}D_X^2 + \tfrac{T^2 \sigma^2}{\gamma N} + \tfrac{13 (k+1)k\sigma^2}{8 \gamma N}.
\end{align}
Finally, using the random output rule \eqref{random_output_vr}, we have $\bbe\left[ \tsum_{t=1}^k\tfrac{t}{8\gamma}\|g_{X, t-1}\|^2  \right] = (\tsum_{t=1}^k\tfrac{t}{8\gamma})\bbe[\|g_{X, R(k)}\|^2]$, thereby completing the proof.
\end{proof}
\vgap

Based on Theorem~\ref{thm_6}, we can deirve the iteration and sample complexities of the VR-SPG method for finding an $\epsilon$-stationary point in expectation. Specifically, we choose the epoch length $T = \lceil\sqrt{N}\rceil$ and the batch size $N$ as $\max\{ 1, \tfrac{84\sigma^2}{\epsilon^2}\}$. Under these settings, the term $\tfrac{42\sigma^2}{N}$ in \eqref{res_thm5_1} is upper bounded by $\epsilon^2/2$, and we require at most 
\begin{align*}
\O\left( \tfrac{\bar\cL D_X}{\epsilon} + \tfrac{\bar\cL lD_X^2}{\epsilon^2} \right),
\end{align*}
iterations of VR-SPG to achieve the desired $\epsilon$-accuracy. Furthermore, considering the batch size parameters $N$ and $b_t$ defined in \eqref{batch_size_vr}, the sample complexity (total number of oracle calls) is bounded by
\begin{align}\label{unified_bound_stochastic_nonconvex}
\O\left( N\tfrac{k}{T} + T^2 \log T + T^2 \tfrac{k}{T} \right) = \O\left( \tfrac{\bar\cL D_X}{\epsilon} + \tfrac{\bar\cL lD_X^2}{\epsilon^2} + \tfrac{\sigma^2}{\epsilon^2} \log \tfrac{\sigma}{\epsilon}  + \tfrac{\bar\cL D_X\sigma}{\epsilon^2} + \tfrac{\bar\cL l D_X^2 \sigma}{\epsilon^3}\right).
\end{align}
 This complexity bound improves the $\mathcal{O}(\epsilon^{-4})$ result in Section~\ref{sec_stochastic} to $\mathcal{O}(\epsilon^{-3})$, \revision{}{which is order-optimal when $\sigma$ is large, supported by the lower bound in \cite{arjevani2023lower}}. Additionally, these complexity bounds leverage the weakly convex structure. \rev{}{Notably, the best-known sample complexity bounds, e.g., \cite{wang2019spiderboost,pham2020proxsarah}, have dominating terms with order $\mathcal{O}\big(\tfrac{\sigma \bar \cL[f(x_0) - f^*]}{\epsilon^3}\big) \leq \mathcal{O}\big(\tfrac{\sigma \bar \cL L D_X^2}{\epsilon^3}\big)$. Therefore, our dominating term in \eqref{unified_bound_stochastic_nonconvex} is sharper when $f(x_0) - f^* \geq \Omega(lD_X^2)$, and the advantage of our algorithm is more significant when  $l \ll L$.} In the convex case where $l=0$, \eqref{unified_bound_stochastic_nonconvex} reduces to a $\tilde{\mathcal{O}}(\epsilon^{-2})$ bound, which is nearly optimal for stochastic convex optimization. To the best of our knowledge, this unified complexity bound for convex and nonconvex optimization is new for variance-reduced methods in the literature. 

\subsection{Auto-conditioned variance-reduced stochastic projected gradient (AC-VR-SPG) method}\label{subsec_vr_adaptive}
In this subsection, we propose an auto-conditioned variant of the VR-SPG method under Assumption~\ref{assump_smooth_stoch}. 
Recalling the definition of $\cL_{x, y}$ in \eqref{define_L_cL}, we present the AC-VR-SPG method in Algorithm~\ref{alg_ac_vr}.
\begin{algorithm}[h]\caption{Auto-conditioned variance-reduced stochastic projected gradient (AC-VR-SPG) method}\label{alg_ac_vr}
	\begin{algorithmic}
		\State{\textbf{Input}: initial point $x_0 \in X$, $\Gamma_0=0$, 
        $\bar \cL_0$ and $\hat \cL_{0}$ such that $0< \hat \cL_{0} = \bar\cL_0 \leq \hat\cL$, 
        batch size parameters $N$, epoch length parameter $T \geq 2$,
        and epoch index $s=0$. 
  }
		\For{$t=1,\cdots, k$}
        \If{$\mathrm{mod}(t, T) = 1$}
        \State{Set $s \leftarrow s+1$.}
        \State{Generate i.i.d. observations $\xi^s_1, ..., \xi^s_N$ and compute
        \begin{align}
        \tilde G_t &= \tfrac{1}{N}\tsum_{i=1}^N \cG(x_{t-1}, \xi^s_i). \label{def_hat_G_1}
        \end{align}
        }
        \Else
        \State{Generate i.i.d. observations $\xi_1^t, ..., \xi_{b_t}^{t}$ and compute
        \begin{align}
        \tilde G_t &= \tfrac{1}{b_t}\tsum_{i=1}^{b_t}[\cG(x_{t-1}, \xi_i^t) - \cG(x_{t-2}, \xi_i^t)] + \tilde G_{t-1}, \label{def_hat_G_2}
        \end{align}
        }
        \EndIf
        \State{Set $\gamma_t = 12\hat \cL_{t-1}$ and compute
       \begin{align}
  x_t = \arg \min_{x\in X} \left\{\langle \tilde G_t , x \rangle + \tfrac{\gamma_t}{2}\|x_{t-1} - x\|^2 \right\}. \label{prox-mapping_4}
  \end{align} 
        } 
        \State{Generate i.i.d. observations $\bar \xi^t_{1}, ..., \bar \xi^t_{B}$ and compute
        \begin{align}
  \bar\cL_t &= \tfrac{1}{B}\tsum_{j=1}^{B} \cL_{x_t, x_{t-1}}(\bar \xi_j^t).\label{evaluate_L_t_5}\\
  \tilde \cL_{t} & = \sqrt{\tfrac{\tsum_{i=1}^{B}\|\cG(x_{t}, \bar \xi_i^t) - \cG(x_{t-1}, \bar\xi_i^t)\|^2}{B \|x_{t}-x_{t-1}\|^2}}. \label{def_tilde_L}
        \end{align}
        \State{Update
        \begin{align}
        \hat \cL_{t} &= \max\{\hat \cL_{t-1}, \bar \cL_{t}, \tilde \cL_{t}\}.
        \end{align}}
        \State{If $\hat \cL_t > \tfrac{3}{2}\hat\cL_{t-1}$, set $\Gamma_t = \gamma_t$; otherwise set $\Gamma_t = \Gamma_{t-1} + \gamma_t$.}
        }
\EndFor
\State{\textbf{Output}: we let the output at time step $k$, denoted as $x_{R(k)}$, be randomly selected from $\{x_0, ..., x_{k-1}\}$ according to the following probability mass function:
\begin{align}\label{random_output_ac_vr}
\mathbb{P}(x_{R(k)} = x_{t-1}|\cF_k) = \left(\tsum_{j=1}^k \tfrac{\Gamma_j}{\gamma_j}\right)^{-1} \cdot \tfrac{\Gamma_t}{\gamma_t}.
\end{align}}
	\end{algorithmic}
\end{algorithm}

The primary distinction between AC-VR-SPG and VR-SPG lies in the requirement for Lipschitz constant knowledge. AC-VR-SPG does not rely on any prior knowledge of the Lipschitz constants. Instead, it dynamically estimates the local Lipschitz constants $\tilde \cL_{t}$ and $\bar \cL_t$ to determine the stepsizes. On the one hand, $\bar \cL_t$, computed using \eqref{evaluate_L_t_5}, is the same local smoothness estimator as \eqref{evaluate_L_t_2} employed in the AC-SPG method (Algorithm~\ref{alg_2}). On the other hand, $\tilde \cL_{t}$ is a mean-squared local smoothness estimator calculated using the same mini-batch of observations, which is used to control the stochastic error of the extrapolation term $\tfrac{1}{b_t}\tsum_{i=1}^{b_t}[\cG(x_{t-1}, \xi_i^t) - \cG(x_{t-2}, \xi_i^t)]$ in \eqref{def_hat_G_2}. Moreover, AC-VR-SPG follows the same output policy as AC-SPG (Algorithm~\ref{alg_2}).

In this subsection, we apply a similar definition of filtration as in Section~\ref{sec_stoch_expectation}. With a slight notational ambiguity, we let $b_t = N$ when $\text{mod}(t, T) = 1$.  Then, for 
$t\in \mathbb{Z}$, we define 
\begin{align*}
\cF_t&:= \sigma(\xi_1^1, ..., \xi_{b_1}^1, \bar \xi_1^1, ..., \bar\xi_{B}^1, 
......, 
\xi_1^t, ..., \xi_{b_t}^t, \bar \xi_1^t, ..., \bar\xi_{B}^t),\\
\cF_{t+1/2}&:= \sigma(\xi_1^1, ..., \xi_{b_1}^1, \bar \xi_1^1, ..., \bar\xi_{B}^1, 
......., 
\xi_1^t, ..., \xi_{b_t}^t, \bar \xi_1^t, ..., \bar\xi_{B}^t, \xi_1^{t+1}, ..., \xi_{b_{t+1}}^{t+1}).
\end{align*}
Clearly, $x_t$ and $L_t$ are $\cF_{t-1/2}$-measurable, $\bar \cL_t$, $\tilde \cL_t$ and $\Gamma_t$ are $\cF_t$-measurable, and $\gamma_t$ is $\cF_{t-1}$-measurable. The batch size $b_t$ should also be  $\cF_{t-1}$-measurable to avoid any searching routines.

To establish the convergence guarantees of AC-VR-SPG, we begin by introducing a technical result that characterizes the error caused by $\hat \cL_t$ underestimating $L_t$.
\begin{lemma}\label{lemma:bound_underestimation_ac_vr_spg}
Assume Assumptions \ref{assump_unbiasness}-\ref{assump_smooth_stoch} hold and recall the definition of $L_t$ in \eqref{define_L_t_stochastic}. We have
\begin{align}\label{bound_underestimation_ac_vr_spg}
\hat\cL_{t-1} \|x_t - x_{t-1}\|^2\cdot \bbe\left[(\tfrac{L_t}{2} - \hat \cL_t)_+ \big| \cF_{t-1/2}\right] \leq \tfrac{17}{4B}\min\{16\sigma^2, \hat \cL^2 \|x_t - x_{t-1}\|^2\}\cdot\mathbb{P}(\hat \cL_t \geq \tfrac{3}{2}\hat \cL_{t-1}|\cF_{t-1/2}).
\end{align}
\end{lemma}
\begin{proof}
In this proof, everything is discussed in $\cF_{t-1/2}$, thus $x_t$, $x_{t-1}$, and $L_t$ are fixed. For simplicity, define $p := \mathbb{P}(\hat \cL_t > \tfrac{3}{2}\hat \cL_{t-1}| \cF_{t-1/2})$ and $V:=\bbe[(\bar\cL_t - L_t)^2|\cF_{t-1/2}]$. If $L_t \leq 2 \hat \cL_{t-1}$, we have $(\tfrac{L_t}{2} - \hat \cL_t)_+ \leq (\tfrac{L_t}{2} - \hat \cL_{t-1})_+ \leq 0$. So we consider $L_t > 2 \hat \cL_{t-1} > 0$. Since $\hat \cL_t \geq \max\{\hat \cL_{t-1}, \bar \cL_t \}$, we have
\begin{align*}
(\tfrac{L_t}{2} - \hat \cL_t)_+ \leq (\tfrac{L_t}{2} - \bar\cL_t)_+ \leq \left(\tfrac{L_t}{2} - \bar\cL_t + \tfrac{\bar\cL_t ^2}{2L_t}\right)_+ = \tfrac{(\bar\cL_t - L_t)^2}{2L_t}.
\end{align*}
Consequently, we obtain
\begin{align}\label{bound_postive_part}
\bbe[(\tfrac{L_t}{2} - \hat \cL_t)_+|\cF_{t-1/2}] \leq \tfrac{L_t}{2} \min\left\{1, \tfrac{V}{L_t^2}\right\}.
\end{align}
Notice that the event $\{\bar \cL_t > \tfrac{3}{2}\hat \cL_{t-1}\}$ implies the event $\{\hat \cL_t > \tfrac{3}{2}\hat \cL_{t-1}\}$, thus $p \geq \mathbb{P}(\bar \cL_t > \tfrac{3}{2}\hat \cL_{t-1}| \cF_{t-1/2})$. Invoking the fact that $\bbe[\bar\cL_t] = L_t$, we have
\begin{align*}
L_t - \tfrac{3}{2}\hat\cL_{t-1} &= \bbe[\bar \cL_t - \tfrac{3}{2}\hat\cL_{t-1} | \cF_{t-1/2}]\\
&\leq \bbe[(\bar \cL_t - \tfrac{3}{2}\hat\cL_{t-1}) \mathbf{1}_{\{\bar \cL_t > \frac{3}{2}\hat\cL_{t-1}\}} | \cF_{t-1/2}]\\
&\leq  \sqrt{\bbe[\mathbf{1}^2_{\{\bar \cL_t > \frac{3}{2}\hat\cL_{t-1}\}}|\cF_{t-1/2}] \cdot \bbe[(\bar \cL_t - \tfrac{3}{2}\hat\cL_{t-1})^2|\cF_{t-1/2}]}\\
& \leq  \sqrt{p[V+(L_t - \tfrac{3}{2}\hat\cL_{t-1})^2]},
\end{align*}
where the second last step follows from Cauchy-Schwarz inequality. Therefore,
\begin{align}\label{bound_p}
p \geq \tfrac{(L_t - \frac{3}{2}\hat\cL_{t-1})^2}{V+ (L_t - \frac{3}{2}\hat\cL_{t-1})^2} \geq \tfrac{1}{1+ 16V/L_t^2},
\end{align}
where the second step follows from $L_t - \tfrac{3}{2}\hat \cL_{t-1} > L_t - \tfrac{3}{2}\cdot \tfrac{L_t}{2} = \tfrac{L_t}{4}$. Combining Ineqs.~\eqref{bound_postive_part} and \eqref{bound_p}, we obtain
\begin{align*}
\bbe[(\tfrac{L_t}{2} - \hat \cL_t)_+|\cF_{t-1/2}] \leq \tfrac{L_t p }{2} \min\left\{1, \tfrac{V}{L_t^2}\right\} (1+ \tfrac{16 V}{L_t^2}) \leq \tfrac{17pV  }{2L_t} \leq \tfrac{17pV  }{4 \hat \cL_{t-1}}.
\end{align*}
Invoking from Lemma~\ref{lemma_L_t_stochastic_new} that $V\|x_t - x_{t-1}\|^2 \leq \tfrac{\min\{16\sigma^2, \hat \cL^2 \|x_t - x_{t-1}\|^2\}}{B}$, we arrive at
\begin{align*}
\hat\cL_{t-1} \|x_t - x_{t-1}\|^2\cdot \bbe\left[(\tfrac{L_t}{2} - \hat \cL_t)_+ \right| \cF_{t-1/2}] \leq \tfrac{17}{4B}\min\{16\sigma^2, \hat \cL^2 \|x_t - x_{t-1}\|^2\}\cdot p,
\end{align*}
which completes the proof.
\end{proof}
\vgap

The next lemma describes the error introduced by $\hat \cL_t^2$ underestimating the mean squared local Lipschitz constant $q_t := \tfrac{\bbe[\|\cG(x_{t-1}, \xi) - \cG(x_t, \xi)\|^2|\cF_{t-1/2}]}{\|x_{t-1}-x_t\|^2}$.
\begin{lemma}\label{lemma:squared_gradient_difference}
Assume Assumptions \ref{assump_unbiasness}-\ref{assump_smooth_stoch} hold. We have 
\begin{align}\label{squared_gradient_difference}
\|x_t - x_{t-1}\|^2\cdot\bbe\left[ \left(q_t- 16 \hat \cL_{t}^2\right)_+ \big| \cF_{t-1/2}\right] \leq \tfrac{\hat \cL^2}{B} \min \left\{4.05\|x_t - x_{t-1}\|^2, \tfrac{4.26\sigma^2}{\hat \cL_{t-1}^2} \right\}\cdot\mathbb{P}(\hat \cL_t \geq \tfrac{3}{2}\hat \cL_{t-1}|\cF_{t-1/2}).
\end{align}
\end{lemma}
\begin{proof}
For notational simplicity, define 
\begin{align*}
Q_t = \tfrac{\tsum_{i=1}^{B}\|\cG(x_{t}, \bar \xi_i^t) - \cG(x_{t-1}, \bar\xi_i^t)\|^2}{B \|x_{t}-x_{t-1}\|^2},
\end{align*}
thus $\tilde \cL_t = \sqrt{Q_t}$. Also define $p := \mathbb{P}(\hat \cL_t > \tfrac{3}{2}\hat \cL_{t-1}| \cF_{t-1/2})$. If $q_t \leq 16 \hat\cL_{t-1}^2$, we have $\left(q_t- 16 \hat \cL_{t}^2\right)_+ \leq 0$ due to $\hat\cL_t \geq \hat\cL_{t-1}$, thus Ineq.~\eqref{squared_gradient_difference} holds directly. Therefore, we only consider the case $q_t > 16 \hat\cL_{t-1}^2$. By the definition of $Q_t$, we have $\bbe[Q_t | \cF_{t-1/2}] = q_t$ and 
\begin{align}
\bbe[Q_t^2| \cF_{t-1/2}] &\leq q_t^2 + \bbe[(Q_t - q_t)^2| \cF_{t-1/2}]\nn\\
&= q_t^2 + \tfrac{1}{B^2}\tsum_{i=1}^B\bbe(\tfrac{\|\cG(x_{t}, \bar \xi_i^t) - \cG(x_{t-1}, \bar\xi_i^t)\|^2}{\|x_{t}-x_{t-1}\|^2} - q_t)^2\nn \\
& \leq  q_t^2 + \tfrac{\hat \cL^2}{B^2}\tsum_{i=1}^B\tfrac{\bbe\|\cG(x_{t}, \bar \xi_i^t) - \cG(x_{t-1}, \bar\xi_i^t)\|^2}{\|x_{t}-x_{t-1}\|^2}\nn\\
& = q_t^2 + \tfrac{\hat \cL^2 q_t}{B}.\label{bound_second_moment_Q_t}
\end{align}
Notice that the event $\{Q_t \geq \tfrac{q_t}{4}\}$ indicates we have $\hat \cL_t \geq \sqrt{Q_t} \geq \tfrac{\sqrt{q_t}}{2}>2\hat \cL_{t-1} > \tfrac{3}{2}\hat \cL_{t-1}$. We have 
\begin{align*}
\bbe[Q_t \mathbf{1}_{\{Q_t \geq q_t/4\}}|\cF_{t-1/2}] = \bbe[Q_t |\cF_{t-1/2}]- \bbe[Q_t \mathbf{1}_{\{Q_t < q_t/4\}}|\cF_{t-1/2}]\geq q_t - \tfrac{q_t}{4} = \tfrac{3q_t}{4}.
\end{align*}
We can also apply the Cauchy-Schwarz inequality to obtain
\begin{align*}
\bbe[Q_t \mathbf{1}_{\{Q_t \geq q_t/4\}}|\cF_{t-1/2}] \leq \sqrt{\bbe[Q_t^2|\cF_{t-1/2}]\bbe[ \mathbf{1}^2_{\{Q_t \geq q_t/4\}}|\cF_{t-1/2}]} \leq \sqrt{\bbe[Q_t^2|\cF_{t-1/2}]\cdot p}.
\end{align*}
Combining the two inequalities above and using Ineq.~\eqref{bound_second_moment_Q_t} yields
\begin{align}\label{lower_bound_p}
p \geq \tfrac{9 q_t^2}{16\bbe[Q_t^2|\cF_{t-1/2}]} \geq \tfrac{9Bq_t}{16(\hat\cL^2+ Bq_t)} \geq \tfrac{9}{32} \min\{1, \tfrac{Bq_t}{\hat\cL^2} \}.
\end{align}
On the other hand, recall that the variance of $Q_t$ can be bounded by $\tfrac{\hat \cL^2 q_t}{B}$. By using Chebyshev inequality, we can bound
\begin{align}\label{upper_bound_P}
\mathbb{P}(Q_t < \tfrac{q_t}{16}|\cF_{t-1/2}) \leq \mathbb{P}(|Q_t - q_t| > \tfrac{15q_t}{16}|\cF_{t-1/2}) \leq \min\{ 1, \tfrac{256 \hat \cL^2}{225B q_t}\}. 
\end{align}
Then we have 
\begin{align}\label{bound_with_hat_cL}
\|x_t - x_{t-1}\|^2\cdot\bbe\left[ \left(q_t- 16 \hat \cL_{t}^2\right)_+ | \cF_{t-1/2}\right] &\overset{(i)}\leq \|x_t - x_{t-1}\|^2\cdot\bbe\left[ q_t \mathbf{1}_{\{Q_t < \tfrac{q_t}{16}\}} | \cF_{t-1/2}\right]\nn\\
&=q_t\|x_t - x_{t-1}\|^2\mathbb{P}(Q_t < \tfrac{q_t}{16}| \cF_{t-1/2})\nn\\
& \overset{(ii)}\leq  \tfrac{256}{225}q_t\|x_t - x_{t-1}\|^2 \min\{1, \tfrac{\hat\cL^2}{Bq_t} \}\nn\\
&\overset{(iii)}\leq \tfrac{256}{225}q_t\|x_t - x_{t-1}\|^2  \cdot \tfrac{32p}{9}\cdot\tfrac{\min\{1, \frac{\hat\cL^2}{Bq_t} \}}{\min\{1, \frac{Bq_t}{\hat\cL^2} \}} \nn\\
&\leq \tfrac{4.05\hat \cL^2\|x_t - x_{t-1}\|^2 p}{B},
\end{align}
where step (ii) follows from Ineq.~\eqref{upper_bound_P}, and step (iii) follows from Ineq.~\eqref{lower_bound_p}. 

Next, we establish the bound using the stochastic variance $\sigma^2$. Define $v_t := \bbe[\|\cG(x_{t-1}, \xi) - \cG(x_t, \xi) - \nabla f(x_{t-1}) + \nabla f(x_t)\|^2|\cF_{t-1/2}]$, thus $v_t \leq 4\sigma^2$. First, consider the case $v_t \geq \tfrac{q_t\|x_t - x_{t-1}\|^2}{2}$. By using step (ii) in Ineq.~\eqref{bound_with_hat_cL}, we have
\begin{align}\label{case_1_bound}
\|x_t - x_{t-1}\|^2\cdot\bbe\left[ \left(q_t- 16\hat \cL_{t}^2\right)_+ | \cF_{t-1/2}\right] \leq \tfrac{512}{225}v_t \min\{1, \tfrac{\hat\cL^2}{Bq_t} \} \leq \tfrac{2048}{225}\sigma^2 \min\{1, \tfrac{\hat\cL^2}{Bq_t} \}.
\end{align}
Now consider the case $v_t < \tfrac{q_t\|x_t - x_{t-1}\|^2}{2}$. The definition of $v_t$ also suggests that $v_t = q_t \|x_t - x_{t-1}\|^2 - \|\nabla f(x_{t-1}) - \nabla f(x_t)\|^2$, thus we have
\begin{align*}
\|\nabla f(x_{t-1}) - \nabla f(x_t)\|^2 > \tfrac{q_t \|x_t - x_{t-1}\|^2}{2}.
\end{align*}
Define $\bar G_t(x) = \tfrac{\sum_{i=1}^B\cG(x, \bar \xi_i^t)}{B}$. When $Q_t < q_t/16$, by using Young's inequality, we have
\begin{align*}
\|\bar G_t(x_{t-1}) - \bar G_t(x_{t})\|^2 \leq \tfrac{1}{B}\tsum_{i=1}^B \|\cG(x_{t-1}, \bar \xi_i^t) - \cG(x_t, \bar \xi_i^t)\|^2 = Q_t \|x_t - x_{t-1}\|^2 < \tfrac{q_t \|x_t - x_{t-1}\|^2}{16}.
\end{align*}
By combining the two inequalities above, we have
\begin{align*}
\|\bar G_t(x_{t-1}) - \bar G_t(x_{t}) - \nabla f(x_{t-1}) + \nabla f(x_t)\| &\geq   \|\nabla f(x_{t-1}) - \nabla f(x_t)\| - \|\bar G_t(x_{t-1}) - \bar G_t(x_{t})\|\\
&\geq (\tfrac{1}{\sqrt{2}} - \tfrac{1}{4})\sqrt{q_t}\|x_t - x_{t-1}\|,
\end{align*}
which indicates that $\bbe\left[q_t \|x_t - x_{t-1}\|^2 \mathbf{1}_{\{Q_t < \tfrac{q_t}{16}\}}| \cF_{t-1/2}\right]  \leq \tfrac{\bbe[\|\bar G_t(x_{t-1}) - \bar G_t(x_{t}) - \nabla f(x_{t-1}) + \nabla f(x_t)\|^2 |\cF_{t-1/2} ]}{(1/\sqrt{2}-1/4)^{2}} \leq \tfrac{4.79 v_t}{B}$.
Then, step (i) in Ineq.~\eqref{bound_with_hat_cL} implies
\begin{align}\label{case_2_bound}
\|x_t - x_{t-1}\|^2\cdot\bbe\left[ \left(q_t- 16 \hat \cL_{t}^2\right)_+ | \cF_{t-1/2}\right] 
&\leq \tfrac{4.79 v_t}{B} \leq \tfrac{19.16 \sigma^2}{B}\leq 19.16\sigma^2 \min\{1, \tfrac{\hat\cL^2}{Bq_t} \},
\end{align}
where the last inequality follows from the fact that $q_t\leq \hat\cL^2$. By combining the two cases in Ineqs.~\eqref{case_1_bound} and \eqref{case_2_bound}, and using Ineq.~\eqref{lower_bound_p} and $q_t > 16\hat \cL_{t-1}^2$, we obtain 
\begin{align*}
\|x_t - x_{t-1}\|^2\cdot\bbe\left[ \left(q_t- 16 \hat \cL_{t}^2\right)_+ | \cF_{t-1/2}\right]\leq 19.16 \sigma^2 \min\{1, \tfrac{\hat\cL^2}{Bq_t} \} \leq \tfrac{68.13\hat \cL^2 \sigma^2 p }{Bq_t}  < \tfrac{4.26\hat \cL^2  \sigma^2 p }{B\hat \cL_{t-1}^2},
\end{align*} which, together with Ineq.~\eqref{bound_with_hat_cL}, completes the proof.
\end{proof}
\vgap

The next lemma characterizes the accumulated expected squared error of the operator $\tilde G_{t}$.
\begin{lemma}\label{lemma:bound_variance_summation}
	Assume Assumptions \ref{assump_unbiasness}-\ref{assump_smooth_stoch} hold and define $\tilde \delta_t := \tilde G_t - \nabla f(x_{t-1})$. We have
	\begin{align}\label{bound_variance_summation}
		\bbe\left[\tsum_{t=1}^k t \|\tilde \delta_t\|^2 \right] \leq \tfrac{k (k+1)\sigma^2}{2 N } + \tsum_{t=1}^{k-1} \mathbf{1}_{\{\mathrm{mod}(t, T) \neq 0\}} T^2 \lceil\tfrac{t}{T} \rceil \bbe\left[\tfrac{q_t \|x_t - x_{t-1}\|^2}{b_{t+1}}\right].
	\end{align}
\end{lemma}
\begin{proof}
	The following inequalities hold:
	\begin{align*}
		\bbe\left[\tsum_{t=1}^k t\|\tilde \delta_t\|^2 \right] &\leq  \tsum_{t=1}^k t \cdot  \bbe\|\tilde \delta_t\|^2  \\
		&\overset{(i)} \leq \tsum_{t=1}^k \left[\tfrac{t \sigma^2}{N} +  t\tsum_{j=\lfloor \frac{t-1}{T}\rfloor \cdot T + 2}^{t }\bbe\left[  \tfrac{q_{j-1} \|x_{j-1} - x_{j-2}\|^2}{b_j}\right] \right]\\
		& \overset{(ii)} \leq \tfrac{k (k+1) \sigma^2}{2 N } + \tsum_{t=1}^{k-1} \mathbf{1}_{\{\mathrm{mod}(t, T) \neq 0\}} T \lceil \tfrac{t}{T}\rceil T  \bbe\left[\tfrac{q_t\|x_t - x_{t-1}\|^2}{b_{t+1}}\right],
	\end{align*}
where step (i) follows from Lemma~\ref{bound_error_variance}, and step (ii) follows from counting the terms in the summation.
	\end{proof}
\vgap

The main convergence properties of AC-VR-SPG (Algorithm~\ref{alg_ac_vr}) are described in the following theorem. 
\begin{theorem}\label{theorem:ac-vr-sgm}
	Assume Assumptions \ref{assump_unbiasness}-\ref{assump_smooth_stoch} hold. Define $\cA :=\min \left\{52.02\hat \cL^2 D_X^2 \log_{\frac{3}{2}}\tfrac{\hat \cL}{\bar \cL_0},~ \tfrac{1.92\hat \cL^2 \sigma^2}{\bar \cL_0^2} + 816 \sigma^2 \log_{\frac{3}{2}}\tfrac{\hat \cL}{\bar \cL_0}\right\}$. If we choose $b_{t+1} = \lceil \tfrac{8\lceil t/T\rceil T^2 \gamma_t}{\Gamma_t} \rceil$, then
	\begin{align}\label{ac-vr-sgm_1}
	\bbe[\tfrac{M_k}{32} \|g_{X, R(k)}\|^2] \leq \big(102 + 39\log(k+1) \big)\hat \cL^2D_X^2 + 6kl\hat\cL D_X^2 + \tfrac{k (k+1)\sigma^2}{N }+ \tfrac{ k \cA}{B},
	\end{align}
and consequently
\begin{align}\label{ac-vr-sgm_2}
	\bbe[ \|g_{X, R(k)}\|^2] \leq 96\left(1+ 2 \lfloor \log_{\frac{3}{2}} \tfrac{\hat \cL}{\bar \cL_0}\rfloor\right) \left[\tfrac{(102 + 39\log(k+1))\hat\cL^2 D_X^2}{k^2} + \tfrac{6\hat \cL lD_X^2}{k} + \tfrac{ (k+1) \sigma^2}{ k N} + \tfrac{\cA}{B k }\right].
\end{align}
\end{theorem}
\begin{proof}
First, note that Ineq.~\eqref{ac-spg_main_recursion} in Proposition~\ref{prop:ac-spg_main_recursion} is still valid for AC-VR-SPG, with $\hat\delta_t(x_{t-1})$ replaced by $\tilde \delta_t$. 
	By setting $\rho = \tfrac{1}{16}$, we have 
	\begin{align*}
		&\tsum_{t=1}^{k}\tfrac{\Gamma_t}{32\gamma_t}\|g_{X, t-1}\|^2 + \tsum_{t=1}^k \tfrac{\Gamma_t\gamma_t}{16}\|x_{t-1} - x_t\|^2\nn\\
		& \leq 94 \hat \cL_{k}^2D_X^2 + 6kl\hat\cL_{k}D_X^2 + \tsum_{t=1}^{k} \tfrac{17\Gamma_t}{16\gamma_t}\| \tilde \delta_t\|^2\nn\\
		&\quad+ \tsum_{t=1}^{k}  \gamma_t\langle \tilde \delta_t,  x^* - x_{t-1}\rangle + \tsum_{t=1}^{k} \Gamma_{t}(\tfrac{L_{t}}{2}- \hat \cL_{t})\|x_{t-1} - x_{t}\|^2. 
	\end{align*}
By using the facts that  $ \Gamma_t/\gamma_t \leq t$ and
\begin{align*}
\tsum_{t=2}^{k}  \gamma_t\langle \tilde\delta_t,  x^* - x_{t-1}\rangle  \leq \tsum_{t=2}^k \tfrac{4 \gamma_t^2}{15 t }\|x^* - x_{t-1}\|^2 + \tsum_{t=2}^k \tfrac{15 t }{16 }\|\tilde \delta_t\|^2 \leq 39 \log k \cdot \hat \cL^2 D_X^2+ \tsum_{t=1}^k \tfrac{15 t }{16 }\|\tilde \delta_t\|^2 ,
\end{align*}
we have
\begin{align*}
	&\tsum_{t=1}^{k}\tfrac{\Gamma_t}{32\gamma_t}\|g_{X, t-1}\|^2 + \tsum_{t=1}^k \tfrac{\Gamma_t\gamma_t}{16}\|x_{t-1} - x_t\|^2\nn\\
	& \leq (94 + 39\log k)\hat \cL^2D_X^2 + 6kl\hat\cL D_X^2 + \gamma_1\langle \tilde\delta_1,  x^* - x_{0}\rangle + \tsum_{t=1}^{k} 2t\| \tilde \delta_t\|^2+ \tsum_{t=1}^{k} \Gamma_{t}(\tfrac{L_{t}}{2}- \hat \cL_{t})\|x_{t-1} - x_{t}\|^2. 
\end{align*}
Taking expectation on both sides of the above inequality, and using Lemma~\ref{lemma:bound_variance_summation} and the batch size $b_{t+1} = \lceil \tfrac{8\lceil t/T\rceil T^2 \gamma_t}{\Gamma_t} \rceil$, we obtain
\begin{align}\label{ac_vr_sgm_r123}
	&\bbe\left[\tsum_{t=1}^{k}\tfrac{\Gamma_t}{32\gamma_t}\|g_{X, t-1}\|^2\right] + \bbe\left[\tsum_{t=1}^k \tfrac{\Gamma_t\gamma_t}{16}\|x_{t-1} - x_t\|^2\right]\nn\\
	& \leq (94 + 39\log k)\hat \cL^2D_X^2 + 6kl\hat\cL D_X^2 + \tfrac{k (k+1)\sigma^2}{ N } + \tfrac{1}{4}\bbe\left[\tsum_{t=1}^{k}  \tfrac{\Gamma_tq_t }{\gamma_t}\|x_t - x_{t-1}\|^2\right]\nn\\
	&\qquad+ \bbe\left[\tsum_{t=1}^{k} \Gamma_{t}(\tfrac{L_{t}}{2}- \hat \cL_{t})\|x_{t-1} - x_{t}\|^2\right]\nn\\
	& \leq (94 + 39\log k )\hat \cL^2D_X^2 + 6kl\hat\cL D_X^2 + \tfrac{k (k+1)\sigma^2}{N } + \underbrace{\tfrac{1}{4}\bbe\left[\tsum_{t=1}^{k}  \tfrac{\Gamma_t}{\gamma_t}(q_t - 16 \hat \cL_t^2)_+\|x_t - x_{t-1}\|^2\right]}_{=: R_1} \nn\\
	&\qquad+ \underbrace{4 \bbe\left[\tsum_{t=1}^{k}  \tfrac{\Gamma_t\hat \cL_t^2 }{\gamma_t}\|x_t - x_{t-1}\|^2\right]}_{=:R_2}+\underbrace{ \bbe\left[\tsum_{t=1}^{k} \Gamma_{t}(\tfrac{L_{t}}{2}- \hat \cL_{t})\|x_{t-1} - x_{t}\|^2\right]}_{=: R_3}. 
\end{align}
It remains to upper bound $R_1$, $R_2$ and $R_3$. For $R_1$, we have
\begin{align*}
	R_1 &\overset{(i)}\leq \tfrac{k}{4}\bbe\left[\tsum_{t=1}^{k} (q_t - 16 \hat \cL^2_t)_+\|x_t - x_{t-1}\|^2\right]\\
	& \overset{(ii)}\leq \tfrac{k \hat \cL^2 }{4B} \bbe\left[\tsum_{t=1}^{k} \min \left\{4.05\|x_t - x_{t-1}\|^2, \tfrac{4.26\sigma^2}{\hat \cL_{t-1}^2} \right\}\cdot\mathbf{1}(\hat \cL_t \geq \tfrac{3}{2}\hat \cL_{t-1})\right]\\
	& \overset{(iii)}\leq \tfrac{k \hat \cL^2 }{4B} \min \left\{\bbe\left[\tsum_{t=1}^{k}4.05D_X^2 \cdot\mathbf{1}(\hat \cL_t \geq \tfrac{3}{2}\hat \cL_{t-1})\right],~ \bbe\left[\tsum_{t=1}^{k}\tfrac{4.26\sigma^2}{\hat \cL_{t-1}^2}\cdot\mathbf{1}(\hat \cL_t \geq \tfrac{3}{2}\hat \cL_{t-1})\right]  \right\}\\
	& \overset{(iv)}\leq \tfrac{k \hat \cL^2 }{4B} \min \left\{4.05D_X^2 \log_{\frac{3}{2}}\tfrac{\hat \cL}{\bar \cL_0}, \tfrac{7.67\sigma^2}{\bar \cL_0^2} \right\},
\end{align*}
where step (i) follows from $\Gamma_t/\gamma_t \leq t \leq k$, step (i) follows from Lemma~\ref{lemma:squared_gradient_difference}, step (iii) follows from $\|x_t - x_{t-1}\|^2\leq D_X^2$ and exchanging minimization and submission, and step (iv) follows from the fact that for any realized trajectory, $\tsum_{t=1}^{k}\mathbf{1}(\hat \cL_t \geq \tfrac{3}{2}\hat \cL_{t-1}) \leq \log_{\frac{3}{2}}\tfrac{\hat \cL}{\bar \cL_0}$, and that $\hat \cL_{t-1}^2$ is growing with ratio greater than $\tfrac{9}{4}$ over each iteration with $\hat \cL_t \geq \tfrac{3}{2}\hat \cL_{t-1}$. 

For $R_2$, we can bound
\begin{align*}
	R_2 &=4 \bbe\left[\tsum_{t=1}^{k}  \tfrac{\Gamma_t\hat \cL_t^2 }{\gamma_t}\|x_t - x_{t-1}\|^2 \mathbf{1}({\hat \cL_t \leq \tfrac{3}{2}\hat \cL_{t-1}})\right] + 4 \bbe\left[\tsum_{t=1}^{k}  \tfrac{\Gamma_t\hat \cL_t^2 }{\gamma_t}\|x_t - x_{t-1}\|^2 \mathbf{1}({\hat \cL_t >\tfrac{3}{2}\hat \cL_{t-1}})\right]\nn\\
	& \overset{(i)}\leq \tfrac{1}{16} \bbe\left[\tsum_{t=1}^{k}  \Gamma_t\gamma_t\|x_t - x_{t-1}\|^2 \mathbf{1}({\hat \cL_t \leq \tfrac{3}{2}\hat \cL_{t-1}})\right]  +  4 \bbe\left[\tsum_{t=1}^{k}  \hat \cL_t^2 \|x_t - x_{t-1}\|^2 \mathbf{1}({\hat \cL_t >\tfrac{3}{2}\hat \cL_{t-1}})\right]\\
	& \overset{(ii)}\leq  \tfrac{1}{16} \bbe\left[\tsum_{t=1}^{k}  \Gamma_t\gamma_t\|x_t - x_{t-1}\|^2 \right]  + \tfrac{36}{5}\hat \cL^2 D_X^2,
\end{align*}
where step (i) follows from $\hat \cL_t \leq \tfrac{3}{2}\hat \cL_{t-1} = \tfrac{\gamma_t}{8}$ when $\hat \cL_t \leq \tfrac{3}{2}\hat \cL_{t-1}$, and $\Gamma_t = \gamma_t$ when $\hat \cL_t >\tfrac{3}{2}\hat \cL_{t-1}$, step (ii) follows from the fact that $\hat \cL_t^2$ is growing with ratio greater than $\tfrac{9}{4}$ over each iteration with $\hat \cL_t >\tfrac{3}{2}\hat \cL_{t-1}$.

Finally, we bound $R_3$ by
\begin{align*}
	R_3 &\overset{(i)}\leq \bbe\left[\tsum_{t=1}^{k} 12t \hat \cL_{t-1}(\tfrac{L_{t}}{2}- \hat \cL_{t})_+\|x_{t-1} - x_{t}\|^2\right]\\
	& \overset{(ii)}\leq \bbe\left[\tsum_{t=1}^{k} \tfrac{51 t}{B} \min \{16\sigma^2, \hat \cL^2 \|x_t - x_{t-1}\|^2\} \mathbf{1}(\hat \cL_t > \tfrac{3}{2}\hat \cL_{t-1})\right]\\
	& \overset{(iii)}\leq \tfrac{51 k}{B} \log_{\frac{3}{2}}\tfrac{\hat \cL}{\bar \cL_0}\min \{16\sigma^2, \hat \cL^2 D_X^2\}.
\end{align*}
where step (i) follows from $\Gamma_t \leq t \gamma_t = 12 t \hat \cL_{t-1}$, step (ii) follows from Lemma~\ref{lemma:bound_underestimation_ac_vr_spg}, and step (iii) follows from $\|x_t - x_{t-1}\|^2 \leq D_X^2$ and $\tsum_{t=1}^{k}\mathbf{1}(\hat \cL_t \geq \tfrac{3}{2}\hat \cL_{t-1}) \leq \log_{\frac{3}{2}}\tfrac{\hat \cL}{\bar \cL_0}$ for any realized trajectory. 

By substituting the above upper bounds for $R_1$, $R_2$, and $R_3$ into Ineq.~\eqref{ac_vr_sgm_r123}, we obtain
\begin{align}\label{ac_vr_sgm_bound}
	&\bbe\left[\tsum_{t=1}^{k}\tfrac{\Gamma_t}{32\gamma_t}\|g_{X, t-1}\|^2\right] \nn\\
	& \leq (94 + 39\log(k+1) + \tfrac{36}{5})\hat \cL^2D_X^2 + 6kl\hat\cL D_X^2 + \tfrac{k (k+1)\sigma^2}{N } +  \tfrac{k \hat \cL^2 }{4B} \min \left\{4.05D_X^2 \log_{\frac{3}{2}}\tfrac{\hat \cL}{\bar \cL_0}, \tfrac{7.67\sigma^2}{\bar \cL_0^2} \right\}\nn\\
	&\quad+ \tfrac{51 k}{B} \log_{\frac{3}{2}}\tfrac{\hat \cL}{\bar \cL_0}\min \{16\sigma^2, \hat \cL^2 D_X^2\}\nn\\
	& \leq (102 + 39\log(k+1) )\hat \cL^2D_X^2 + 6kl\hat\cL D_X^2 + \tfrac{k (k+1)\sigma^2}{N }\nn\\
	&\quad+  \tfrac{k }{B} \min \left\{52.02\hat \cL^2 D_X^2 \log_{\frac{3}{2}}\tfrac{\hat \cL}{\bar \cL_0},~ \tfrac{1.92\hat \cL^2 \sigma^2}{\bar \cL_0^2} + 816 \sigma^2 \log_{\frac{3}{2}}\tfrac{\hat \cL}{\bar \cL_0}\right\}. 
\end{align}
Recall the definition $M_k:=\tsum_{t=1}^k \tfrac{\Gamma_t}{\gamma_t}$, the LHS of Ineq.~\eqref{ac_vr_sgm_bound} satisfies
\begin{align*}
	\bbe\left[\tsum_{t=1}^{k}\tfrac{\Gamma_t}{32\gamma_t}\|g_{X, t-1}\|^2 \right]= \bbe\left[\left(\tsum_{t=1}^{k}\tfrac{\Gamma_t}{32\gamma_t}\right)\|g_{X, R(k)}\|^2 \right] = \bbe\left[\tfrac{M_k}{32}\|g_{X, R(k)}\|^2 \right].
\end{align*}
Combining the two inequalities above yields Ineq.~\eqref{ac-vr-sgm_1}. By further using \eqref{eq:lower_bound_M_k} in Lemma~\ref{lemma_bound_M_k}, we obtain Ineq.~\eqref{ac-vr-sgm_2}.
\end{proof}
\vgap

We now add a few remarks on the complexity bounds of AC-VR-SPG. Specifically, we choose the batch sizes $N = T^2$ and $B=T$, where $T$ is the epoch length, then Ineq.~\eqref{ac-vr-sgm_2} implies
\begin{align}\label{ac-vr-spg-3}
	\bbe[ \|g_{X, R(k)}\|^2] \leq \widetilde{O} \left(\tfrac{\hat\cL^2 D_X^2}{k^2} + \tfrac{\hat \cL lD_X^2}{k} + \tfrac{ \sigma^2}{ T^2} + \tfrac{\cA}{T k }\right).
\end{align}
In order to find an $\epsilon$-stationary point, we set $T$ on the order of $\max\{1, \tfrac{\sigma}{\epsilon}\}$. This requires estimating the variance $\sigma^2$, and in practice we can obtain empirical estimates of $\sigma^2$ by running AC-VR-SPG for a few iterations. As a consequence, the total number of iterations $k$ required by AC-VR-SPG to find an $\epsilon$-approximate solution is bounded by
\begin{align*}
\widetilde\O\left(\tfrac{\hat \cL D_X}{\epsilon} + \tfrac{\hat \cL l D_X^2}{\epsilon^2} + \tfrac{\cA}{\epsilon^2 \max\{1, \sigma/\epsilon\}}\right) \leq \widetilde\O\left(\tfrac{\hat\cL^2 D_X}{\bar \cL_0\epsilon } + \tfrac{\hat \cL l D_X^2}{\epsilon^2} \right),
\end{align*}
where the inequality follows from the bound $\cA \leq \widetilde\O(\tfrac{\hat\cL^2 D_X \sigma}{\bar \cL_0})$. Notice that in the deterministic setting with $\sigma = 0$, we have $\cA=0$, so that the iteration complexity reduces to $\widetilde\O\left(\tfrac{\hat \cL D_X}{\epsilon} + \tfrac{\hat \cL l D_X^2}{\epsilon^2}\right)$, which matches the iteration complexity of the deterministic AC-PG method (Algorithm~\ref{alg_1}), up to logarithmic factors.
Counting the total batch sizes of $N$, $B$, and $b_t$, and invoking Ineq.~\eqref{eq:bound_gamma_t_Gamma_t} in Lemma~\ref{lemma_bound_M_k}, we can bound the total sample complexity $\O\left(N \lceil\tfrac{k}{T}\rceil + \tsum_{t=1}^k b_t + kB\right)$ by
\begin{align*}
 \widetilde{\O}\left(\tfrac{\hat \cL D_X}{\epsilon} + \tfrac{\hat \cL l D_X^2}{\epsilon^2}  + \tfrac{\sigma^2}{\epsilon^2} + \tfrac{\hat \cL D_X \sigma}{\epsilon^2} + \tfrac{\hat\cL l D_X^2 \sigma}{\epsilon^3} + \tfrac{\cA}{\epsilon^2 }\right) \leq \widetilde{\O}\left(\tfrac{\hat \cL D_X}{\epsilon} + \tfrac{\hat \cL l D_X^2}{\epsilon^2}  + \tfrac{\sigma^2}{\epsilon^2} + \tfrac{\hat \cL^2 D_X \sigma}{\bar \cL_0\epsilon^2} + \tfrac{\hat\cL l D_X^2 \sigma}{\epsilon^3} \right).
\end{align*}
This is a unified sample complexity bound that nearly matches the sample complexity of VR-SPG (Algorithm~\ref{alg_vr}) up to logarithmic factors, except for the only polynomial overhead $\tfrac{\hat \cL}{\bar \cL_0}$ lying in a higher order term $\tfrac{\hat \cL^2 D_X \sigma}{\bar \cL_0\epsilon^2}$ due to the bound $\cA \leq \widetilde\O(\tfrac{\hat\cL^2 D_X \sigma}{\bar \cL_0})$.

\section{Numerical experiments}\label{sec_numerics}
In this section, we provide two sets of numerical experiments to demonstrate the empirical performance of the proposed algorithms. The major purpose of conducting the experiments is to show the effectiveness and advantage of the auto-conditioning and variance reduction techniques.
\subsection{Box-constrained nonconvex quadratic programming}
The first set of experiments is conducted by applying the PG method in Algorithm~\ref{alg_1_0} and the AC-PG method in Algorithm~\ref{alg_1} to solve the box-constrained quadratic program
\begin{equation}
\min_{x\in\mathbb{R}^n} f(x):=\frac{1}{2}x^\top Q x + c^\top x, ~~\text{s.t.}~~ x_i \in [a_i, b_i], \forall\, i = 1,\ldots, n,    
\end{equation}
where $Q\in\mathbb{R}^{n\times n}$ is a symmetric indefinite matrix. In the experiments, we generate $Q= \frac{1}{2}(\tilde Q + \tilde Q^\top)$, where components of $\tilde Q$ follow independent and identically  standard Gaussian distribution; $c$ is generated by standard Gaussian distribution as well. For the PG method, we set $\gamma_t = \|Q\|, \forall\, t$, and for the AC-PG, we set the initial estimate of Lipschitz constant to $L_0 = \theta \|Q\|$ with $\theta\in \{0.1, 0.2, 0.5, 0.001\}$. \revision{}{To avoid potential computational issues, we add a tiny perturbation when estimating the local Lipschitz constant $L_t$, i.e., $L_t = \tfrac{2(f(x_{t}) - f(x_{t-1}) - \langle \nabla f(x_{t-1}), x_t - x_{t-1}\rangle)}{\|x_t - x_{t-1}\|^2 + 10^{-10}}.$ We apply similar treatments to experiments in the next subsection.} Also, we set $n=100$ and $a_i = -5$ and $b_i = 5$ for each $i=1,\ldots,n$. \revision{}{Notice that without the compact box constraint, the problem may have an unbounded objective because $Q$ is indefinite.} Ten independent trials are run. The average and standard deviation results of the norm of the projected gradient mapping at each iterate are plotted in the left of Figure~\ref{fig:qp}. The average values and standard deviation of the estimated local Lipschitz constant are shown in the right of Figure~\ref{fig:qp}. The results clearly show the effectiveness of the proposed AC-PG method and advantage over the classic PG method. The very small number of $\theta=0.001$ is used to demonstrate that even if starting from a significant underestimate of the Lipschitz constant, AC-PG can automatically adjust the estimate and quickly converge to a stationary solution.

\begin{figure}[ht]
    \centering
\begin{tabular}{cc}
    \includegraphics[width=0.45\linewidth]{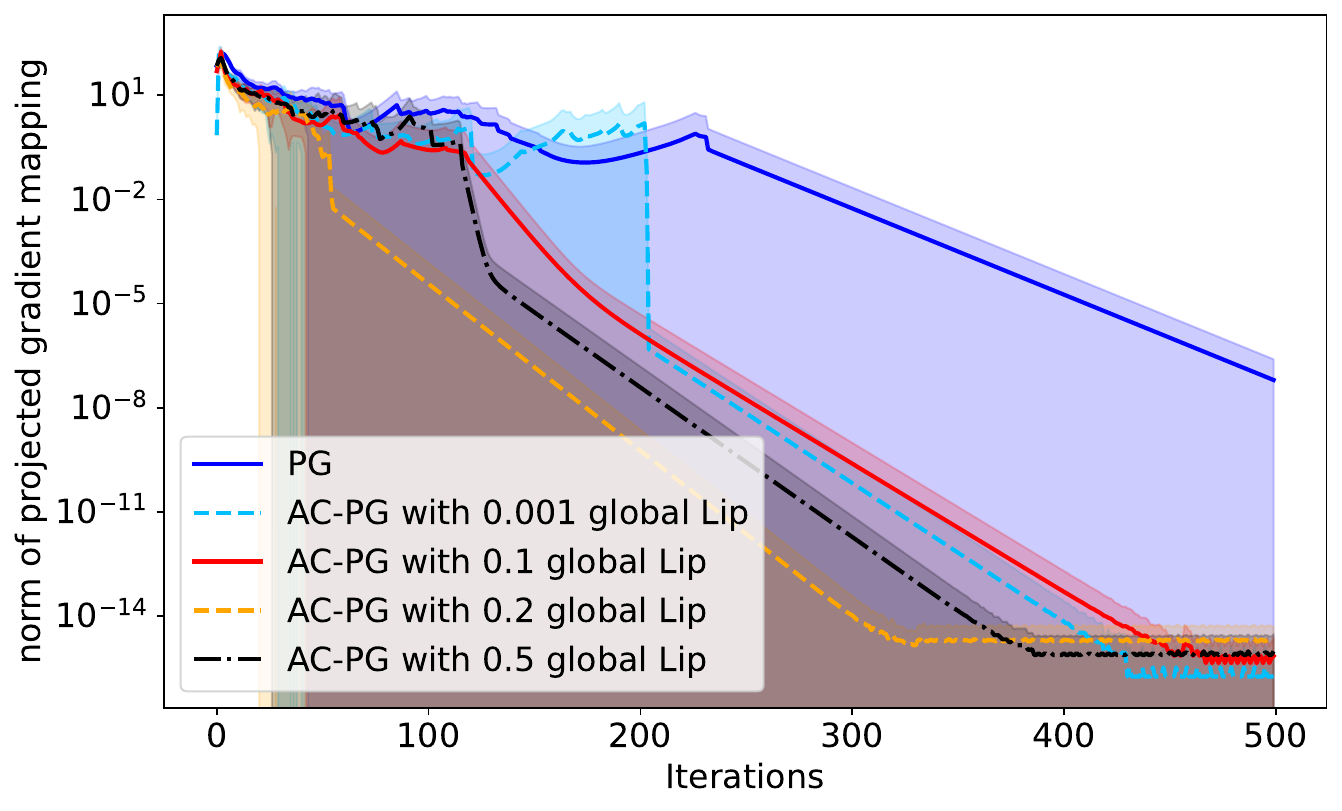} &
    \includegraphics[width=0.435\linewidth]{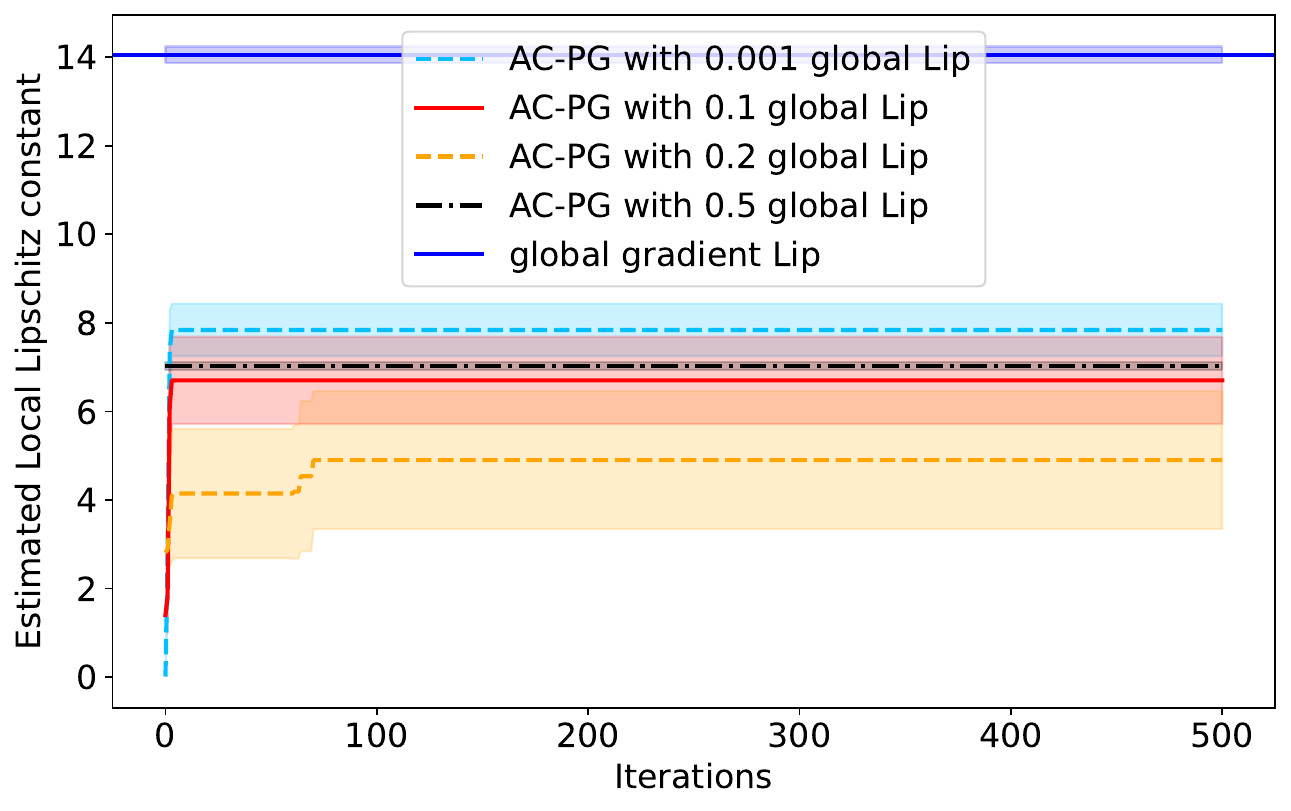}
\end{tabular}     
    \caption{Average and standard deviation results by the PG and AC-PG methods on solving 10 independently generated instances of box-constrained quadratic programming. Left: norm of projected gradient mapping at each iteration; Right: estimated local gradient Lipschitz constants by AC-PG starting from different initial estimates.}
    \label{fig:qp}
\end{figure}

\subsection{Semisupervised smoothed support vector machine}
The second set of experiments is conducted by applying the SPG method in Algorithm~\ref{alg_2_0}, the AC-SPG method in Algorithm~\ref{alg_2}, the VR-SPG in Algorithm~\ref{alg_vr}, and the VR-SPG in Algorithm~\ref{alg_ac_vr} on solving the semisupervised smoothed support vector machine (SVM) \cite{ChaSinKee08-1, ghadimi2016mini}, formulated as
\begin{equation}\label{eq:svm}
\min_{x \in\cX, b\in \mathcal{B}} f(x,b):=\lambda_1 \mathbb{E}_{u_1, v}\left[\max\left\{0, 1- v\big(u_1^\top x + b\big)\right\}^2\right] + \lambda_2 \mathbb{E}_{u_2} \exp\left\{-5\big(u_2^\top x+b\big)^2\right\} + \frac{\lambda_3}{2}\|x\|^2. 
\end{equation}
Here, $\cX\subseteq \mathbb{R}^{n}$, $\mathcal{B}\subseteq \mathbb{R}$, and $v\in \{+1, -1\}$ is a label corresponding to the sample $u_1$. In the experiments, we set $n\in\{10,100\}$, $\cX=\{x: \|x\|\le 10\}$ 
and $\mathcal{B} = [-2,2]$, and we choose $\lambda_1=\lambda_2=0.5$ and $\lambda_3=1$. \revision{}{The explicit inclusion of the two sets $\mathcal{X}$ and $\mathcal{B}$ is to ensure that the instance satisfies the assumption of compact domain in our analysis, but it does not lose generality here. Let us denote $(x^*,b^*)$ as the minimizer of the objective of \eqref{eq:svm} over $\mathbb{R}^{n+1}$. With our choice of $\lambda_1, \lambda_2$ and $\lambda_3$, it holds $\frac{1}{2}\|x^*\|^2 \le f(x^*,b^*) \le f(0,0) = 1$. Thus $\|x^*\| \le \sqrt{2}$, i.e., $x^*$ is finite. Suppose $u_1$ is a unit random vector. Then $|u_1^T x^*| \le \|x^*\|$ and $$\max\left\{0, 1- v\big(u_1^\top x^* + b\big)\right\}^2 \ge 
\max\big\{0, 1-\|x^*\| - vb\big\}^2.$$ Let $p_1 = \mathbb{P}(v = +1) > 0$ and $p_2 = \mathbb{P}(v = -1) > 0$. Then the first expectation term at $(x^*,b^*)$ in \eqref{eq:svm} is lower bounded by $\lambda_1 p_1 \max\big\{0, 1-\|x^*\| - b^*\big\}^2 + \lambda_1 p_2 \max\big\{0, 1-\|x^*\| + b^*\big\}^2$. Hence, with $\lambda_1=\lambda_2=0.5$, it holds
$$p_1 \max\big\{0, 1-\|x^*\| - b^*\big\}^2 + p_2 \max\big\{0, 1-\|x^*\| + b^*\big\}^2 \le \frac{1}{\lambda_1} f(x^*,b^*) \le 2,$$
which indicates $1-\|x^*\|  - \sqrt{\frac{2}{p_1}} \le b^* \le \sqrt{\frac{2}{p_2}} - 1+ \|x^*\|$. Therefore, $b^*$ falls in a bounded interval, and if $\|x^*\| \le 1$ and $p_1=p_2=0.5$, then $b^*\in[-2,2]$.}

We adopt two ways to generate data samples. For the first way, we pre-generate a big data set of $M$ samples for $u_1$ and another big data set of $M$ samples for $u_2$, and we set $v=\mathrm{sign}(u_1^\top \bar x + \bar b)$ for each sample $u_1$ in the generated data set, where $\bar x$ and $\bar b$ are generated randomly by the standard Gaussian distribution. Each sample in the two data sets is generated by $ u / \|u\|$ with $u$ following the standard multivariable Gaussian distribution. In this way, the two expectations in \eqref{eq:svm} can be written as the averages with respect to the two finite data sets. We set $M=2\times 10^5$. Due to the data normalization, the global gradient Lipschitz constant of \eqref{eq:svm} can be explicitly computed as $L = 8\lambda_1 + 40\lambda_2(1+e^{-1}) + \lambda_3$. Each method runs to $k=10^3$ iterations. For SPG and VR-SPG, we set $\gamma_t = 2L, \forall\, t$, and for AC-SPG and AC-VR-SPG, we set $\hat L_0 = \theta L$ and $\gamma_t = 3 \hat L_{t-1}$ for all $t$ with $\theta\in \{0.1, 0.2, 0.5, 0.001\}$. In addition, at each iteration of SPG and AC-SPG, we pick $25k$ samples without replacement for $u_1$ and $u_2$ uniformly at random from the pre-generated data sets; for VR-SPG and AC-VR-SPG, we set $T=10$, $N=M$, and $b_t = 5,000$ for each $t$ such that $\mathrm{mod}(t, T) \neq 1$. 
Notice that with this choice, the total samples used by VR-SPG and AC-VR-SPG within $k$ iterations are slightly fewer than those by SPG and AC-SPG. \revision{}{For AC-SPG, we apply Estimator I in \eqref{evaluate_L_t_2} for local Lipschitz estimation.} We run 10 independent trials and report results in Figure~\ref{fig:SVM-det} for the dimension $n=10$ and in Figure~\ref{fig:SVM-det-n100} for the dimension $n=100$, where the projected mapping is computed by using all $M$ samples. The results demonstrate the advantage of variance reduction and the effectiveness of auto conditioning. Again, we see that the auto-conditioned methods starting from a tiny initial estimate (with $\theta=0.001$) of Lipschitz constant can automatically adjust the estimates and converge to a stationary solution. 

\begin{figure}[ht]
    \centering
    \includegraphics[width=0.45\linewidth]{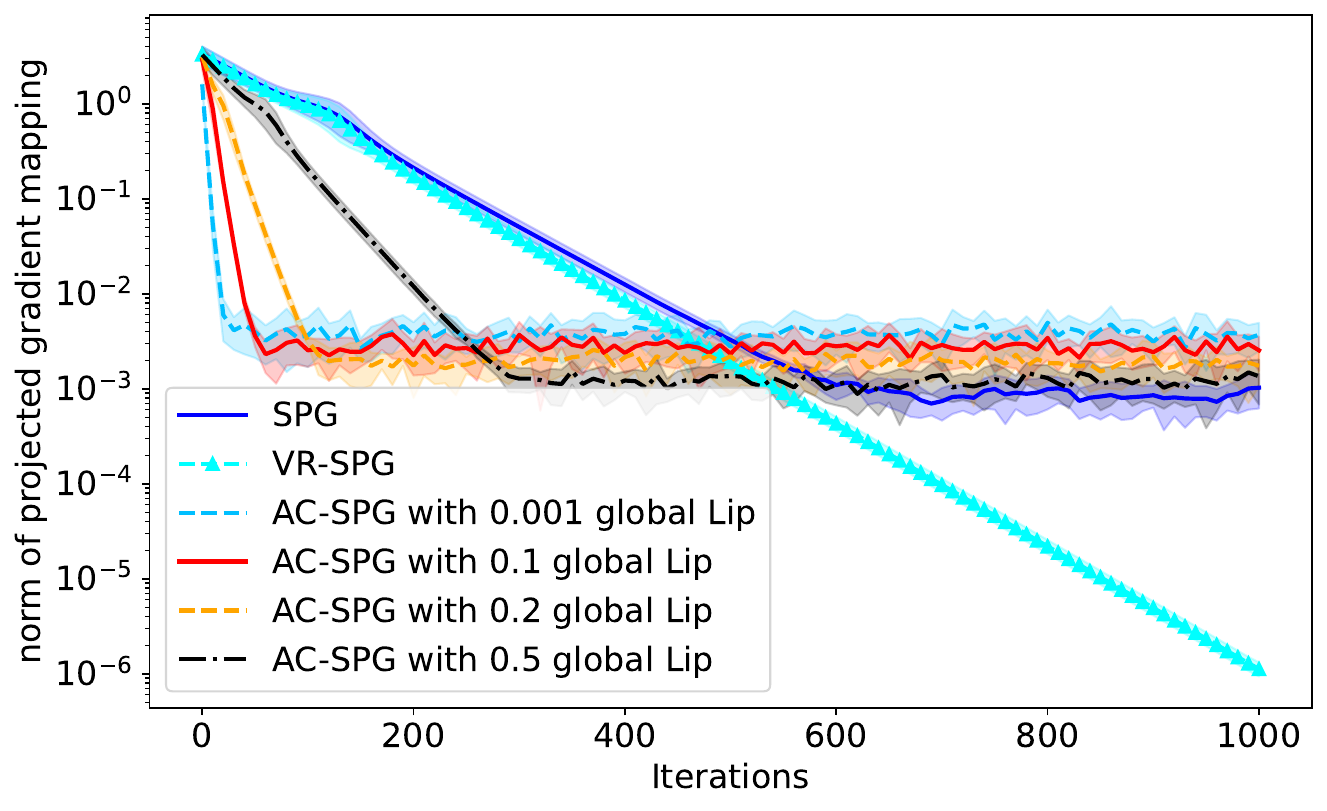}~~
    \includegraphics[width=0.435\linewidth]{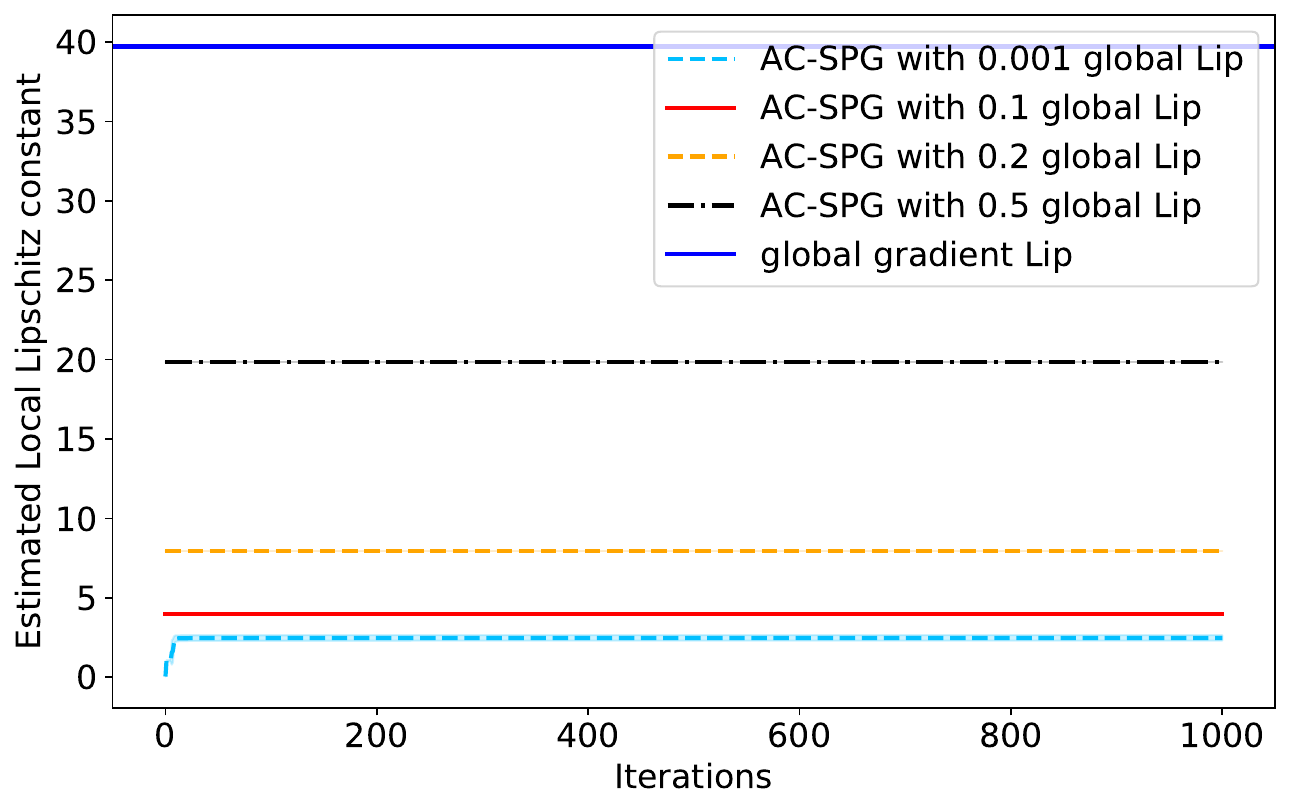} \\
    \includegraphics[width=0.45\linewidth]{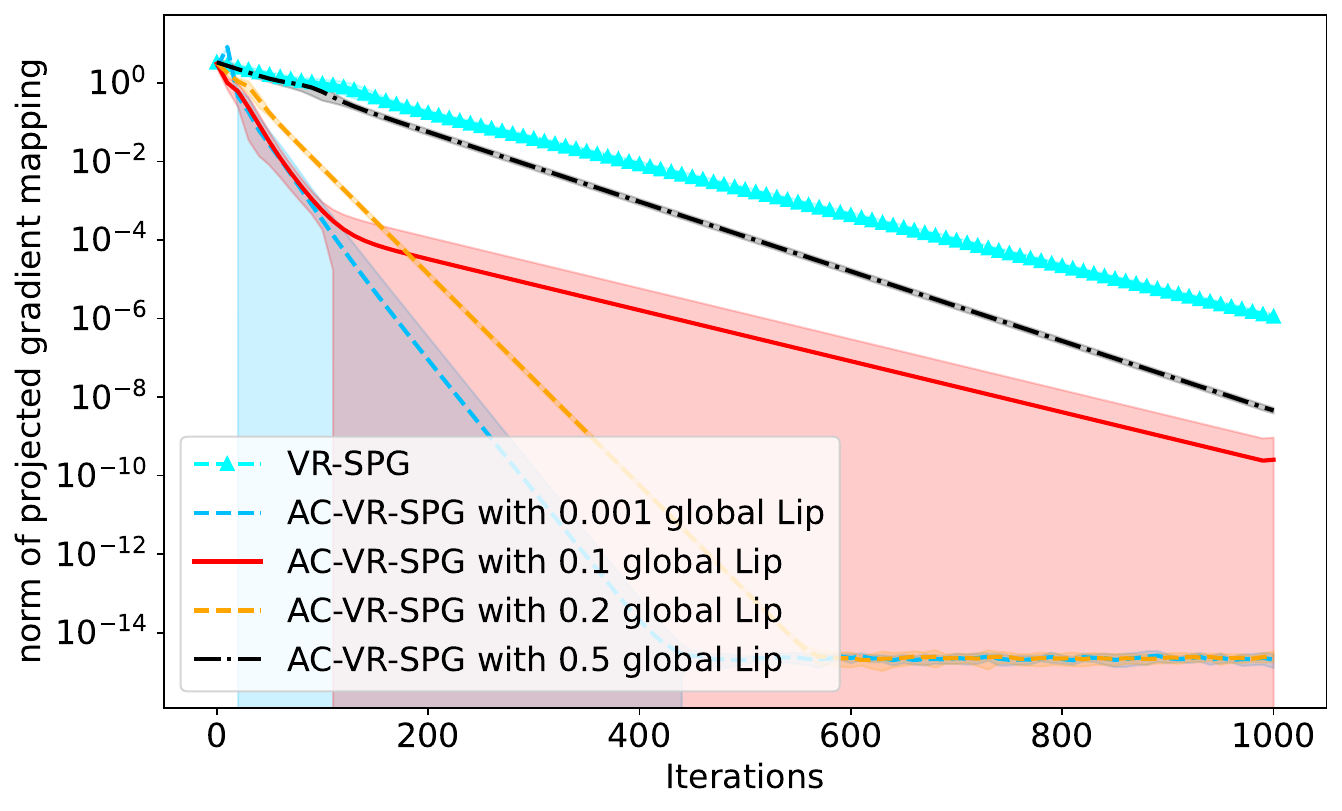}~~
    \includegraphics[width=0.435\linewidth]{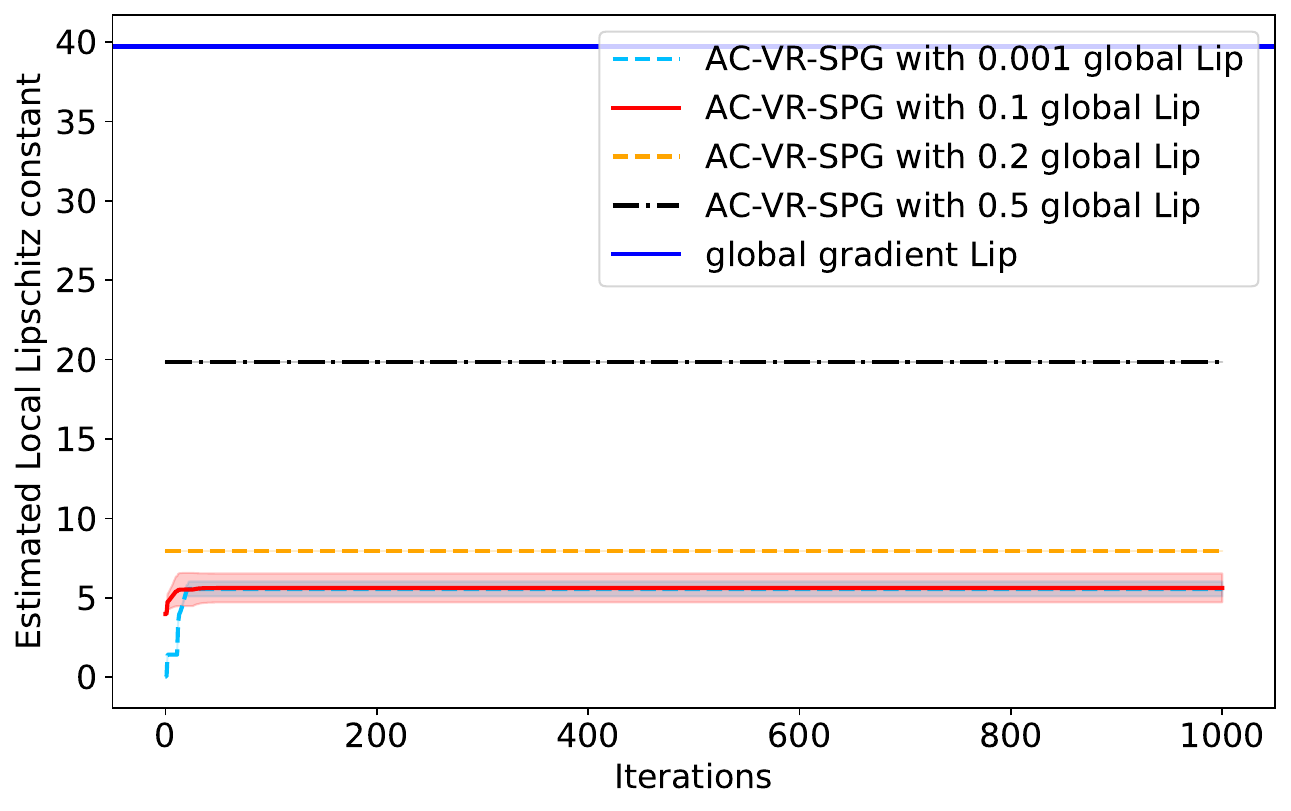}
    \caption{Average and standard deviation results by SPG, VR-SPG, AC-SPG, and AC-VR-SPG on solving 10 independently generated instances of \eqref{eq:svm} of dimension $n=10$ with pre-generated data sets for $u_1$ and $u_2$.}
    \label{fig:SVM-det}
\end{figure}

\begin{figure}[ht]
    \centering
    \includegraphics[width=0.45\linewidth]{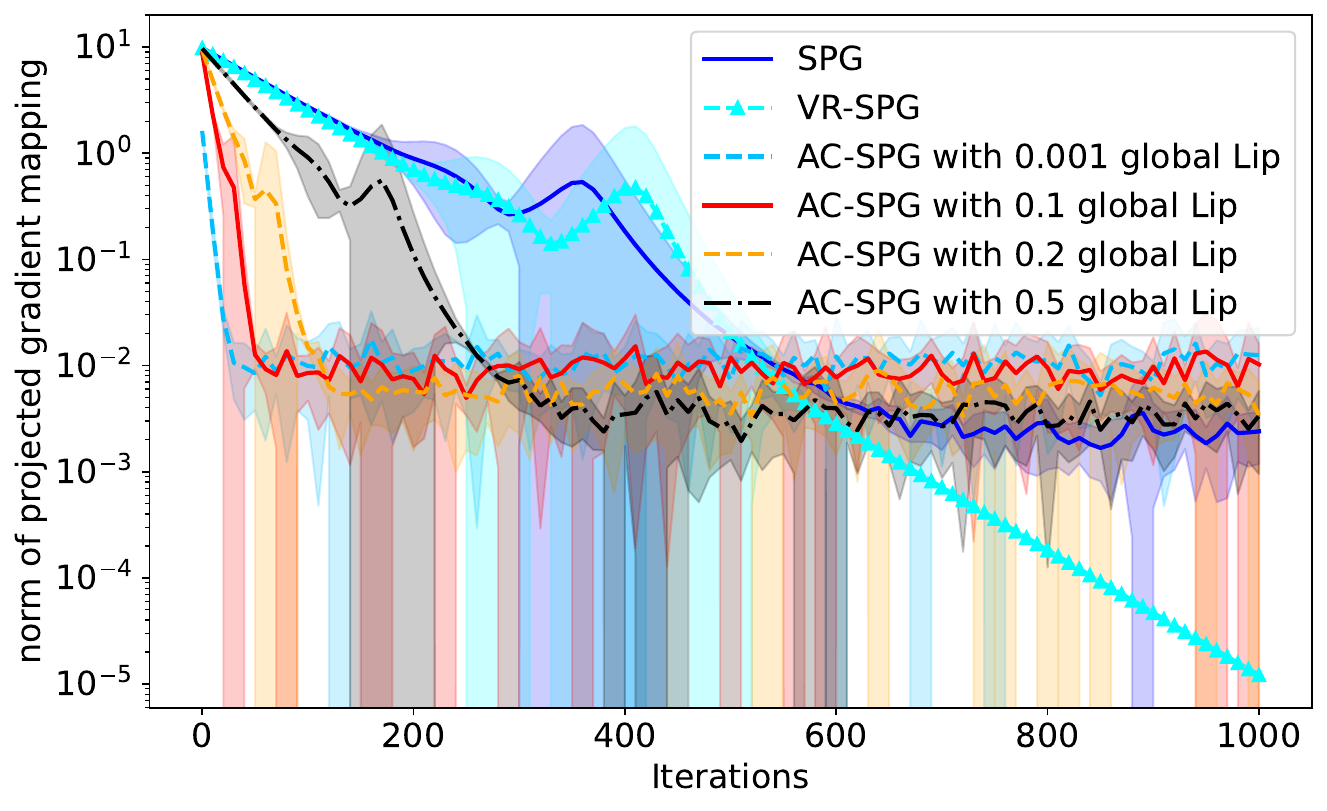}~~
    \includegraphics[width=0.435\linewidth]{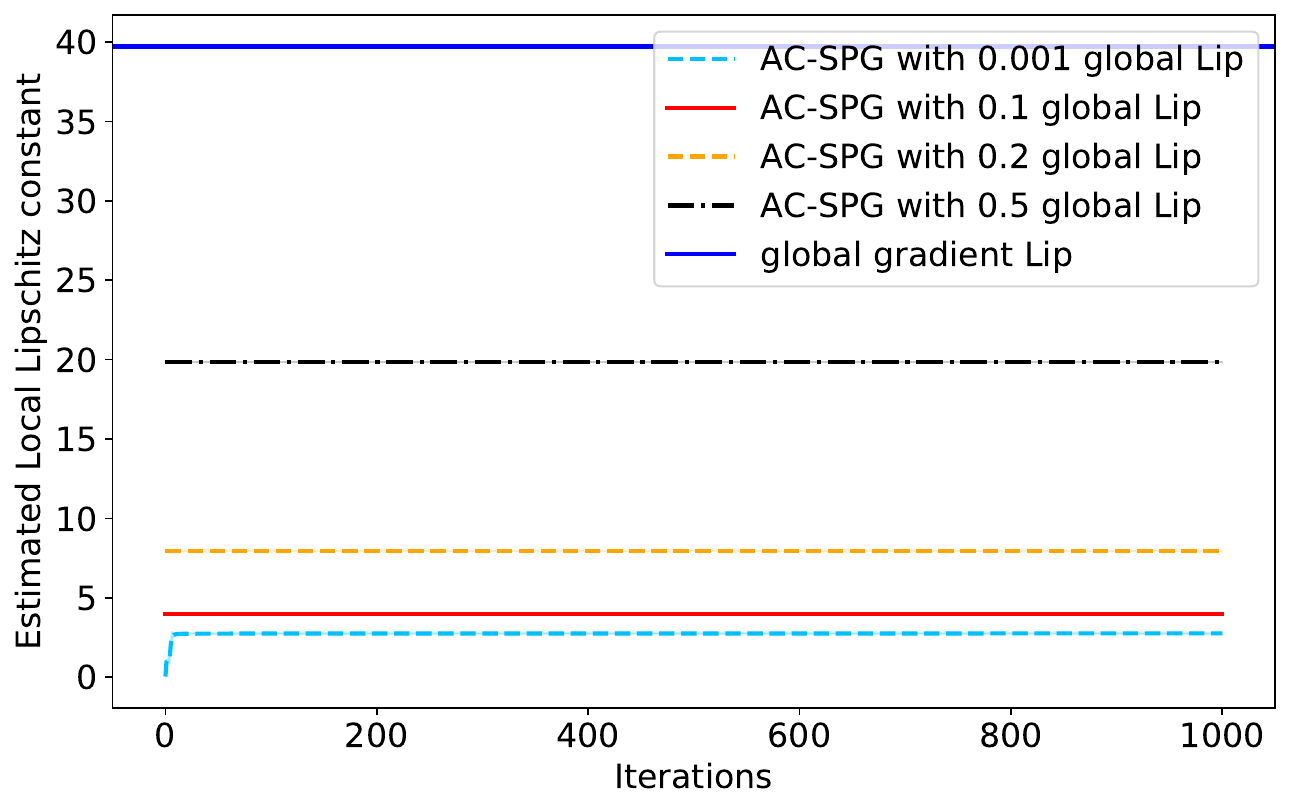} \\
    \includegraphics[width=0.45\linewidth]{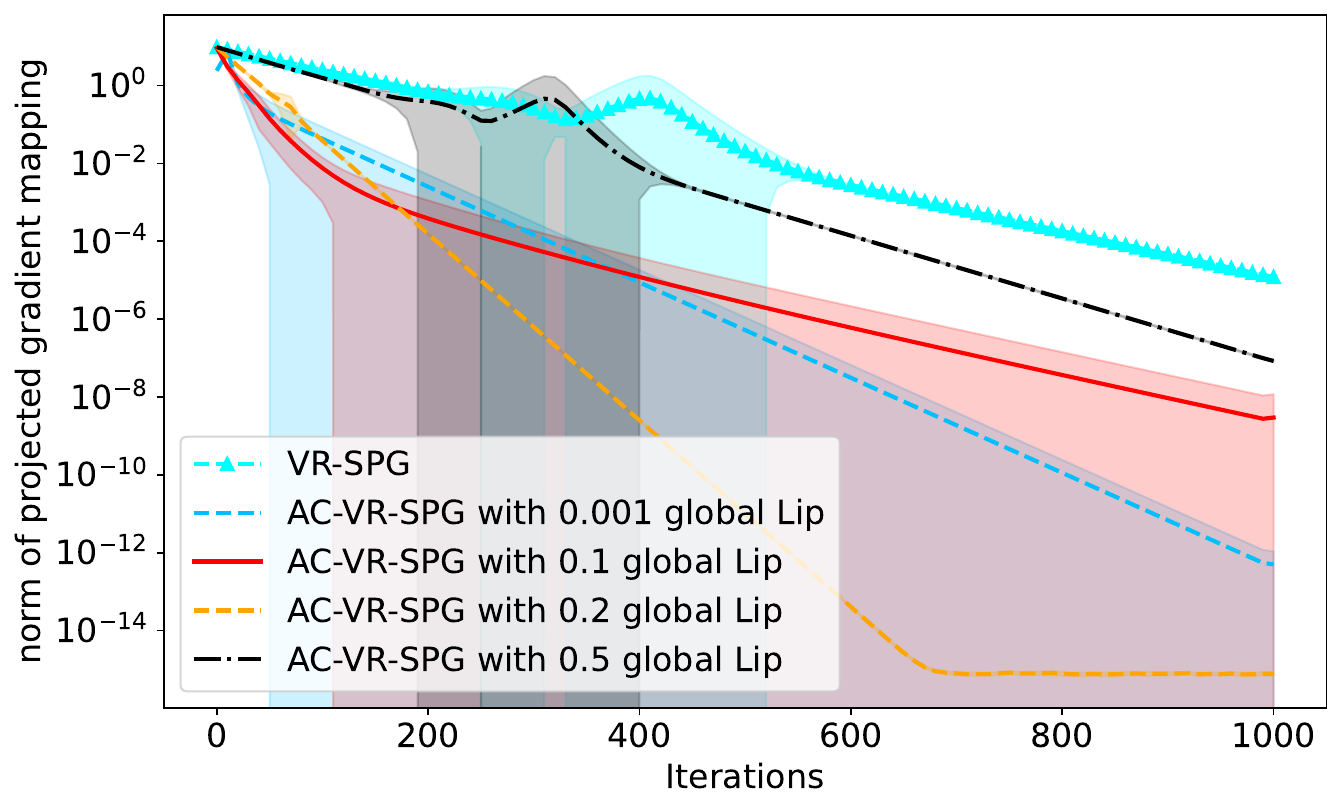}~~
    \includegraphics[width=0.435\linewidth]{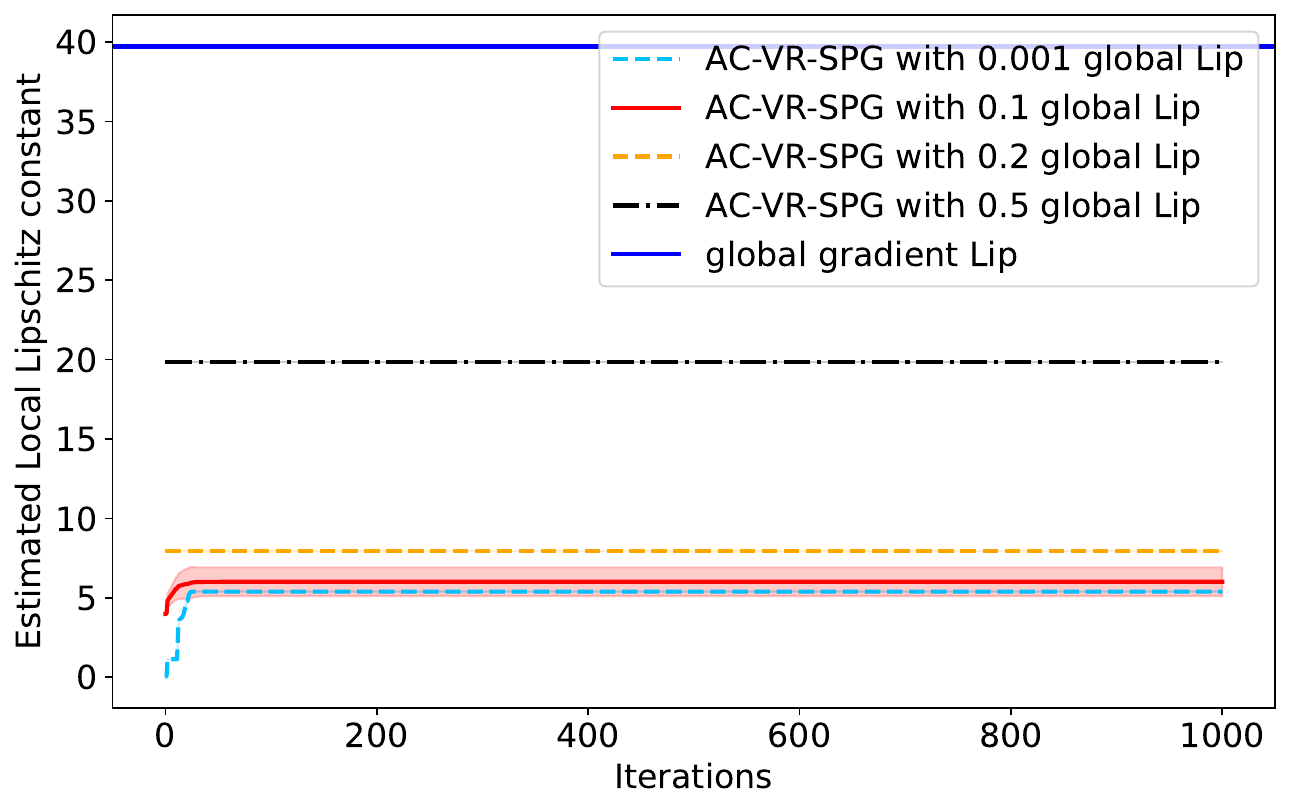}
    \caption{Average and standard deviation results by SPG, VR-SPG, AC-SPG, and AC-VR-SPG on solving 10 independently generated instances of \eqref{eq:svm} of dimension $n=100$ with pre-generated data sets for $u_1$ and $u_2$.}
    \label{fig:SVM-det-n100}
\end{figure}

For the second way, we solve the problem in an online manner, i.e., generating data samples while they are needed. Each sample of $u_1$ and $u_2$ is generated by $ u / \|u\|$ with $u$ following the standard multivariable Gaussian distribution. The global gradient Lipschitz constant is the same as above. At each iteration of SPG and AC-SPG, we randomly generate $25,000$ samples for each of $u_1$ and $u_2$; for VR-SPG and AC-VR-SPG, we set $T=10$, $N=2\times 10^5$, and $b_t = 5,000$ for each $t$ such that $\mathrm{mod}(t, T) \neq 1$.  
All other parameters are set the same as those for the first way of data generation. 
We approximately compute the projected gradient mapping by generating $10^6$ samples at each evaluated iterate. The results of running 10 independent trials are shown in Figure~\ref{fig:SVM-stoc} and Figure~\ref{fig:SVM-stoc-n100}, which again demonstrate the effectiveness of auto-conditioned methods.

\begin{figure}[ht]
    \centering
    \includegraphics[width=0.45\linewidth]{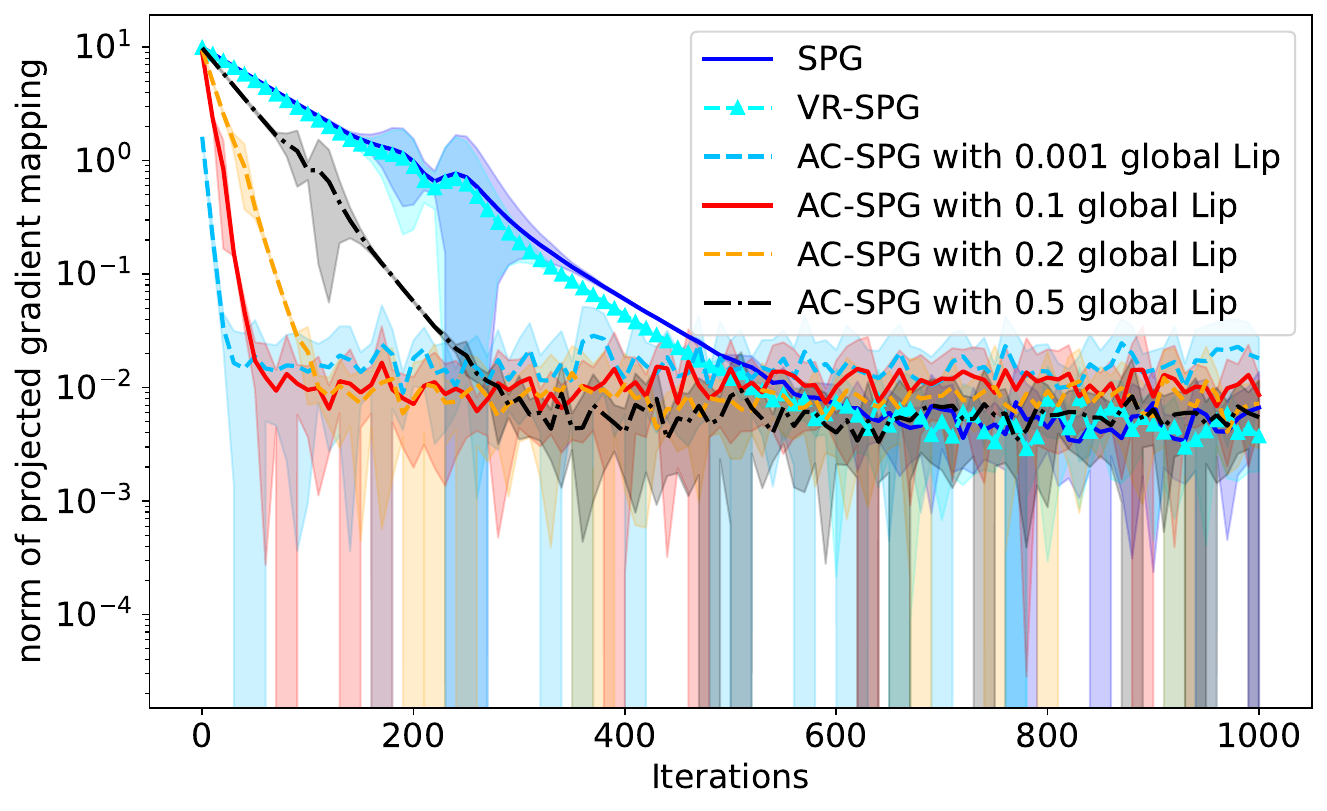}~~
    \includegraphics[width=0.435\linewidth]{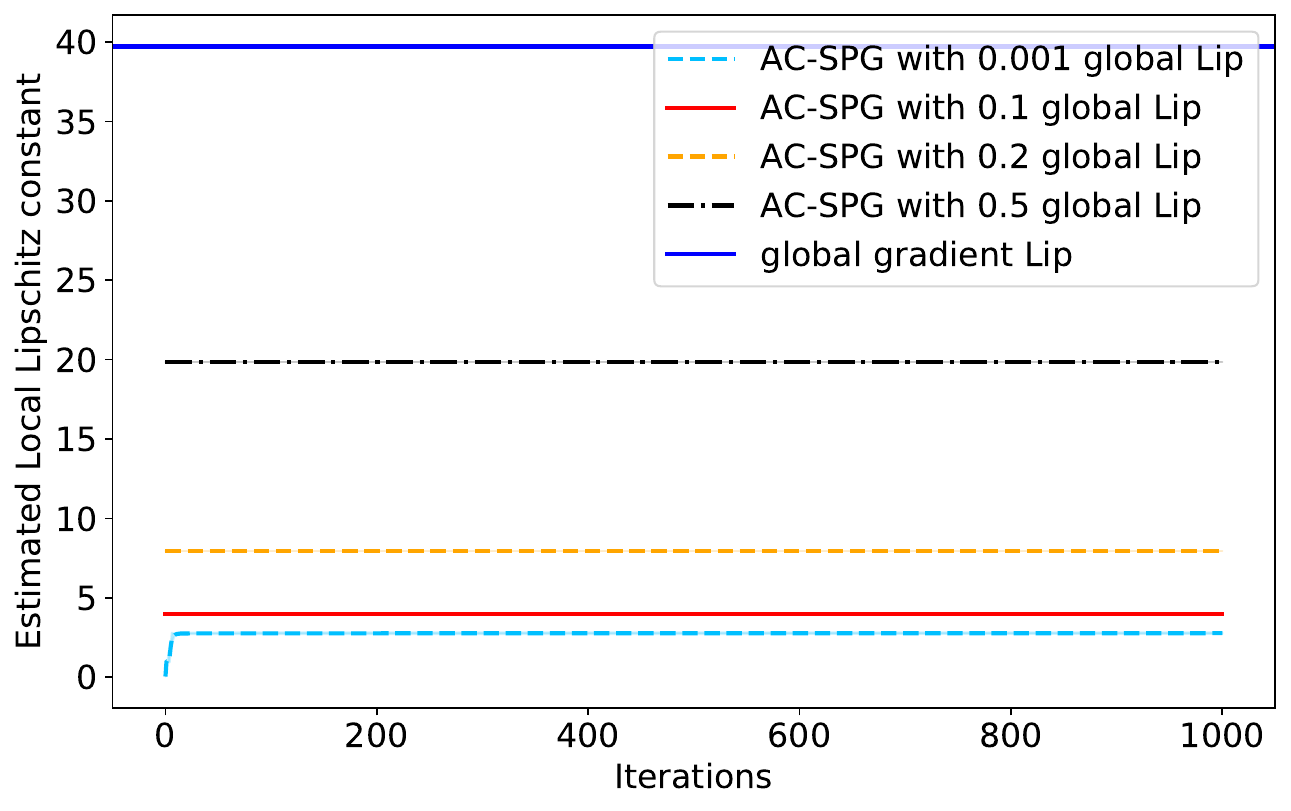} \\
    \includegraphics[width=0.45\linewidth]{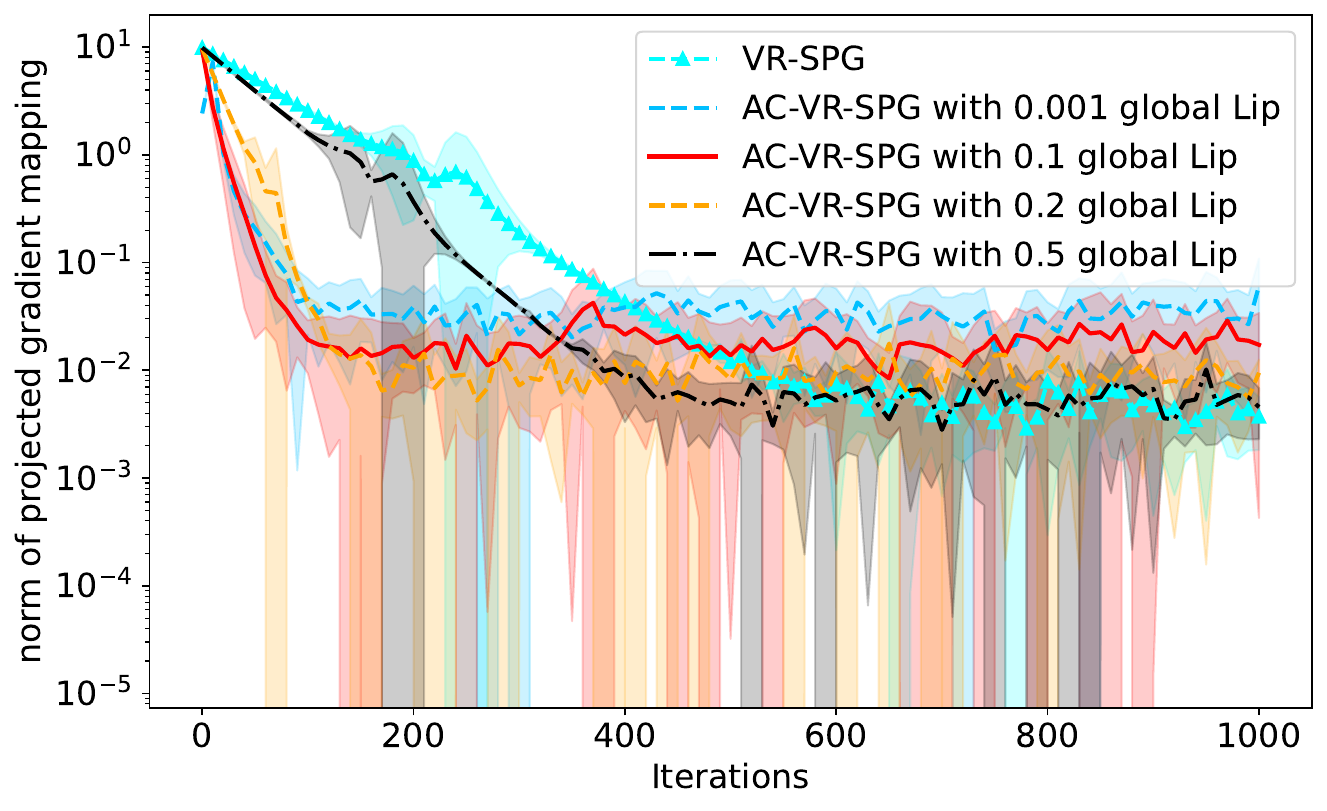}~~
    \includegraphics[width=0.435\linewidth]{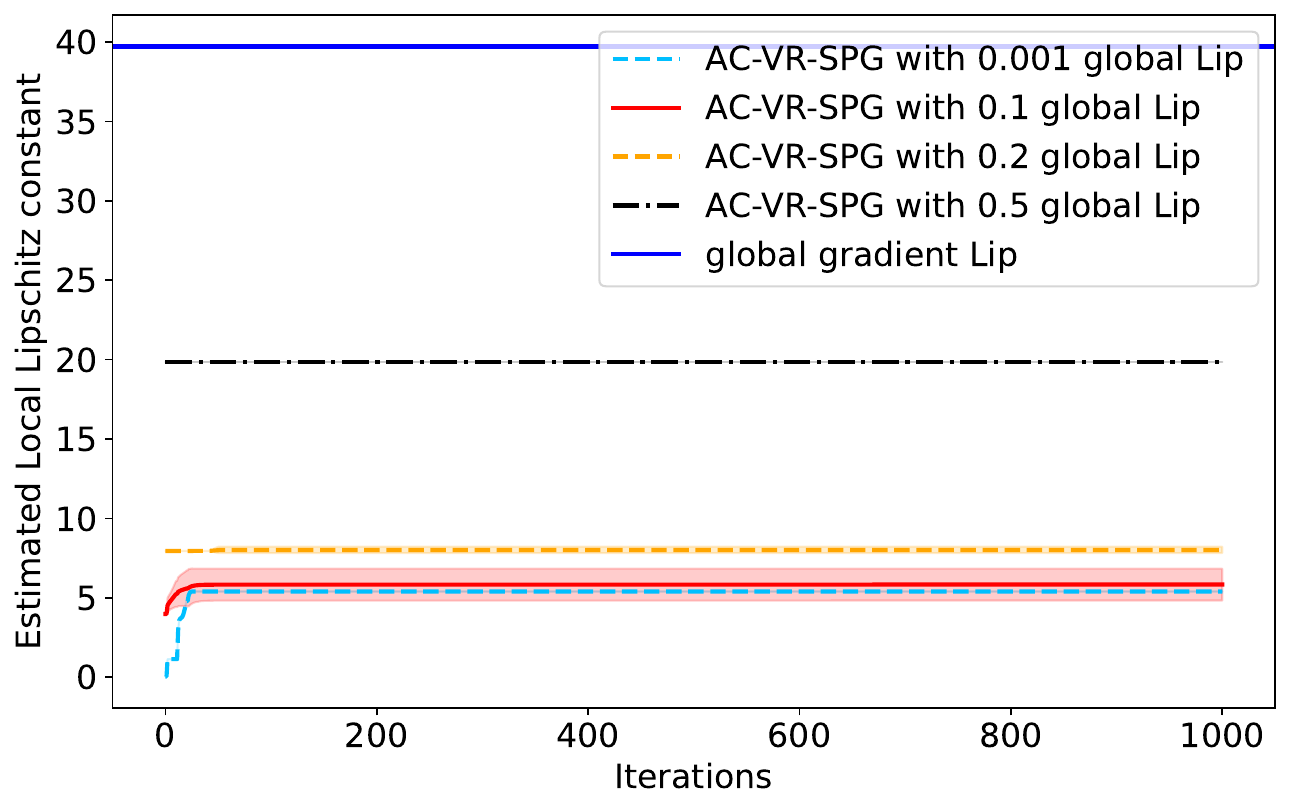}
    \caption{Average and standard deviation results by SPG, VR-SPG, AC-SPG, and AC-VR-SPG on solving 10 independently generated instances of \eqref{eq:svm} of dimension $n=10$ in an online manner, i.e., with samples generated for $u_1$ and $u_2$ while needed.}
    \label{fig:SVM-stoc}
\end{figure}

\begin{figure}[ht]
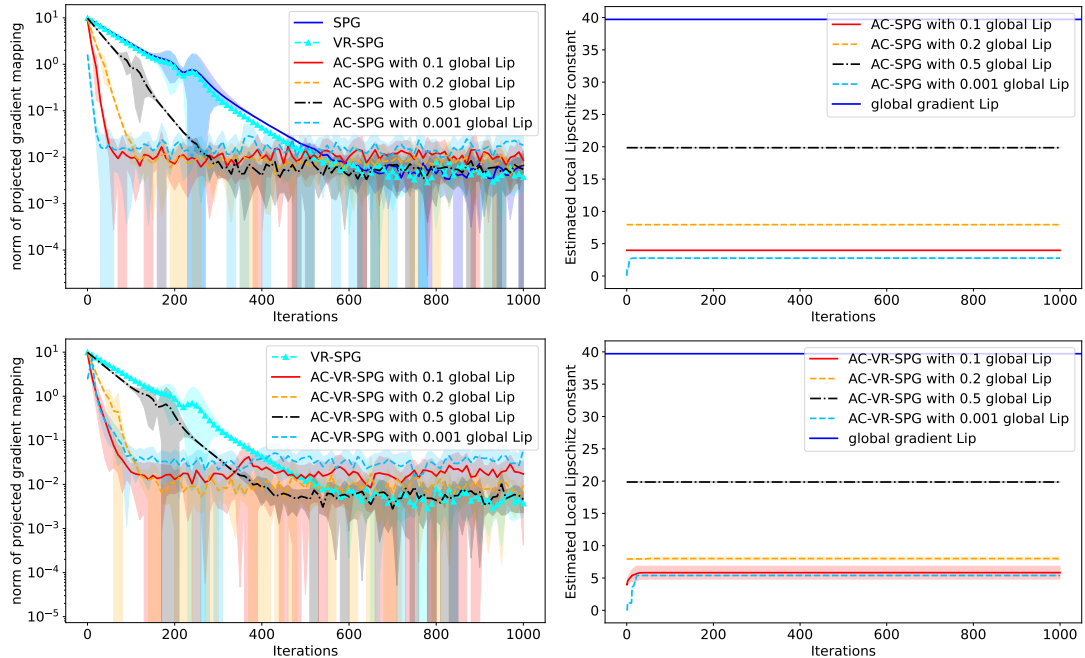

    \centering
    \includegraphics[width=0.45\linewidth]{pics/dim=100/SVM_nrm_grad_new.pdf}~~
    \includegraphics[width=0.435\linewidth]{pics/dim=100/SVM_Lip_new.pdf} \\
    \includegraphics[width=0.45\linewidth]{pics/dim=100/SVM_nrm_grad_VR_new.pdf}~~
    \includegraphics[width=0.435\linewidth]{pics/dim=100/SVM_Lip_VR_new.pdf}
    \caption{Average and standard deviation results by SPG, VR-SPG, AC-SPG, and AC-VR-SPG on solving 10 independently generated instances of \eqref{eq:svm} of dimension $n=100$ in an online manner, i.e., with samples generated for $u_1$ and $u_2$ while needed.}
    \label{fig:SVM-stoc-n100}
\end{figure}

\section{Concluding remarks}
In summary, we present a novel class of projected gradient (PG) methods for nonconvex optimization over a convex compact set. 
We first provide a novel analysis of the \revision{``vanilla''}{constant-stepsize} PG method, which achieves the best-known iteration complexity, unifying convex and nonconvex optimization. Then we propose the parameter-free and line search-free AC-PG method, which attains the same iteration complexity up to an additive constant. Next, we generalize the PG methods to the stochastic setting by introducing a stochastic projected gradient (SPG) method with new unified complexity bounds. We also provide auto-conditioned stepsize policies for SPG to deal with unknown Lipschitz parameters. Finally, we propose the variance-reduced stochastic gradient (VR-SPG) method under the assumption of Lipschitz continuity of the stochastic gradient. VR-SPG achieves the best-known sample complexity,  which is nearly order-optimal in both convex and nonconvex settings. We also develop the auto-conditioned variant of VR-SPG with comparable complexity bounds. Our theoretical guarantees are corroborated by numerical experiments.

\appendix
\section{Appendix}
{\color{black}
\subsection{Extension to the composite optimization setting}\label{appendix_prox_extension}
In this appendix, we extend our unified analysis for PG, i.e., Theorem~\ref{thm_0}, to the composite optimization setting. Specifically, we consider
\begin{align}\label{def_composite_problem_0}
\Psi^* := \min_{x \in X} \{\Psi(x):= f(x) +h(x) \},
\end{align}
where $f$ is a smooth function as defined in the main parts of the paper and $h: X \rightarrow \bbr$ is a ``prox-friendly'' closed convex function such that the ``proximal-mapping'' subproblem
\begin{align*} 
x^+ = \arg \min_{u \in X}\left\{\langle g, u\rangle + h(u) + \tfrac{\gamma}{2}\|x-u\|^2\right\}
\end{align*}
admits a closed-form solution or is easily computable by some efficient computational procedure. We assume $\bar X:=\{x \in X| h(x) < +\infty\}$ is a closed convex set. \revision{}{Additionally, we assume $x_0 \in \bar X$ and $D_{\bar X, x_0}:= \max_{x, y\in \mathcal{E}_{\bar X}(x_0)}\|x-y\| < \infty$, where  $\mathcal{E}_{\bar X}(x_0):= \{x \in \bar X| \Psi(x) \leq \Psi(x_0)\}$ with a slight abuse of notation.}
To iteratively solve problem~\eqref{def_composite_problem_0}, we modify step \eqref{prox-mapping_0} in Algorithm~\ref{alg_1_0} to
\begin{align}\label{proximal_gradient_update}
  x_t = \arg \min_{x\in X} \left\{\langle \nabla f(x_{t-1}) , x \rangle + h(x) + \tfrac{\gamma_t}{2}\|x_{t-1} - x\|^2 \right\}. 
  \end{align}
  We define $P_{\bar X}(x, g, \gamma) := \gamma(x-x^+)$ and $g_{\bar X, t} \equiv P_{\bar X}(x_t, \nabla f(x_t), \gamma_{t+1}) := \gamma_{t+1}(x_{t} - x_{t+1})$.  The convergence guarantees for this proximal version of the PG method are characterized in the following theorem.

\begin{theorem}\label{thm_0_proximal}
Let $\{x_t\}$ be generated by the proximal gradient method with updates \eqref{proximal_gradient_update}. If the stepsizes satisfy $\gamma_t = \gamma \geq L, \forall\, t$, we have
\begin{align}\label{res_1_thm_0_proximal}
\min_{0\leq t \leq k-1}\|g_{\bar X,t}\|^2 \leq \left(\tsum_{t=1}^{k} \big(\tfrac{2t-1}{2\gamma} - \tfrac{tL}{2\gamma^2}\big)\right)^{-1}\left(\tfrac{\gamma}{2}\|x_0 - x^*\|^2+ \tfrac{lk}{2}D^{2}_{\bar X, x_0}\right).
\end{align}
Specifically, if $\gamma_t = L, \forall\, t$, then
\begin{align}\label{res_2_thm_0_proximal}
\min_{0\leq t \leq k-1}\|g_{\bar X,t}\|^2 \leq \tfrac{2L^2\|x_0 - x^*\|^2}{k(k-1)} + \tfrac{2 L l D^{2}_{\bar X, x_0}}{k-1}.
\end{align}
\end{theorem}
\begin{proof}
By the optimality condition of \eqref{proximal_gradient_update}, we have the following three-point inequality (e.g., Lemma 3.5 in \cite{LanBook2020}): for any $x \in X$,
\begin{align*}
\langle \nabla f(x_{t-1}), x_t - x \rangle + h(x_t)  + \tfrac{\gamma_t}{2}\|x_{t-1} - x_t\|^2 + \tfrac{\gamma_t}{2}\|x_t - x\|^2 \leq h(x) + \tfrac{\gamma_t}{2}\|x_{t-1} - x\|^2.
\end{align*}
Adding $f(x_{t-1})$ on both sides of the above inequality and rearranging the terms, we obtain
\begin{align}\label{eq_1_0_0_prox}
&f(x_{t-1}) + \langle \nabla f(x_{t-1}), x_t - x_{t-1}\rangle + h(x_t) + \tfrac{\gamma_t}{2}\|x_{t-1} - x_t\|^2 + \tfrac{\gamma_t}{2}\|x_t - x\|^2\nn\\
&\leq f(x_{t-1}) +\langle \nabla f(x_{t-1}), x - x_{t-1}\rangle + h(x) + \tfrac{\gamma_t}{2}\|x_{t-1} - x\|^2.
\end{align}
Then, by applying the upper curvature condition \eqref{smoothness_2} on the LHS of the above inequality and invoking the definition $\Psi(x) = f(x) + h(x)$, we obtain
\begin{align}\label{eq_1_0_prox}
&\Psi(x_t) + (\tfrac{\gamma_t}{2} - \tfrac{L}{2})\|x_{t-1} - x_t\|^2 + \tfrac{\gamma_t}{2}\|x_t - x\|^2\nn\\
&\leq f(x_{t-1}) +\langle \nabla f(x_{t-1}), x - x_{t-1}\rangle - f(x) + \Psi(x) + \tfrac{\gamma_t}{2}\|x_{t-1} - x\|^2.
\end{align}
Setting $x = x_{t-1}$ in the above inequality generates
\begin{align}\label{eq_2_0_prox}
(\gamma_t -\tfrac{L}{2})\|x_t - x_{t-1}\|^2 \leq \Psi(x_{t-1}) - \Psi(x_t).
\end{align}
Clearly, this inequality indicates that the function value $\Psi$ is non-increasing over iterations as we choose $\gamma_t \geq L/2$. Thus we have  $\forall t \in \mathbb{Z}_+, x_t \in \mathcal{E}_{\bar X}(x_0)$, and consequently, $\|x_t - x^*\|^2 \leq D_{\bar X, x_0}^2$. Meanwhile, setting $x = x^*$ in Ineq.~\eqref{eq_1_0_prox} and using the weakly convex condition $f(x_{t-1}) +\langle \nabla f(x_{t-1}), x^* - x_{t-1}\rangle - f(x^*) \leq  l\|x_{t-1} - x^*\|^2/2$ yields
\begin{align}\label{eq_3_0_prox}
&\Psi(x_t) + (\tfrac{\gamma_t}{2} - \tfrac{L}{2})\|x_{t-1} - x_t\|^2 + \tfrac{\gamma_t}{2}\|x_t - x^*\|^2\leq \Psi(x^*) + \tfrac{l}{2}\|x_{t-1} - x^*\|^2 + \tfrac{\gamma_t}{2}\|x_{t-1} - x^*\|^2.
\end{align}
By taking a telescope sum of the above inequality from $t=1$ to $k$, and applying $\gamma_t = \gamma$, we obtain
\begin{align}\label{eq_4_0_prox}
&\tsum_{t=1}^k \left[\Psi(x_t) -\Psi(x^*) + (\tfrac{\gamma}{2} - \tfrac{L}{2})\|x_{t-1} - x_t\|^2 \right]\nn\\
& \leq \tfrac{\gamma}{2}\|x_0 - x^*\|^2  - \tfrac{\gamma}{2}\|x_k -x^*\|^2 + \tsum_{t=1}^k \tfrac{l}{2}\|x_{t-1}-x^*\|^2\nn\\
& \leq \tfrac{\gamma}{2}\|x_0 - x^*\|^2 + \tfrac{lk}{2}D^{2}_{\bar X, x_0}.
\end{align}
On the other hand, we can lower bound the LHS of the above inequality by 
\begin{align}\label{eq_5_0_0_prox}
&\tsum_{t=1}^k \left[\Psi(x_t) -\Psi(x^*) + (\tfrac{\gamma}{2} - \tfrac{L}{2})\|x_{t-1} - x_t\|^2 \right]\nn\\
& \geq \tsum_{t=1}^k \left[\Psi(x_t) -\Psi(x_k) + (\tfrac{\gamma}{2} - \tfrac{L}{2})\|x_{t-1} - x_t\|^2 \right]\nn\\
& = \tsum_{t=1}^k \left[\tsum_{j=t}^{k-1}\big(\Psi(x_j) -\Psi(x_{j+1})\big) + (\tfrac{\gamma}{2} - \tfrac{L}{2})\|x_{t-1} - x_t\|^2 \right]\nn\\
& \overset{(i)}\geq \tsum_{t=1}^{k} \left[\tsum_{j=t}^{k-1}(\gamma - \tfrac{L}{2})\|x_j - x_{j+1}\|^2 + (\tfrac{\gamma}{2} - \tfrac{L}{2})\|x_{t-1} - x_t\|^2 \right]\nn\\
&  = \tsum_{t=1}^{k}\big((t-\tfrac{1}{2})\gamma - \tfrac{t}{2}L\big)\|x_{t-1} - x_t\|^2\nn\\
& = \tsum_{t=1}^{k}\big(\tfrac{2t-1}{2\gamma} - \tfrac{tL}{2\gamma^2}\big)\|g_{\bar X, t-1}\|^2,
\end{align}
where step $(i)$ follows from \eqref{eq_2_0_prox}.
By combining Ineqs.~\eqref{eq_4_0_prox} and \eqref{eq_5_0_0_prox}, we arrive at
\begin{align*}
\min_{0\leq t \leq k-1}\|g_{\bar X,t}\|^2 &\leq \left(\tsum_{t=1}^{k} \big(\tfrac{2t-1}{2\gamma} - \tfrac{tL}{2\gamma^2}\big)\right)^{-1}\tsum_{t=1}^{k}\big(\tfrac{2t-1}{2\gamma} - \tfrac{tL}{2\gamma^2}\big)\|g_{\bar X, t-1}\|^2\\
&\leq \left(\tsum_{t=1}^{k} \big(\tfrac{2t-1}{2\gamma} - \tfrac{tL}{2\gamma^2}\big)\right)^{-1}\left(\tfrac{\gamma}{2}\|x_0 - x^*\|^2+ \tfrac{lk}{2}D^{2}_{\bar X, x_0}\right),
\end{align*}
which completes the proof.
\end{proof}

\vgap
It should be noted that, if $\bar X:=\{x \in X| h(x) < +\infty\}$ is a convex compact set, all other results proposed in this paper can be generalized to the composite optimization setting in a similar manner. 
}

{\color{black}
\subsection{A unified analysis for weakly convex, convex, and strongly convex problems}\label{subsection_strongly_convex_extension}
In this subsection, we consider the more general setting where $f$ can be potentially weakly convex, convex, or strongly convex, i.e., for $l \in [-L, L]$,
\begin{align}\label{general_convexity}
f(x)- f(y) - \langle \nabla f(y), x - y\rangle \geq - \tfrac{l}{2}\|x-y\|^2, \quad \forall x, y \in X.
\end{align}
In other words, we generalize Ineq.~\eqref{weakly_convex} for negative $l$ to allow strong convexity. For this more general setting, we can also provide unified convergence guarantees for the projected gradient method (Algorithm \ref{alg_1_0}) in terms of the norm of the projected gradient mapping.
\begin{theorem}\label{theorem_for_extension}
Let $\{x_t\}$ be generated by Algorithm \ref{alg_1_0} and $\{\theta_t\}$ be a sequence of nonnegative numbers. Assume $W_t := \left( \tsum_{j=1}^{t}\theta_j\right)(\tfrac{1}{\gamma_t} - \tfrac{L}{2\gamma_t^2}) -\tfrac{\theta_t}{2\gamma_t} > 0$. We have
\begin{align}\label{extension_theorem}
&\min_{0 \leq t \leq k-1}\|g_{X, t}\|^2\leq \left( \tsum_{t=1}^k W_t\right)^{-1} \left(\tfrac{\theta_1(\gamma_1 +l)}{2}\|x_0-x^*\|^2 + \tsum_{t=1}^{k-1} \tfrac{\theta_{t+1}(\gamma_{t+1} + l)-\theta_t\gamma_t}{2}\|x_t - x^*\|^2  \right).
\end{align}
Specifically, if $l < 0$, we set $\gamma_t = L$ and $\theta_t = \theta^t := \left(\tfrac{L}{L+ l} \right)^t$, then we have\footnote{\revision{}{When $l = -L$, we have $\theta_t = +\infty$ and $\min_{0 \leq t \leq k-1}\|g_{X, t}\|^2 = 0$ for any $k \geq 2$, which means that the problem can be solved in one iteration.}} for any $k\geq 2$, 
\begin{align}\label{extension_strongly_convex}
\min_{0 \leq t \leq k-1}\|g_{X, t}\|^2\leq \left( \tsum_{t=2}^k \tsum_{j=1}^{t-1}\left(\tfrac{L}{L+l} \right)^{j}\right)^{-1} \cdot L^2\|x_0-x^*\|^2.
\end{align}
If $l \geq 0$, we set $\gamma_t = L$ and $\theta_t = 1$, $\forall t$, then we have for any $k\geq 2$, 
\begin{align}\label{extension_general_weakly_convex}
\min_{0\leq t \leq k-1}\|g_{X,t}\|^2 \leq \tfrac{2L^2\revision{D_X^2}{\|x_0 - x^*\|^2}}{k(k-1)} + \tfrac{2 L l \revision{D_X^2}{D_{X, x_0}^2}}{k-1}.
\end{align}
\end{theorem}
\begin{proof}
Following similar arguments as in the proof of Theorem~\ref{thm_0}, we can show that Ineqs.~\eqref{eq_2_0} and \eqref{eq_3_0} still hold, i.e.,
\begin{align}\label{eq_2_ext_0}
(\gamma_t -\tfrac{L}{2})\|x_t - x_{t-1}\|^2 \leq f(x_{t-1}) - f(x_t).
\end{align}
and
\begin{align}\label{eq_3_ext_0}
&f(x_t) + (\tfrac{\gamma_t}{2} - \tfrac{L}{2})\|x_{t-1} - x_t\|^2 + \tfrac{\gamma_t}{2}\|x_t - x^*\|^2\leq f(x^*) + \tfrac{l}{2}\|x_{t-1} - x^*\|^2 + \tfrac{\gamma_t}{2}\|x_{t-1} - x^*\|^2.
\end{align}
Here, we allow $l \in [-L, L]$ due to the more general assumption~\eqref{general_convexity}. Next, we multiply $\theta_t$ on both sides of Ineq.~\eqref{eq_3_ext_0} and take a telescope sum from $t=1$ to $k$. By rearranging the terms, we obtain
\begin{align*}
&\tsum_{t=1}^k \theta_t\left[ f(x_t) - f(x^*) + (\tfrac{\gamma_t}{2} - \tfrac{L}{2})\|x_{t-1}-x_t\|^2\right]\\
& \leq \tfrac{\theta_1(\gamma_1 +l)}{2}\|x_0-x^*\|^2 + \tsum_{t=1}^{k-1} \tfrac{\theta_{t+1}(\gamma_{t+1}+ l)-\theta_t\gamma_t}{2}\|x_t - x^*\|^2 - \tfrac{\theta_k\gamma_k}{2}\|x_k - x^*\|^2.
\end{align*}
Next, we lower bound the LHS of the above inequality, i.e.,
\begin{align*}
&\tsum_{t=1}^k \theta_t\left[ f(x_t) - f(x^*) + (\tfrac{\gamma_t}{2} - \tfrac{L}{2})\|x_{t-1}-x_t\|^2\right]\\
&\overset{(i)}\geq \tsum_{t=1}^k \theta_t\left[ f(x_t) - f(x_k) + (\tfrac{\gamma_t}{2} - \tfrac{L}{2})\|x_{t-1}-x_t\|^2\right]\\
& = \tsum_{t=1}^k \theta_t\left[ \tsum_{j=t}^{k-1}(f(x_j) - f(x_{j+1})) + (\tfrac{\gamma_t}{2} - \tfrac{L}{2})\|x_{t-1}-x_t\|^2\right]\\
& = \tsum_{t=2}^k \left[\left( \tsum_{j=1}^{t-1}\theta_j\right)\left(f(x_{t-1}) - f(x_t) \right) + \theta_t (\tfrac{\gamma_t}{2} - \tfrac{L}{2})\|x_{t-1}-x_t\|^2\right] + \theta_1 (\tfrac{\gamma_1}{2} - \tfrac{L}{2})\|x_{0}-x_1\|^2\\
& \overset{(ii)}\geq \tsum_{t=2}^k \left(\left( \tsum_{j=1}^{t-1}\theta_j\right)(\gamma_t - \tfrac{L}{2}) + \theta_t (\tfrac{\gamma_t}{2} - \tfrac{L}{2})\right)\|x_{t-1}-x_t\|^2+ \theta_1 (\tfrac{\gamma_1}{2} - \tfrac{L}{2})\|x_{0}-x_1\|^2\\
&  \overset{(iii)}= \tsum_{t=1}^k \left(\left( \tsum_{j=1}^{t}\theta_j\right)(\gamma_t - \tfrac{L}{2}) -  \tfrac{\theta_t\gamma_t}{2} \right)\|x_{t-1}-x_t\|^2\\
& \overset{(iv)}= \tsum_{t=1}^k \left(\left( \tsum_{j=1}^{t}\theta_j\right)(\tfrac{1}{\gamma_t} - \tfrac{L}{2\gamma_t^2}) -\tfrac{\theta_t}{2\gamma_t}\right)\|g_{X, t-1}\|^2,
\end{align*}
where step (i) follows from $f(x_k) \geq f(x^*)$, step (ii) follows from Ineq.~\eqref{eq_2_ext_0}, step (iii) follows from rearranging the terms, and step (iv) follows from the definition of $g_{X,t}$ in \eqref{def_projected_gradient}. By combining the two sets of inequalities above, we obtain
\begin{align*}
&\tsum_{t=1}^k \left(\left( \tsum_{j=1}^{t}\theta_j\right)(\tfrac{1}{\gamma_t} - \tfrac{L}{2\gamma_t^2}) -\tfrac{\theta_t}{2\gamma_t}\right)\|g_{X, t-1}\|^2\\
&\leq \tfrac{\theta_1(\gamma_1 +l)}{2}\|x_0-x^*\|^2 + \tsum_{t=1}^{k-1} \tfrac{\theta_{t+1}(\gamma_{t+1} + l)-\theta_t\gamma_t}{2}\|x_t - x^*\|^2 - \tfrac{\theta_k\gamma_k}{2}\|x_k - x^*\|^2, 
\end{align*}
which indicates that
\begin{align*}
&\min_{0 \leq t \leq k-1}\|g_{X, t}\|^2\\
&\leq \left( \tsum_{t=1}^k \left(\left( \tsum_{j=1}^{t}\theta_j\right)(\tfrac{1}{\gamma_t} - \tfrac{L}{2\gamma_t^2}) -\tfrac{\theta_t}{2\gamma_t}\right)\right)^{-1} \left(\tfrac{\theta_1(\gamma_1 +l)}{2}\|x_0-x^*\|^2 + \tsum_{t=1}^{k-1} \tfrac{\theta_{t+1}(\gamma_{t+1} + l)-\theta_t\gamma_t}{2}\|x_t - x^*\|^2  \right),
\end{align*}
and we complete the proof of \eqref{extension_theorem}. Ineq.~\eqref{extension_strongly_convex} follows from applying $\gamma_t = L$ and $\theta_t = \theta^t := \left(\tfrac{L}{L+ l} \right)^t$ in the inequality above and rearranging the term. Ineq.~\eqref{extension_general_weakly_convex} follows from applying $\gamma_t = L$ and $\theta_t = 1$ in the inequality above, and utilizing the non-ascent property to bound $\|x_t - x^*\|^2 \leq D_{X, x_0}^2$.
\end{proof}
\vgap

In view of Ineq.~\eqref{extension_strongly_convex} in Theorem~\ref{theorem_for_extension}, in the strongly convex setting where $l<0$, to find an $\epsilon$-stationary solution such that $\|g_{X, t}\| \leq \epsilon$, the number of PG iterations is bounded by
\begin{align*}
\mathcal{O}\left(\left(-\tfrac{L}{l} \right)\log\left( \tfrac{L\|x_0 - x^*\|}{\epsilon}\right)\right).
\end{align*}
}

\renewcommand\refname{Reference}

\bibliographystyle{abbrv}
\bibliography{arXiv_new_20261116}

\end{document}